\documentclass[10pt]{article}
\usepackage[margin=1in]{geometry}
\usepackage{amsmath}
\usepackage{amssymb}
\usepackage{amsthm}
\usepackage{verbatim}

\newcommand{\R}{{\mathbb R}}

\newcommand{\N}{\mathbb{N}}

\newcommand{\Sym}{\mathcal{S}}
\newcommand{\Symv}{{\vec{\mathcal{S}}}}

\newcommand{\BX}{\hat{B}}
\newcommand{\T}{\mathcal{T}}
\newcommand{\TX}{\hat{\T}}
\newcommand{\Q}{\mathcal{Q}}
\newcommand{\QX}{\hat{\Q}}
\newcommand{\I}{\mathcal{I}_n}
\newcommand{\IX}{\hat{\mathcal{I}}_n}

\newcommand{\pair}[1]{\langle #1 \rangle}

\newcommand{\alphav}{{\vec{\alpha}}}
\newcommand{\kappav}{{\vec{\kappa}}}

\newcommand{\deltav}{{\vec{\delta}}}
\newcommand{\lambdav}{{\vec{\lambda}}}
\newcommand\tauv{{\vec{\tau}}}
\newcommand{\Deltav}{{\vec{\Delta}}}
\newcommand\wl{\omega}
\newcommand\wlv{{\vec\wl}}
\newcommand{\deltabar}{\bar{\delta}}
\newcommand{\CP}{\mathcal{C}}
\newcommand{\CPv}{{\vec{\CP}}}

\newcommand{\OP}[1]{\theta_{#1}} % orbit partition
\newcommand{\SGP}[1]{\CP_{#1}} % subgraph partition
\newcommand{\SGPv}[1]{\CPv_{#1}} % subgraph partition
\newcommand{\Graphs}{\mathcal{G}}
\newcommand{\nuv}{\vec{\nu}}
\newcommand{\rhov}{\vec{\rho}}

\newcommand{\Y}[1]{Y_{\pair{#1}}}
\newcommand\piv{{\vec{\pi}}}
\newcommand{\Piv}{\vec{\Pi}}

\newtheorem{theorem}{Theorem}[section]
\newtheorem{lemma}[theorem]{Lemma}

\newtheorem{definition}[theorem]{Definition}
\newtheorem{corollary}[theorem]{Corollary}

\newtheorem{example}[theorem]{Example}

\DeclareMathOperator{\conv}{conv}

\begin{document}
\title{Sherali-Adams Relaxations of Graph Isomorphism Polytopes}

\author{Peter N. Malkin\footnote{BHP-Billiton Ltd, Resource and Business Optimisation. The paper was researched while the author was at the Department of Mathematics, University of California, Davis. Partially supported by NSF grant DMS-0914107 and an IBM OCR award.}}

\maketitle

\begin{abstract}
We investigate the Sherali-Adams lift \& project hierarchy applied to a graph isomorphism polytope whose integer points encode the isomorphisms between two graphs.  
In particular, the Sherali-Adams relaxations characterize a new vertex classification algorithm for graph isomorphism, which we call the \emph{generalized vertex classification algorithm}.  
This algorithm generalizes the classic vertex classification algorithm and generalizes the work of Tinhofer on polyhedral methods for graph automorphism testing.  
We establish that the Sherali-Adams lift \& project hierarchy when applied to a graph isomorphism polytope needs $\Omega(n)$ iterations in the worst case before converging to the convex hull of integer points.
We also show that this generalized vertex classification algorithm is also strongly related to the well-known Weisfeiler-Lehman algorithm, which we show can also be characterized in terms of the Sherali-Adams relaxations of a semi-algebraic set whose integer points encode graph isomorphisms.
\end{abstract}

\section{Introduction}

Classical combinatorial optimization problems have been traditionally approached
by means of linear programming relaxations.
Recently, there has been significant interest in understanding \emph{lift-and-project} techniques for
constructing hierarchies of such relaxations for combinatorial problems.  
These lift-and-project methods lift relaxations
to refined systems of polynomial equations and inequalities, and project back to the original space, offering tighter
approximations of the convex hull in question.  
Such lift-and-project procedures have been proposed by
Lov\'asz and Schrijver \cite{LovaszShrijver1991}, 
Lasserre \cite{Lasserre2001}, Sherali-Adams \cite{SheraliAdams1990}, among many others.
The key computational interest in these relaxations is that linear optimization over the $k^{th}$ relaxation can be done in polynomial time for fixed $k$. %(e.g. $O(n^k)$ time). 

Many relaxations of combinatorial problems have been investigated using these methods.  These methods have been shown not to converge in polynomial time for various NP-hard problems.  
Arora, Bollob\'{a}s, Lov\'{a}sz and Tourlakis \cite{AroraTourlakis} prove that Vertex Cover does not converge after $\Omega(\epsilon \log n)$ rounds of Lov\'{a}sz-Schrijver in the worst case, where $n$ is the number of vertices.  
Shoenebeck, Trevisan, and Tulsiani show that Max Cut does not converge after $\Omega(\epsilon n)$ rounds of Lov\'{a}sz-Schrijver in the worst case.  Shoenebeck further proved that Vertex Cover and $k$-Uniform Hypergraph Vertex Cover require at least $\Omega(n)$ rounds of the Lasserre relaxation in the worst case.  

The problems of determining whether a simple undirected graph has a non-trivial automorphism (graph automorphism problem) and whether two simple undirected graphs are isomorphic (graph isomorphism problem) are important problems in complexity theory as they are two of the few problems not known to be in $P$ nor to be $NP$-complete.
The combinatorial approaches such as the Weisfeiler-Lehman method (\cite{WeisfeilerLehman68}) do not give a polynomial time algorithm for graph automorphism or graph isomorphism (\cite{CaiFurer1992}).
It is then natural to study lift-and-project techniques to approach the graph isomorphism and automorphism problems applied to algebraic sets of points encoding graph automorphism and isomorphism.
In \cite{DeLoeraHillarMalkinOmar2010}, the authors investigated semidefinite relaxations of sets of points encoding the graph isomorphism and automorphism problems.
In this paper, we investigate the application of the Sherali-Adams technique to sets of points that encode the graph isomorphism and automorphism problems towards the ultimate goal of better understanding the Sherali-Adams technique and the complexity of these problems.

Specifically, we show that the Sherali-Adams relaxations of a well-known polytope encoding of graph isomorphism do not converge after $\Omega(n)$ steps in the worst case.
In order to do this, we show that the Sherali-Adams relaxations of this polytope correspond to an algorithm for graph isomorphism that is strongly related to the famous Weisfeiler-Lehman algorithm;
this result can be thought of as a $k$-dimensional analog of the work of Tinhofer \cite{Tinhofer1986}.
Lastly, we also give a semi-algebraic set encoding of graph isomorphism whose Sherali-Adams relaxations correspond to the Weisfeiler-Lehman algorithm.
There is thus a surprisingly strong correspondence between combinatorial approaches and polyhedral approaches to graph isomorphism.

\subsection{Background}

In this section, we introduce some of the most prevalent algorithms to deal with the graph isomorphism and automorphism problems.

For the remainder of the exposition, we consider simple undirected graphs with vertex set $V=\{1,2,...,n\}$.
We denote the set of edges of a graph $G$ by $E_G$ and the adjacency matrix of $G$ by $A_G$.
The neighbors of a vertex $u\in V$ are the vertices in the set  $\delta_G(u)=\{v\in V:\{u,v\}\in E_G\}$.
An \emph{isomorphism} from a graph $G$ to a graph $H$ (both with vertex set $V$) 
is a bijection $\psi: V \to V$ that is adjacency preserving.  That is, 
$\{u,v\} \in E_G$ if and only if $\{\psi(u),\psi(v)\} \in E_H$.  When $G=H$, we call
$\psi$ an \emph{automorphism} of $G$.
We denote the set of automorphisms of a graph $G$ as $AUT(G)$ and the set of isomorphisms
from graphs $G$ to $H$ as $ISO(G,H)$.

A famous combinatorial approach for the graph automorphism and isomorphism problems is the classic vertex classification algorithm (see e.g. \cite{CorneilGotlieb1970,ReadCorneil1977,McKay1980})  (C-V-C algorithm).
In the C-V-C automorphism algorithm, the vertex set $V$ of a graph $G$ is partitioned into equivalence classes $\{V_1,V_2,\ldots,V_m\}$ that are invariant under automorphisms, that is, for all automorphisms $\psi\in AUT(G)$, we have $\psi(V_i)=\{\psi(v):v\in V_i\}=V_i$ for all $1\leq i \leq m$.
If the C-V-C algorithm returns a complete partition, that is, $|V_i|=1$ for $1\leq i \leq m$, then $G$ is asymmetric (i.e. has no non-trivial automorphism).
The C-V-C algorithm starts from the trivial vertex partition $\{V\}$ and proceeds by iteratively refining the vertex partition based on the equivalence classes of the neighbors of a given vertex.
In particular, given a vertex partition $\{V_1,V_2,\ldots,V_m\}$, the vertices $u,v \in V_i$ for $1\le i\le m$ are in the same refined equivalence class in the next iteration if $|\delta_G(u)\cap V_j|=|\delta_G(v)\cap V_j|$ for $1\le j \le m$.
This process is repeated until the partition stabilizes, which happens in at most $n$ iterations.

For the isomorphism problem, the C-V-C isomorphism algorithm partitions the vertex sets of graphs $G$ and $H$ into equivalence classes $(V_1,...,V_m)$ and $(W_1,...,W_m)$ respectively such that for all isomorphisms $\psi\in ISO(G,H)$ we have $\psi(V_i)=W_i$.
If $|V_i|\ne |W_i|$ for some $1\leq i \leq m$, then there is no isomorphism from $G$ to $H$, and also, if there exists $u\in V_s$ and $v\in W_s$ such that $|\delta_G(u)\cap V_j|\ne|\delta_H(v)\cap W_j|$ for some $1\le j \le m$, then again there is no isomorphism from $G$ to $H$.
The sequences $(V_1,...,V_m)$ and $(W_1,...,W_m)$ are each constructed separately using
the C-V-C automorphism algorithm with the modification that the
sets in the partition are ordered at each stage in a way
that is invariant under isomorphism (see Section \ref{sec:VC} for details).

Despite the simplicity of the C-V-C algorithm, it works well in practice and is the basis of many implementations of graph isomorphism and graph automorphism testing including the well-known \texttt{nauty} package of McKay \cite{McKay1980}.  
Such empirical successes have been theoretically justified by many authors (see \cite{BabaiErdos1980}, \cite{BabaiLuks1983}, \cite{BabaiKucera1979}).
Notably, a slight variant of the C-V-C algorithm gives a linear-time graph isomorphism algorithm that
works for almost all graphs \cite{BabaiKucera1979}.
Despite its generic success, the C-V-C algorithm does not work at all for some of
the most important classes of graphs, namely regular graphs.
Indeed, if the graph is regular, then the C-V-C algorithm returns the trivial partition.
Despite this, the algorithm is a well known tool used in graph isomorphism algorithms in order to drastically reduce
 the search space of candidates, and the algorithm can be combined with backtracking technique in order to have a complete algorithm (see e.g., \cite{McKay1980}).

The ineffectiveness of the C-V-C algorithm on regular graphs prompted the need for an extension of the method giving rise to the Weisfeiler-Lehman algorithm \cite{Weisfeiler76, WeisfeilerLehman68} (W-L algorithm).
The aim of the $k$-dim W-L algorithm is to partition the set of $k$-tuples of vertices, $V^k=\{(u_1,...,u_k):u_1,...,u_k\in V\}$, into equivalence classes
$\{V_1^k,...,V^k_m\}$ that are invariant under automorphism meaning that $\psi(V^k_i)=\{(\psi(u_1),...,\psi(u_k)):u\in V^k_i\}=V^k_i$ for all $\psi\in AUT(G)$.
If the partition is complete, that is, $|V^k_i|=1$ for $i=1,...,m$, then the graph $G$ is asymmetric.
In the $k$-dim W-L algorithm, we start with a partition $\{V_1^k,...,V^k_m\}$ of $k$-tuples of vertices into subgraph type, that is, $u\equiv v$ if and only if 
$u_i=u_j \Leftrightarrow v_i=v_j$ and $\{u_i,u_j\} \in E_G \Leftrightarrow
\{v_i,v_j\} \in E_G$ for all $1\le i < j\le n$, or in other words, the ordered subgraphs of $G$ induced by the ordered sets of vertices $\{u_1,...,u_k\}$ and $\{v_1,...,v_k\}$ are identical.
Then, analogous to the C-V-C algorithm, we iteratively refine the partition as follows:
the vertices $u,v \in V^k_i$ for $1\le i\le m$ are in the same refined equivalence class in the next iteration if
there exists a bijection $\psi: V \rightarrow V$ such that $(u_1,...,u_{r-1},w,u_{r+1},...,u_k) \equiv (v_1,...,v_{r-1},\psi(w),v_{r+1},...,v_k)$
for all $1\leq r \leq k$ and for all $w\in V$. 
Note that the $n$-dim W-L algorithm is trivially necessary and sufficient due to the conditions on the initial partition.

For the isomorphism problem, analogous to the C-V-C algorithm, the $k$-dim W-L algorithm partitions $V^k$ into two sequences of equivalence classes $(V^k_1,...,V^k_m)$ for $G$ and $(W^k_1,...,W^k_m)$ for $H$ such that $\psi(V^k_i)=W^k_i$ for all $\psi\in ISO(G,H)$.
The partitions are constructed in a similar manner to the automorphism case (see Section \ref{sec:VC+WL} for details).

The $k$-dim W-L algorithm does not give a polynomial time algorithm for the graph isomorphism problem as established in \cite{CaiFurer1992} where they constructed a class of pairs of non-isomorphic graphs $(G_n,H_n)$ with $O(n)$ vertices such that the W-L algorithm needs $\Omega(n)$ iterations to distinguish the graphs.  
The graphs $(G_n,H_n)$ are presented in \cite{CaiFurer1992}.

Isomorphisms of graphs can naturally be represented by permutation matrices: an isomorphism $\psi \in ISO(G,H)$ can be represented by the $n$ by $n$ matrix $X_{\psi} = (X_{uv})_{1 \leq u,v \leq n}$ with $X_{uv} = 1$ if $\psi(u) = v$, and $X_{uv}=0$ otherwise.  We define $\Psi_{G,H}\subseteq \{0,1\}^{n\times n}$ to be the set of permutation matrices representing $ISO(G,H)$,
and we define $\Psi_G=\Psi_{G,G}$ as the set of permutation matrices
representing $AUT(G)$.
The isomorphism problem is thus the same problem as determining if $\Psi_{G,H}\ne \emptyset$, and the automorphism problem is the same problem as determining if $\Psi_G\neq \{\I\}$
where $\I$ is the identity matrix.
It is then natural to consider polyhedral relaxations of the sets $\Psi_{G,H}$ and
$\Psi_G$ and determine under what conditions the relaxations answer the isomorphism and automorphism problems.

Tinhofer \cite{Tinhofer1986} examined the following polyhedral relaxation of 
$\Psi_{G,H}$:
\[\T_{G,H}=\{X\in[0,1]^{n\times n}: XA_G=A_HX, Xe = X^Te = e\}\]
where $e$ is the $n$-dimensional column vector of 1's.
Note that the graphs $G$ and $H$ are isomorphic
if and only if there exists a permutation matrix $X\in \{0,1\}^{n\times n}$
such that $XA_G=A_HX$.  Thus $\Psi_{G,H}$ is precisely the integer points in the polytope $\T_{G,H}$, i.e. $\Psi_{G,H} = \T_{G,H} \cap \{0,1\}^{n \times n}$.  Tinhofer studied the polytopes $\T_{G,H}$ in hopes that most of them would be integral.
In \cite{Tinhofer1986}, Tinhofer showed that $\T_{G,H}=\emptyset$ if and only if the C-V-C algorithm determines that $G$ and $H$ are not isomorphic, and that $\T_G=\{\I\}$ if and only if the C-V-C algorithm determines that $G$ is asymmetric.

Tinhofer also established a strong relationship between $\T_{G,H}$ and the C-V-C algorithm.  
If $\{V_1,...,V_m\}$ is the partition of the vertices of $G$ given by the C-V-C algorithm, then for all $u,v \in V$, we have $u\not\equiv v$ if and only if $X_{uv}=0$ for all $X\in \T_G$.  Analogously, if $(V_1,...,V_m)$ is a partition of the vertices of $G$ and $(W_1,...,W_m)$ is a partition of the vertices of $H$ returned by the C-V-C algorithm,
then for all $u\in V_s, v\in W_t$, we have $s\ne t$ if and only if $X_{uv}=0$ for all $X\in \T_{G,H}$.
In other words, Tinhofer shows that the polyhedral relaxations $\T_{G,H}$ and $\T_G$ are geometric analogues of the C-V-C algorithm for the isomorphism and automorphism problems respectively.

In this paper, we examine Sherali-Adams relaxations of $\T_{G,H}$ and $\T_G$.
In general, given a semi-algebraic set 
$$P = \{x \in [0,1]^n \ | \ f_1(x)=0,...,f_s(x)=0, g_1(x)\ge 0,...,g_t(x)\ge0 \}$$
 where the $f_j(x)$ and $g_j(x)$ are polynomials in $\mathbb{R}[x_1,...,x_n]$,
the Sherali-Adams relaxations of $P$ are linear inequality descriptions of relaxations of $\conv(P\cap \{0,1\}^n)$, the convex hull
of the integer points $P$.
In particular, the Sherali-Adams hierarchy of relaxations $[0,1]^n\supseteq P^1 \supseteq \cdots \supseteq P^n = \conv(P\cap \{0,1\}^n)$ is a hierarchy of polytope relaxations of $P\cap \{0,1\}^n$ obtained by lifting $P$ to a non-linear polynomial system of equations and inequalities, linearizing the system giving an extended formulation and projecting back down.  

The $k^{th}$ relaxation $P^k$ is obtained as follows.  
First, we generate a set of polynomial inequalities given by
\begin{eqnarray}\label{eqn:sherali}
\prod_{i \in I} x_i  f_j(x) &=& 0 \quad \forall I \subseteq \{1,...,n\}, |I|\le k-1, 1\leq j \leq s\\
\prod_{i \in I} x_i \prod_{i \in J\setminus I} (1-x_i) g_j(x)&\ge& 0 \quad \forall I\subseteq J \subseteq \{1,...,n\}, |J|\le k-1,1\le j \le t,\\
\prod_{i \in I} x_i \prod_{i \in J\setminus I} (1-x_i)&\ge& 0 \quad \forall I\subseteq J \subseteq  \{1,...,n\}, |J|\le k.
\end{eqnarray}
All of these polynomial equations and inequalities, which include the original set of equations and inequalities, are satisfied on $P$.
We then expand the polynomials and make all monomials square-free by replacing any occurrence of $x_i^2$ by $x_i$ because $x_i^2 - x_i = 0$ is valid on $P\cap\{0,1\}^n$.
Next, we \emph{linearize} the system of equations
by replacing each monomial $\prod_{i \in I} x_i$ where $I\subseteq\{1,...,n\}$
and $|I|\geq 1$ with a new variable $y_I$.
The result is a set of linear inequalities describing a polyhedron $\hat{P}^k$ in extended $y$ space.  
To achieve $P^k$, we finally project the extended polyhedron $\hat{P}^k$ onto the space of $x_i=y_{\{i\}}$ variables,
that is, $P^k=\{x: x_i=y_{\{i\}}\ \forall 1\leq i \leq n, y \in \hat{P}^k\}$.  Note that at most $n$ iterations are needed before the Sherali-Adams relaxations converge to the convex hull,
that is, $P^n =  \conv(P\cap \{0,1\}^n)$ (see e.g. \cite{Laurent2003}).

\subsection{Our Contribution}

We generalize Tinhofer's work by studying the \emph{Sherali-Adams} relaxations of the polytopes $\T_{G,H}$ and $\T_G$.
First, we introduce a combinatorial algorithm, called the \emph{generalized vertex classification algorithm} or \emph{$k$-dimensional vertex classification algorithm} ($k$-dim C-V-C algorithm), whose relationship with $\T^k_G$ and $\T^k_{G,H}$ is analogous to the relationship between the C-V-C algorithm for the automorphism and isomorphism problem and $\T_G$ and $\T_{G,H}$ respectively.

For the automorphism problem, analogous to the C-V-C automorphism algorithm and the $k$-dim W-L automorphism algorithm, the $k$-dim C-V-C automorphism algorithm partitions $V^k$ into equivalence classes $\{V^k_1,...,V^k_m\}$ such that $\psi(V^k_i)=V^k_i$ for all $\psi\in AUT(G)$  (see Section \ref{sec:VC} for details).
For the isomorphism problem, analogous to the C-V-C algorithm and the $k$-dim W-L algorithm, the $k$-dim C-V-C algorithm partitions $V^k$ into two sequences of equivalence classes $(V^k_1,...,V^k_m)$ for $G$ and $(W^k_1,...,W^k_m)$ for $H$ such that $\psi(V^k_i)=W^k_i$ for all $\psi\in ISO(G,H)$.
The partitions are constructed in a similar manner to the automorphism case (see Section \ref{sec:VC}).

The following theorem summarizes the relationship between the $k$-dim C-V-C algorithm for the automorphism and isomorphism problem to the corresponding $k$th Sherali-Adams relaxations $\T^k_G$ and $\T^k_{G,H}$ respectively.

\begin{theorem}\label{thm:SheraliAdamsVC}
The polytope $\T^k_G= \{\I\}$ if and only if the $k$-dim C-V-C algorithm determines that $G$ is asymmetric.
The polytope $\T^k_{G,H}= \emptyset$ if and only if the $k$-dim C-V-C algorithm determines that $G$ and $H$ are not isomorphic.
\end{theorem}

We will actually prove a more general and stronger result than the above theorem that illustrates that the the polytopes $\T^k_G$ and $\T_{G,H}$ are geometric analogues of the $k$-dim C-V-C automorphism and isomorphism algorithms respectively 
(see Section \ref{sec:Comparison}). 

Above, we introduced a vertex classification algorithm that corresponds to the Sherali-Adams relaxations of a given polytope.
Somewhat surprisingly, we also found a relaxation of $\Psi_{G,H}$ whose Sherali-Adams relaxations naturally correspond to the W-L algorithm.
Consider the following semi-algebraic set $\Q_{G,H}$:
\begin{align*}
X_{u_1v_1}X_{u_2v_2} &= 0 \ \forall \{u_1,u_2\}\in E_G, \{v_1,v_2\}\not\in E_H,\\
X_{u_1v_1}X_{u_2v_2} &= 0 \ \forall \{u_1,u_2\}\not\in E_G, \{v_1,v_2\}\in E_H, \\
Xe = X^Te &= e, X\in[0,1]^{n\times n}.
\end{align*}
Note that $\Q_{G,H}\cap\{0,1\}^{n\times n} = \Psi_{G,H}$
because the equations $X_{u_1v_1}X_{u_2v_2}=0$ enforce that edges must map onto edges and non-edges (2-vertex independent sets) must map onto non-edges.
We define $\Q_G=\Q_{G,G}$, and $\Q^k_{G,H}$ (respectively $\Q^k_G$) as the $k$th Sherali-Adams relaxations of $\Q_{G,H}$ (respectively $\Q_G$).
The following theorem summarizes the relationship between the W-L algorithm and the Sherali-Adams relaxations of $\Q_{G,H}$.

\begin{theorem}\label{thm:SheraliAdamsWL}
Let $k>1$. The polytope $\Q^{k+1}_{G}= \{\I\}$ if and only if the $k$-dim W-L algorithm
determines that $G$ is asymmetric.
The polytope $\Q^{k+1}_{G,H}= \emptyset$ if and only if the $k$-dim W-L algorithm
determines that $G$ and $H$ are not isomorphic.
\end{theorem}

Again, we will actually prove a more general and stronger result than the above theorem that illustrates that the the polytopes $\Q^k_G$ and $\Q_{G,H}$ are geometric analogues of the $k$-dim W-L automorphism and isomorphism algorithms respectively (see Section \ref{sec:Comparison}).

The C-V-C algorithm and the W-L algorithm are strongly related:
the $k$-dim W-L algorithm is stronger than the $k$-dim C-V-C algorithm, but the $(k+1)$-dim C-V-C algorithm is stronger than the $k$-dim W-L algorithm.
This relationship is summarized in the following lemma, which compares the corresponding Sherali-Adams relaxations:
\begin{lemma}\label{lem:VCWL}
The inclusions $\Q^{k+1}_{G,H}\subseteq \T^k_{G,H} \subseteq \Q^{k}_{G,H}$ and $\Q^{k+1}_{G}\subseteq \T^k_{G} \subseteq \Q^{k}_{G}$ hold.
\end{lemma}
Together with Theorem \ref{thm:SheraliAdamsVC} and Theorem \ref{thm:SheraliAdamsWL},
Lemma \ref{lem:VCWL} implies that if the $k$-dim C-V-C algorithm determines that $G$ and $H$ are not isomorphic, then the $k$-dim W-L algorithm also determines that $G$ and $H$ are not isomorphic,
and if the $k$-dim W-L algorithm determines that $G$ and $H$ are not isomorphic, then the $(k+1)$-dim C-V-C algorithm determines that $G$ and $H$ are not isomorphic. Analogous statements hold for automorphism.

Combining Theorem \ref{thm:SheraliAdamsVC}, Theorem \ref{thm:SheraliAdamsWL} and Lemma \ref{lem:VCWL} and the result established by \cite{CaiFurer1992} on the complexity of the W-L algorithm, we arrive at the following complexity results for the Sherali-Adams relaxations of the graph isomorphism polytopes $\T^k_{G,H}$ and $\Q^k_{G,H}$.
\begin{corollary}\label{cor:Sherali}
There exists a class of pairs of non-isomorphic graphs $(G_n,H_n)$ where $G_n$ and $H_n$ have $O(n)$ vertices such that the Sherali-Adams procedure applied to the sets $\T_{G,H}$ and $\Q_{G,H}$ needs $\Omega(n)$ iterations to converge to $\conv(\Psi_{G_n,H_n})=\emptyset$.
\end{corollary}

Our paper is organized as follows.
In Section \ref{sec:CombAlg}, we give a detailed description of the combinatorial algorithms for graph automorphism and isomorphism.  
In Section \ref{sec:SA}, we describe the Sherali-Adams relaxation of the sets $\T_{G,H}$ and $\Q_{G,H}$.
Section \ref{sec:Comparison} shows that Sherali-Adams relations are geometric analogues of the corresponding combinatorial algorithms thus proving Theorems~\ref{thm:SheraliAdamsVC} and \ref{thm:SheraliAdamsWL}.

\section{Combinatorial Algorithms}
\label{sec:CombAlg}
In this section, we describe in detail the combinatorial algorithms for automorphism and isomorphism as described in the introduction, and we describe how they relate to each other.
There are strong similarities between the combinatorial algorithms we consider in this section.
The similarities are such that we first present a general framework for the algorithms to avoid having to prove analogous results.

\subsection{Automorphism Algorithms}
In this section, we present a general framework for an automorphism algorithm.
First, we give some necessary definitions and background.

We denote by $\Pi^k$ the set of all partitions of $V^k$.
For any $\pi \in \Pi^k$, we say two sets $u,v$ are equivalent with respect to $\pi$ if they lie in the same cell of $\pi$.  This is denoted by $u \equiv_{\pi} v$.
Given $\pi,\tau\in\Pi^k$, we write $\pi\leq \tau$ if $u\equiv_\pi v$ implies $u\equiv_\tau v$ for all $u,v\in V^k$.
It is well known that $\leq$ is a partial order on $\Pi^k$ and that $\Pi^k$ forms a complete lattice under $\leq$.  
That is, given $\pi,\tau\in \Pi^k$, there is a unique coarsest partition $\pi\wedge\tau\in\Pi^k$ such that $\pi\geq\pi\wedge\tau$ and $\tau\geq \pi\wedge\tau$ --
each cell of $\pi\wedge\tau$ is a non-empty intersection of a cell in $\pi$ and a cell in $\tau$.

Let $\Graphs$ be the set of all simple undirected graphs with vertex set $V=\{1,...,n\}$.
Let $\psi:V\rightarrow V$ be a bijection.
For a $k$-tuple $u\in V^k$, we define $\psi(u)=(\psi(u_1),...,\psi(u_k))$,
and for a set of $k$-tuples $W\subseteq V^k$, we define $\psi(W)=\{\psi(u):u\in W\}$.
Given a graph $G\in\Graphs$, a partition $\pi=\{\pi_1,...,\pi_m\}$ and a bijection $\psi:V \rightarrow V$, % we define $\psi(\pi)=\{\psi(\pi_1),...,\psi(\pi_m)\}$.
we define $AUT(G,\pi)=\{\psi\in AUT(G):\psi(\pi_i)=\pi_i \ \forall 1\leq i\leq m\}$.
In this paper, we address the more general graph automorphism question of whether $|AUT(G,\pi)|=1$ for a given partition $\pi\in\Pi^k$,
which is useful when using a backtracking algorithm for solving the automorphism problem (see for example \cite{McKay1980}).

We denote by $\OP{G}^k(\pi)\in \Pi^k$, the $k$-dimensional orbit partition of graph $G$ with respect to $\pi$, the partition where $u\equiv_{\OP{G}^k(\pi)} v$ if there exists $\psi \in AUT(G,\pi)$ such that $\psi(u)=v$, and we write $\OP{G}^k$ for the orbit partition of $G$.
\begin{example}
Let $G$ be the graph with vertex set $\{1,2,3,4\}$ and edge set 
\[
E_G = \{ \{1,2\},\{2,3\},\{2,4\},\{3,4\} \}.
\]
This graph has two automorphisms: the trivial automorphism and the automorphism that swaps vertices 3 and 4.
The orbit partition is $\OP{G}^k=\{V^2_1,...,V^2_{10}\}$ where
\begin{align*}
V_1^2 &= \{(1,1)\},&
V_2^2 &= \{(2,2)\},&
V_3^2 &= \{(3,3),(4,4)\},& \\
V_4^2 &= \{(1,2)\},&
V_5^2 &= \{(2,1)\},&
V_6^2 &= \{(1,3),(1,4)\},& \\
V_7^2 &= \{(3,1),(4,1)\},&
V_8^2 &= \{(2,3),(2,4)\},&
V_9^2 &= \{(3,2),(4,2))\},\\
V_{10}^2 &= \{(3,4),(4,3)\}.&
\end{align*}
\end{example}

We now define an equivalence relation that provides a general framework for graph automorphism.
\begin{definition}\label{def:alpha}
An equivalence relation $\alpha$ on $\Graphs \times \Pi^k \times V^k$ is a \emph{vertex classification (V-C) equivalence relation} if it has the following properties
for all $G\in\Graphs$, $\pi\in \Pi^k$ and $u,v\in V^k$:
\begin{enumerate}
\item $(G,\pi,u) \equiv_\alpha (G,\pi,v)$ implies $u\equiv_\pi v$; \label{def1}
\item $(G,\pi,u) \equiv_\alpha (G,\pi,v)$ implies $(G,\pi',v)\equiv_\alpha(G,\pi',u)$ for all $\pi'\in\Pi^k$ where $\pi\leq\pi'$; and \label{def2}
\item $(G,\pi,u)\equiv_\alpha(G,\pi,\psi(u))$ for all automorphisms $\psi\in AUT(G,\pi)$. \label{def3}
%\item $(G,\pi,u)\equiv_\alpha(G,\pi,v)$ implies $(G,\pi,\sigma(u))\equiv_\alpha(G,\pi,\sigma(v))$ for all $\sigma\in\Sym_k$ if $\pi$ is $\Sym_k$-invariant; \label{def4}
\end{enumerate}
\end{definition}

\begin{example}
For the C-V-C algorithm, we use the vertex classification relation $\delta$ where for a partition $\pi\in\Pi$ and graph $G\in\Graphs$, we define
$(G,\pi,u) \equiv_\delta (G,\pi,v)$ if and only if $u\equiv_{\pi}v$ and $|\delta_G(u)\cap \pi_i|=|\delta_G(v)\cap \pi_i|$ for all $\pi_i\in \pi.$
\end{example}

Using a V-C relation, we can construct a function that refines partitions for the graph $G$ as follows:
\begin{definition}\label{def:equiv fn}
Let $\alpha$ be a V-C equivalence relation, and let $G\in\Graphs$.
We define $\alpha_G: \Pi^k \rightarrow \Pi^k$ where $u \equiv_{\alpha_G(\pi)} v$ if $(G,\pi,u)\equiv_{\alpha}(G,\pi,v)$
for all $\pi\in\Pi^k$ and $u,v\in\Pi^k$. % for $\pi\in\Pi^k$.
We say that a partition is $\alpha_G$-stable if $\alpha_G(\pi)=\pi$,
and we denote by $\alpha^*_G(\pi)$ as the fixed point given by recursively applying $\alpha_G$ to $\pi$,
and specifically, we define $\alpha_G^k=\alpha_G^*(\{V^k\})$ where $\{V^k\}$ is the trivial partition.
\end{definition}
This function on partitions has the following useful properties: (1) $\pi\geq\alpha_G(\pi)$, which is clear from the definition;
(2) $\pi'\geq\pi$ implies $\alpha_G(\pi') \geq \alpha_G(\pi)$, which follows since $(G,\pi,u)\equiv_{\alpha}(G,\pi,v)$ implies $(G,\pi',u)\equiv_{\alpha}(G,\pi',v)$;
and (3) $\alpha_G(\OP{G}^k(\pi))=\OP{G}^k(\pi)$,
and moreover, $AUT(G,\pi)=AUT(G,\alpha_G(\pi))=AUT(G,\alpha_G^*(\pi))$,
which follows since Definition \ref{def:alpha} property (\ref{defv3}) implies that $u \equiv_{\alpha_G(\pi)} \psi(u)$ for all $\psi\in AUT(G,\pi)$ and for all $u,v\in V^k$.

The $k$-dim $\alpha$-V-C automorphism algorithm is thus to compute $\alpha^*(\pi)$.
We call such an algorithm a $k$-dimensional $\alpha$-Vertex Classification algorithm (or $k$-dim $\alpha$-V-C algorithm for short).
Since $\pi\geq \OP{G}^k(\pi)$, by the properties of $\alpha$, we must have $\pi\geq \alpha^*_G(\pi)\geq \OP{G}^k(\pi)$.
Thus, if $\alpha^*_G(\pi)$ is a complete partition, we have shown that $|AUT(G,\pi)|=1$.

From above, the orbit partitions $\OP{G}^k(\pi)$ are $\alpha_G$-stable for all $\pi\in\Pi^k$, and $\alpha^*_G(\pi)$ is $\alpha_G$-stable by definition.
The set of $\alpha_G$-stable partitions is closed under the $\vee$ operator meaning that if $\pi_1$ and $\pi_2$ are $\alpha_G$-stable, then $\pi_1\vee\pi_2$ is $\alpha_G$-stable. 
This follows since $\pi_1\vee\pi_2\geq\pi_1,\pi_2$ implies $\pi_1\vee\pi_2\geq\alpha_G(\pi_1\vee\pi_2)\geq\pi_1,\pi_2$, which implies that $\pi_1\vee\pi_2=\alpha_G(\pi_1\vee\pi_2)$ since $\pi_1\vee\pi_2$ is the unique finest partition coarser than $\pi_1$ and $\pi_2$.
Since the complete partition is $\alpha_G$-stable, we conclude that for every partition $\pi \in \Pi^k$, there is a unique coarsest $\alpha_G$-stable partition finer than $\pi$.
Crucially, we now show that $\alpha^*_G(\pi)$ is the unique coarsest $\alpha_G$-stable partition finer than $\pi$ as follows:
Let $\pi'$ be the unique coarsest $\alpha_G$-stable partition finer than $\pi$.
Since $\pi\geq\pi'$, we have $\alpha_G(\pi)\geq\alpha_G(\pi')=\pi'$,
and thus, $\alpha^*_G(\pi)\geq\pi'$, 
but $\alpha^*_G(\pi)$ is $\alpha_G$-stable and $\pi'$ is the unique coarsest $\alpha_G$-stable partition finer than $\pi$, 
and therefore, $\alpha^*_G(\pi) = \pi'$.

A simple but useful V-C equivalence relation is the \emph{combinatorial} equivalence relation $\CP$ on $\Graphs\times\Pi^k\times V^k$:
First, we define $\CP$ as a combinatorial equivalence relation on $V^k$ where $u\equiv_\CP v$ if $u_i=u_j \Leftrightarrow v_i=v_j$ for all $1\leq i,j\leq k$.
Also, we define $\CP$ as an equivalence relation on $\Graphs\times V^k$:
For all $G\in\Graphs$ and $u,v\in V^k$, we define $(G,u) \equiv_{\CP} (G,v)$ if $u\equiv_\CP v$ and $\{u_i,u_j\}\in E_G \Leftrightarrow \{v_i,v_j\}\in E_G$ for all $1\leq i,j\leq k$.
Then, for all $G\in\Graphs$, $\pi\in\Pi^k$ and $u,v\in V^k$, we define $(G,\pi,u) \equiv_{\CP} (G,\pi,v)$ if $u\equiv_\pi v$ and $(G,u)\equiv_\CP(H,v)$.
It is straight-forward to show that this is indeed a V-C equivalence relation.
As in Definition \ref{def:equiv fn}, we define the partition refinement function $\CP_G:\Pi^k\rightarrow\Pi^k$,
and a $\CP_G$-stable partition is thus a partition $\pi\in\Pi^k$ where $\CP_G(\pi)=\pi$.
Note that by construction, $\CP_G(\pi)=\CP_G^*(\pi)$ and $\pi$ is $\CP_G$-stable if and only if $\pi\leq\CP_G^k$.
We define $\SGP{G}^k = \CP_G^*(\{V^k\})$, that is, the partition where $u \equiv_{\SGP{G}^k} v$ if and only if $(G,u)\equiv_\CP(H,v)$.
This is the coarsest $\CP_G$-stable partition.

%We are interested in V-C equivalence relations that compute $\CP$-stable partitions, that is, if $\alpha$ is a V-C equivalence relation, we want that $\alpha_G^*(\pi)$ is $\CP$-stable for all $\pi\in\Pi^k$.
%If $\alpha$ does not have this property, then we have two options:
%(1) First, we can preprocess the given partition using $\CP$ and then apply $\alpha$, that is, we compute $\alpha_G^*(\CP^*_G(\pi))$ for all $\pi$.
%Since $\alpha_G^*(\CP^*_G(\pi))\leq \CP^*_G(\pi)) \leq \CP^k_G$, we have that $\alpha_G^*(\pi))$ is $\CP$-stable as required.
%(2) Second, we can simply combine $\CP$ and $\alpha$ giving another V-C equivalence relation as follows:
%we define the equivalence relation $\kappa$ on $\Graphs\times\Pi^k\times V^k$ where $(G,\pi,u)\equiv_{\kappa}(G,\pi,v)$ if $(G,\pi,u)\equiv_{\CP}(G,\pi,v)$ and $(G,\pi,u)\equiv_{\alpha}(G,\pi,v)$.
%It is straight-forward to verify that $\kappa$ is a V-C equivalence relation and that
%$\kappa^*_G(\pi)$ is $\CP$-stable for all $\pi\in\Pi^k$.
%These two approaches are equivalent in the sense that $\alpha_G^*(\CP^*_G(\pi)) = \kappa^*_G(\pi)$ for all $\pi\in\Pi^k$ since $\alpha_G^*(\CP_G^*(\pi))$ and $\kappa^*_G(\pi)$ are $\alpha$ and $\CP$ stable subpartitions of $\pi$.
%In this paper, we favour the second approach since we find it leads to a more concise exposition.

A useful property of partitions is invariance under permutation of tuple components, defined as follows:
Let $\sigma\in\Sym_k$ be a permutation where $\Sym_k$ is the symmetric group on $k$ elements.
For all tuples $u\in V^k$, we define $\sigma(u)$ as the permutation of the tuples components, that is, $\sigma(u)_i = u_{\sigma(i)}$ for all $1\leq i \leq k$.
We also define $\sigma(\pi)=\{\sigma(\pi_1),...,\sigma(\pi_m)\}$ where $\sigma(\pi_i)=\{\sigma(u):u\in \pi_i\}$ for all $1\leq i\leq m$.
We say a partition $\pi\in\Pi^k$ is $\Sym_k$-invariant if $\sigma(\pi)=\pi$ for all $\sigma\in\Sym_k$ 
or equivalently $\sigma(u) \equiv_{\pi} \sigma(v)$ for all $\sigma\in\Sym_k$
(or equivalently $u \equiv_{\sigma(\pi)} v$ for all $\sigma\in\Sym_k$).
For example, the partition $\SGP{G}^k$ is $\Sym_k$-invariant.
If a given partition $\pi$ is not $\Sym_k$-invariant, then we can easily compute the coarsest $\Sym_k$-invariant partition finer than $\pi$ using the following useful V-C equivalence relation.
First, let $\Sigma_k=\{\sigma_1,\sigma_2\}\subset \Sym_k$ be a generating set of $\Sym_k$; it is well known that we need only two permutations to generate $\Sym_k$: a two cycle and a $k$-cycle.
We define the \emph{symmetric} equivalence relation $\Sym$ on $\Pi^k\times V^k$ where $(\pi,u)\equiv_{\Sym}(\pi,v)$ if $\sigma_i(u) \equiv_\pi \sigma_i(v)$ for all $\sigma_i\in \Sigma_k$.
This equivalence relation induces a V-C equivalence relation on $\Graphs\times\Pi^k\times V^k$
where $(G,\pi,u) \equiv_\Sym (G,\pi,v)$ if $(\pi,u)\equiv_\Sym(\pi,v)$ for all $G\in\Graphs$, $\pi\in\Pi^k$ and $u,v\in V^k$.
As in Definition \ref{def:equiv fn}, we define the partition refinement function $\Sym:\Pi^k\rightarrow \Pi^k$, and a partition $\pi \in\Pi^k$ is \emph{$\Sym$-stable} if $\Sym(\pi)=\pi$,
and we denote by $\Sym^*(\pi)$ the unique coarsest $\Sym$-stable partition that is finer than $\pi$.
%Then, we have that $AUT(G,\pi)=AUT(G,\Sym^*(\pi))$ for all $G\in\Graphs$.
It is straight-forward to show that $\pi$ is $\Sym$-stable if and only if $\pi$ is $\Sym_k$-invariant, and thus, $\Sym^*(\pi)$ is $\Sym_k$-invariant;
thus, $\Sym^*(\pi)$ is the coarsest $\Sym_k$-invariant subpartition of $\pi$.
%It the following, we will use the term $\Sym$-stable instead of $\Sym_k$-invariant.

\subsubsection{Comparing and Combining V-C equivalence relations}
\label{sec:vc compare}
Here, we consider the general case of comparing two vertex classification equivalence relations $\alpha$ and $\kappa$.

First, Lemma \ref{lem:vc equiv} below gives a condition under which
the $k$-dim $\alpha$-V-C algorithm is at least as strong as the $k$-dim $\kappa$-V-C algorithm, that is, $\alpha^*_G(\pi) \leq \kappa^*_G(\pi)$
and $\alpha^*_G(\pi)$ is $\kappa_G$-stable,
in which case, we say that $\alpha$ implies $\kappa$.
\begin{lemma}\label{lem:vc equiv}
Let $G\in\Graphs$.
Let $\alpha$ and $\kappa$ be two V-C equivalence relations such that 
for all $\pi\in\Pi^k$ where $\alpha(\pi)=\pi$, we have
$(G,\pi,u)\equiv_{\alpha}(G,\pi,v)$ implies $(G,\pi,u)\equiv_{\kappa}(G,\pi,v)$ for all $u,v\in V^k$.
Then, for all $\pi\in\Pi^k$, we have $\alpha_G(\pi)=\pi$ implies $\kappa_G(\pi)=\pi$.
Moreover, for all $\pi\in\Pi^k$, we have $\alpha^*_G(\pi) \leq \kappa^*_G(\pi)$.
\end{lemma}
\begin{proof}
Firstly, assuming $\alpha(\pi)=\pi$, since $(G,\pi,u)\equiv_{\alpha}(G,\pi,v)$ implies $(G,\pi,u)\equiv_{\kappa}(G,\pi,v)$ for all $u,v\in V^k$, 
we immediately have that $\kappa(\pi)=\pi$.
Then, $\alpha^*_G(\pi)$ is $\alpha_G$-stable and thus $\kappa_G$-stable, and since $\kappa^*_G(\pi)$ is the unique maximal $\kappa_G$-stable subpartition of $\pi$, we must have $\alpha^*_G(\pi) \leq \kappa^*_G(\pi)$.
\end{proof}

Second, we consider when we wish to combine $\alpha$ and $\kappa$ in order to construct a stronger V-C algorithm that implies both $\alpha$ and $\kappa$; there are two ways we can combine the equivalence relations.
(1) First, we can define a new equivalence relation $\lambda$ on $\Graphs\times\Pi^k\times V^k$ where 
$(G,\pi,u)\equiv_{\lambda}(G,\pi,v)$ if $(G,\pi,u)\equiv_{\alpha}(G,\pi,v)$ and $(G,\pi,u)\equiv_{\kappa}(G,\pi,v)$
for all $G\in\Graphs$, $\pi\in\Pi^k$ and $u,v\in V^k$.
Then, clearly $\lambda$ is a V-C equivalence relation that implies $\alpha$ and $\kappa$, and by Lemma \ref{lem:vc equiv} above, we have
$\lambda^*_G(\pi) \leq \alpha^*_G(\pi),\kappa^*_G(\pi)$ and $\lambda^*_G(\pi)$ is $\alpha_G$ and $\kappa_G$ stable for all $\pi\in\Pi^k$.
(2) Second, we first apply $\kappa_G$ and then $\alpha_G$, that is, given $\pi$, we compute $\alpha_G^*(\kappa_G^*(\pi))$. This partition is $\alpha_G$-stable but not necessarily $\kappa_G$-stable unless $\alpha$ is $\kappa$-stable as defined below.
\begin{definition}\label{def:vc stable}
Let $\alpha,\kappa$ be two V-C equivalence relations. 
We say that $\alpha$ is $\kappa$-stable if for all $G\in\Graphs$ and $\pi\in\Pi$ where $\kappa_G(\pi)=\pi$, 
we have $\kappa_G(\alpha_G(\pi)) = \alpha_G(\pi)$.
\end{definition}
Thus, if $\alpha$ is $\kappa$-stable, then $\alpha_G^*(\kappa_G^*(\pi))$ is $\alpha_G$-stable and $\kappa_G$-stable,
and moreover, the two approaches are equivalent meaning that $\alpha_G^*(\kappa_G^*(\pi)) = \lambda_G^*(\pi)$.
This follows since $\lambda_G^*(\pi)\leq\kappa_G^*(\pi)$ implies $\lambda_G^*(\pi)\leq\alpha_G^*(\kappa_G^*(\pi))$,
and $\alpha_G^*(\kappa_G^*(\pi))\leq\pi$ implies $\alpha_G^*(\kappa_G^*(\pi))\leq\lambda_G^*(\pi)$ as required. 

Specifically, in this paper, we are interested in V-C equivalence relations that imply $\CP$ and $\Sym$.
So, given a V-C equivalence relation $\alpha$ that does not have this property, we will combine it with $\CP$ and $\Sym$.
The V-C equivalence relations we are interested in are $\CP$-stable and $\Sym$-stable, so we can apply either of the two approaches above.
In fact, every V-C equivalence relation is $\CP$-stable, since $\pi$ is $\CP_G$ stable if and only if $\pi\leq\CP_G^k$, and $\pi\leq\CP_G^k$ implies
$\alpha_G(\pi)\leq\CP_G^k$.
In this paper, we use the first approach for combining $\alpha$ with $\CP$ and $\Sym$ since we find it leads to a more concise exposition.

\subsection{Isomorphism Algorithms}
In this section, we present a general framework for an isomorphism algorithm.  First, we give some necessary definitions and background.

For isomorphism, we need to work with ordered partitions.
We denote by $\Piv^k$ the set of all ordered partitions of $V^k$, that is, $\Piv^k = \{ (\pi_1,...,\pi_m) : \{\pi_1,....,\pi_m\} \in \Pi^k\}$.
For an ordered partition $\piv\in\Piv^k$, we denote $\piv_i$ as the $i$th set in the partition.
For a partition $\piv=(\piv_1,...,\piv_m)\in\Piv^k$ and a bijection $\psi: V\rightarrow V$, 
we define $\psi(\piv)=(\psi(\piv_1),...,\psi(\piv_m))$.
We define $ISO(G,\piv, H, \tauv)=\{\psi\in ISO(G,H):\psi(\piv)=\tauv\}$.
Analogous to the automorphism algorithm, we address the more general graph isomorphism question of whether $ISO(G,\piv, H, \tauv) = \emptyset$ for given partitions $\piv,\tauv\in\Pi^k$, which is useful when using a backtracking algorithm for solving the isomorphism problem as in the automorphism case (see for example \cite{McKay1980}).

For a $k$-tuple $u\in V^k$ and an ordered partition $\piv\in\Piv^k$, we denote the index of the equivalence class of $u\in V^k$ as $[u]_\piv=i$ where $u\in\piv_i$.
For ordered partitions $\piv=(\pi_1,\ldots,\pi_m)\in\Piv^k$, there is a natural preorder $\leqq_\piv$ on $V^k$ extending the equivalence $\equiv_\piv$ relation as follows: for $u\in \piv_s, v \in \piv_t$, we write $u \leqq_\piv v$ if and only if $s\leq t$.
In the following, let $\piv,\tauv\in\Piv^k$.
We write $\piv \leq \tauv$ if $u \equiv_\piv v \Rightarrow u \equiv_\tauv v$ for all $u,v\in V^k$,
and we write $\piv \simeq \tauv$ if $u \equiv_\piv v \Leftrightarrow u \equiv_\tauv v$ for all $u,v\in V^k$.
We write $\piv \preceq \tauv$ if $u \leqq_\piv v$ implies $u \leqq_\tauv v$.
So, $\piv \preceq \tauv$ implies that $\piv \leq \tauv$.
We write $\piv \approx \tauv$ if $|\piv|=|\tauv|$ and $|\piv_i|=|\tauv_i|$ for all $1\leq i \leq |\piv|$.
We write $(\piv,\tauv) \leq(\piv',\tauv')$ if $[u]_\piv = [v]_\tauv \Rightarrow [u]_{\piv'} = [v]_{\tauv'}$,
and we write $(\piv,\tauv) \simeq(\piv',\tauv')$ if $[u]_\piv = [v]_\tauv \Leftrightarrow [u]_{\piv'} = [v]_{\tauv'}$.

\begin{definition}
\label{def:alphav}
We call a total preorder $\alphav$ on $\Graphs \times \Piv^k \times V^k$ a \emph{vertex classification (V-C) preorder} if  $\alphav$ has the following properties
for all $G,H\in\Graphs$, $\piv,\tauv\in \Piv^k$ and $u,v\in V^k$: 
\begin{enumerate}
\item $(G,\piv,u) \leqq_{\alphav} (H,\tauv,v)$ implies $[u]_\piv \leq [v]_\tauv$; \label{defv1}
\item $(G,\piv,u) \equiv_{\alphav} (H,\tauv,v)$ implies $(G,\piv',v)\equiv_{\alphav}(H,\tauv',u)$ for all $\piv',\tauv'\in\Piv^k$ where $(\piv,\tauv)\leq(\piv',\tauv')$; and \label{defv2}
\item $(G,\piv,u) \equiv_{\alphav} (\psi(G),\psi(\piv),\psi(u))$ for all bijections $\psi:V \rightarrow V$. \label{defv3}.
\end{enumerate}
\end{definition}

\begin{example}
For the C-V-C algorithm, we use the vertex classification relation $\delta$ as follows: given $G,H\in\Graphs$ and $\piv,\tauv\in\Piv$, we define
$(G,\piv,u) \leqq_{\deltabar} (H,\tauv,v)$ if and only if $[u]_\piv < [v]_\tauv$ or $[u]_\piv = [v]_\tauv$ and
$$(|\delta_G(u)\cap \piv_1|,...,|\delta_G(u)\cap \piv_m|) \leq_{lex} (|\delta_G(v)\cap \tauv_1|,...,|\delta_G(v)\cap \tauv_m|)$$
where $\leq_{lex}$ is the standard lexicographic ordering.
\end{example}

Using a V-C preorder, we can construct an ordered partition function that refines ordered partitions.
\begin{definition}\label{def:preorder fn}
Let $\alphav$ be a V-C preorder, and let $G\in\Graphs$.
We define $\alphav_G: \Piv^k \rightarrow \Piv^k$ where for all $\piv\in\Piv^k$ and $u,v\in V^k$, we have $u \leqq_{\alphav_G(\piv)} v$ if $(G,\piv,u)\leqq_{\alphav}(G,\piv,v)$.
We say that a partition $\piv$ is $\alphav_G$-stable if $\alphav_G(\piv) = \piv$,
and we denote by $\alphav^*_G(\piv)$ the fixed point reached by recursively applying $\alphav$ to $\piv$,
and specifically, we define $\alphav_G^k=\alphav_G^*((V^k))$ where $(V^k)$ is the trivial ordered partition.
\end{definition}
The function $\alphav_G$ has the properties that $\piv \succeq \alphav_G(\piv)$,
and crucially, $\psi(\alphav_G(\piv)) = \alphav_{\psi(G)}(\psi(\piv))$ since $(G,\piv,u) \equiv_\alphav (G,\piv,v)$ implies 
$(\psi(G),\psi(\piv),\psi(v))\equiv_\alphav(\psi(G),\psi(\piv),\psi(v))$ by property (3).

The concept of a V-C preorder is analogous to the concept of a V-C equivalence relation from before.
Specifically, a V-C preorder $\alphav$ induces a V-C equivalence relation $\alpha$ on $\Graphs\times\Pi^k\times V^k$ where we define $(G,\pi,u) \equiv_{\alpha} (G,\pi,v)$ if $(G,\piv,u) \equiv_{\alphav} (G,\piv,v)$ for all $\piv\simeq\pi$.
By property (2) above, $(G,\piv,u) \equiv_{\alphav} (G,\piv,v)$ if and only if $(G,\piv',u) \equiv_{\alphav} (G,\piv',v)$ for all $\piv'\simeq\piv$.
Thus, given $\piv\simeq\pi$, we have $\alphav_G(\piv)\simeq\alpha_G(\pi)$ and $\alphav^*_G(\piv)\simeq\alpha_G^*(\pi)$, and also, the partition $\piv$ is $\alphav_G$-stable if and only if $\pi$ is $\alpha_G$-stable.
This is useful since it means that proving results for $\alphav$ implies results for $\alpha$.

Before describing the algorithm for graph isomorphism, we need to define an equivalence relation on $\Graphs\times\Piv^k$.
This equivalence condition will serve as a useful necessary but not sufficient condition for isomorphism.
\begin{definition}\label{def:preorder equiv}
Given a V-C preorder $\alphav$, we define an equivalence relation $\equiv_{\alphav}$ on $\Graphs\times \Piv^k$ where for all $G,H\in\Graphs$, $\piv,\tauv\in\Piv^k$ and $u,v\in V^k$, we have $(G,\piv) \equiv_{\alphav} (H,\tauv)$ if $\alphav_G(\piv)\approx\alphav_H(\tauv)$ and $(G,\piv,u) \equiv_{\alphav}(H,\tauv,v) \Leftrightarrow [u]_{\alphav_G(\piv)}=[v]_{\alphav_H(\tauv)}$.
\end{definition}
In other words, $(G,\piv) \equiv_{\alphav} (H,\tauv)$  if and only if there exists a bijection of tuples $\gamma: V^k\rightarrow V^k$ where 
for all $u\in V^k$, we have $(G,\piv,u) \equiv_{\alphav}(H,\tauv,\gamma(u))$. %, in which case, we have $\gamma(\alphav_G(\piv))=\alphav_H(\tauv)$.
The equivalence relation $\equiv_{\alphav}$ has analogous properties to those defining the preorder $\alphav$.
Note that we always have $(G,\piv) \equiv_{\alphav} (G,\piv)$.
\begin{lemma}\label{lem:equiv props}
Let $\alphav$ be a vertex classification preorder on $\Graphs\times\Piv^k\times V^k$.
Then, for all $G,H\in\Graphs$ and $\piv,\tauv\in\Piv^k$,
\begin{enumerate}
\item $(G,\piv) \equiv_{\alphav} (H,\tauv)$ implies $(\alphav_G(\piv), \alphav_H(\tauv))\leq (\piv,\tauv)$ and $\piv \approx \tauv$;
\item $(G,\piv) \equiv_{\alphav} (H,\tauv)$ implies $(G,\piv')\equiv_{\alphav}(H,\tauv')$ and $(\alphav_G(\piv),\alphav_H(\tauv))\leq(\alphav_G(\piv'),\alphav_H(\tauv'))$ for all $\piv',\tauv'\in\Piv^k$ where $(\piv,\tauv)\leq(\piv',\tauv')$; and
\item $(G,\piv) \equiv_{\alphav}(\psi(G),\psi(\piv))$ for all bijections $\psi:V\rightarrow V$.
\end{enumerate}
\end{lemma}
\begin{proof}
(1) Assume $(G,\piv) \equiv_{\alphav} (H,\tauv)$.
Since $[u]_{\alphav_G(\piv)} = [v]_{\alphav_H(\tauv)}$ implies $(G,\piv,u) \equiv_{\alphav}(H,\tauv,v)$ and thus $[u]_\piv = [v]_\tauv$, we have $(\alphav_G(\piv), \alphav_H(\tauv))\leq (\piv,\tauv)$,
which together with the fact that $\alphav_G(\piv)\approx\alphav_H(\tauv)$ implies $\piv \approx \tauv$.

(2) Again assume $(G,\piv) \equiv_{\alphav} (H,\tauv)$.
First, we have $(\piv',\tauv')\geq(\piv,\tauv)$ implies $(\alphav_G(\piv'),\alphav_H(\tauv'))\geq(\alphav_G(\piv),\alphav_H(\tauv))$.
This follows since $[u]_{\alphav_G(\piv)}=[v]_{\alphav_H(\tauv)}$ implies $(G,\piv,u)\equiv_{\alphav}(G,\tauv,v)$, which by Definition \ref{def:alphav} implies that $(G,\piv',u)\equiv_{\alphav}(H,\tauv',v)$, and therefore,
$[u]_{\alphav_G(\piv')} = [v]_{\alphav_H(\tauv')}$ as required.
Second, since $(G,\piv) \equiv_{\alphav} (H,\tauv)$, there exists a tuple bijection $\gamma:V^k\rightarrow V^k$ such that $(G,\piv,u) \equiv_{\alphav} (H,\tauv,\gamma(u))$ for all $u \in V^k$ implying that $(G,\piv',u) \equiv_{\alphav} (H,\tauv',\gamma(u))$ for all $u \in V^k$ by property of Definition \ref{def:alphav}.
Hence, $(G,\piv') \equiv_{\alphav} (H,\tauv')$. 

(3) Let $\psi:V\rightarrow V$ by a bijection, which thus induces a tuple bijection $\psi:V^k\rightarrow V^k$.
Then, from property (3) of Definition \ref{def:alphav}, we have $(G,\piv,u)\equiv_{\alphav}(\psi(G),\psi(\piv),\psi(u))$ for all $u\in V^k$,
and thus, $(G,\piv)\equiv_\alphav(\psi(G),\psi(\piv))$.
\end{proof}

It follows from part (3) of Lemma \ref{lem:equiv props} above that $(G,\piv) \not\equiv_{\alphav}(H,\tauv)$ implies $ISO(G,\piv,H,\tauv)=\emptyset$.
Moreover, if $(G,\piv) \equiv_{\alphav}(H,\tauv)$, then $ISO(G,\piv,H,\tauv) = ISO(G,\alphav_G(\piv),H,\alphav_H(\tauv))$.
This follows from property (3) of Definition \ref{def:alphav}, which says that if $(G,\piv,u) \not\equiv_{\alphav}(H,\tauv,v)$, then there does not exist $\psi\in ISO(G,\piv,H,\tauv)$ such that $\psi(u)=v$. 
Then, it follows from the definition of $(G,\piv) \equiv_{\alphav}(H,\tauv)$ that 
%$\alphav_G(\piv)\approx\alphav_H(\tauv)$ and $(G,\piv,u) \equiv_{\alphav}(H,\tauv,v) \Leftrightarrow [u]_{\alphav_G(\piv)}=[v]_{\alphav_H(\tauv)}$, we thus have 
$ISO(G,\piv,H,\tauv) = ISO(G,\alphav_G(\piv),H,\alphav_H(\tauv))$.
Hence, if $(G,\alphav_G(\piv)) \not \equiv_{\alphav} (H,\alphav_H(\tauv))$, then $ISO(G,\alphav_G(\piv),H,\alphav_H(\tauv)) =\emptyset = ISO(G,\piv,H,\tauv)$.

We show below in Lemma \ref{lem:iso} that under the condition that $\piv\approx\tauv$, we have $ISO(G,\piv,H,\tauv)=ISO(G,\alpha_G^*(\piv),H,\alpha_H^*(\tauv))$.
So, if $(G,\alphav^*_G(\piv))\not\equiv_{\alphav} (H,\alphav^*_H(\tauv))$, then $ISO(G,\alpha_G^*(\piv),H,\alpha_H^*(\tauv)) =\emptyset$,
and therefore, $ISO(G,\piv,H,\tauv)=\emptyset$.
We can now present the $\alphav$-V-C algorithm for isomorphism given $G,H\in\Graphs$ and $\piv,\tauv\in\Piv^k$ where $\piv\approx\tauv$:
first, we compute $\alphav^*_G(\piv)$; second, we compute $\alphav^*_H(\tauv)$; and finally, we check if
$(G,\alphav^*_G(\piv))\not \equiv_{\alphav} (H,\alphav^*_H(\tauv))$, and if so, then we have shown that $ISO(G,\piv,H,\tauv)=\emptyset$.
We may assume without loss of generality that $\piv\approx\tauv$ since it is easily verifiable and if $\piv\not\approx\tauv$, then trivially $ISO(G,\piv,H,\tauv)=\emptyset$.

The question arises whether during the isomorphism algorithm, we should not only check if $(G,\alphav^*_G(\piv))\equiv_{\alphav} (H,\alphav^*_H(\tauv))$,
but we should also check if $(G,\alphav^r_G(\piv))\equiv_{\alphav} (H,\alphav^r_H(\tauv))$ for all $r$ where $\alphav_G^r$ and $\alpha_H^r$ are the results of applying $\alphav_G$ and $\alphav_H$ respectively recursively $r$ times.
The following useful lemma and corollary establishes that this is not necessary.
\begin{lemma}\label{lem:equiv for all r}
Let $G,H\in\Graphs$ and let $\piv,\tauv,\piv',\tauv'\in\Piv^k$ where $(G,\piv')\equiv_{\alphav^*}(H,\tauv')$ and 
$(\piv',\tauv')\leq(\piv,\tauv)$.
For all $r\geq 0$,  we have $(\piv',\tauv')\leq (\alphav_G^r(\piv),\alphav_H^r(\tauv))$ and $(G,\alphav^r_G(\piv)) \equiv_{\alphav}(H,\alphav^r_H(\tauv))$. Also, $(G,\alphav^*_G(\piv)) \equiv_{\alphav}(H,\alphav^*_H(\tauv))$.
\end{lemma}
\begin{proof}
We show by induction that for all $r\geq0$, we have $(\piv',\tauv')\leq(\alphav_G^r(\piv),\alphav_H^r(\tauv))$ thus implying that $(G,\alphav_G^r(\piv))\equiv_{\alphav}(H,\alphav_H^r(\tauv))$ by Lemma \ref{lem:equiv props} as required.
It is true for $r=0$. Assume it is true for some $r$.
Then, $(\piv',\tauv')\leq(\alphav_G^r(\piv),\alphav_H^r(\tauv))$ implies that 
$(\piv',\tauv')\leq(\alphav_G^{r+1}(\piv),\alphav_H^{r+1}(\tauv))$ by Lemma \ref{lem:equiv props} part (2) and
since $\alphav_G(\piv')=\piv'$ and $\alphav_H(\tauv')=\tauv'$.
\end{proof}

\begin{corollary}\label{cor:equiv for all r}
Let $G,H\in\Graphs$ and let $\piv,\tauv\in\Piv^k$ where $\piv\approx \tauv$ and $(G,\alphav^*_G(\piv)) \equiv_{\alphav}(H,\alphav^*_H(\tauv))$.
Then for all $r$,  we have $(\alphav^*_G(\piv), \alphav^*_H(\tauv))\leq (\alphav_G^r(\piv),\alphav_H^r(\tauv))$ and $(G,\alphav^r_G(\piv)) \equiv_{\alphav}(H,\alphav^r_H(\tauv))$.
\end{corollary}
\begin{proof}
Since $\alphav_G^*(\piv)\preceq\piv$, $\alphav^*_H(\tauv)\preceq\tauv$, $\alphav_G^*(\piv)\approx\alphav_H^*(\tauv)$
and $\piv\approx\tauv$, we have $(\alphav_G^*(\piv),\alphav_H^*(\tauv))\leq (\piv,\tauv)$.
The result follows immediately by applying Lemma \ref{lem:equiv for all r}.
\end{proof}

Using the above lemma, we can now show that $ISO(G,\piv,H,\tauv)=ISO(G,\alpha_G^*(\piv),H,\alpha_H^*(\tauv))$ where $\piv\approx\tauv$.
\begin{lemma}\label{lem:iso}
Let $G,H\in\Graphs$. Let $\piv,\tauv\in\Piv^k$ where $\piv\approx \tauv$.
Then, $ISO(G,\piv,H,\tauv)=ISO(G,\alpha_G^*(\piv),H,\alpha_H^*(\tauv))$.
\end{lemma}
\begin{proof}
Recall that for all $\piv',\tauv'\in\Piv^k$, if $(G,\piv') \equiv_{\alphav}(H,\tauv')$, then $ISO(G,\piv',H,\tauv') = ISO(G,\alphav_G(\piv'),H,\alphav_H(\tauv'))$, and if $(G,\piv') \not\equiv_{\alphav}(H,\tauv')$, then $ISO(G,\piv',H,\tauv')=\emptyset$.
Firstly, assume $(G,\alpha_G^*(\piv))\equiv_\alphav (H,\alpha_H^*(\tauv))$.
Then, by Corollary \ref{cor:equiv for all r}, for all $r\geq0$, we have $(G,\alphav^r_G(\piv)) \equiv_{\alphav}(H,\alphav^r_H(\tauv))$ implying that
$ISO(G,\alpha_G^r(\piv),H,\alpha_H^r(\tauv)) = ISO(G,\alphav_G^{r+1}(\piv'),H,\alphav_H^{r+1}(\tauv'))$.
Thus, $ISO(G,\piv,H,\tauv)=ISO(G,\alpha_G^*(\piv),H,\alpha_H^*(\tauv))$.

Secondly, assume $(G,\alpha_G^*(\piv))\not\equiv_\alphav (H,\alpha_H^*(\tauv))$.
Then, $ISO(G,\alpha_G^*(\piv),H,\alpha_H^*(\tauv))=\emptyset$.
There exists $r\geq0$ such that $(G,\alphav^r_G(\piv)) \not\equiv_{\alphav}(H,\alphav^r_H(\tauv))$ and 
$(G,\alphav^{\bar{r}}_G(\piv))\equiv_{\alphav}(H,\alphav^{\bar{r}}_H(\tauv))$ for all $\bar{r}<r$.
Then, we have $ISO(G,\alpha_G^r(\piv),H,\alpha_H^r(\tauv))=\emptyset$ since $(G,\alphav^r_G(\piv)) \not\equiv_{\alphav}(H,\alphav^r_H(\tauv))$, and also, we have $ISO(G,\piv,H,\tauv)=ISO(G,\alpha_G^r(\piv),H,\alpha_H^r(\tauv))$ since $ISO(G,\alpha_G^{\bar{r}}(\piv),H,\alpha_H^{\bar{r}}(\tauv)) = ISO(G,\alphav_G^{\bar{r}+1}(\piv'),H,\alphav_H^{\bar{r}+1}(\tauv'))$ for all $\bar{r}<r$.
Thus, again $ISO(G,\piv,H,\tauv)=ISO(G,\alpha_G^*(\piv),H,\alpha_H^*(\tauv))$ as required.
\end{proof}

There is one more definition we will use, which we define for notational convenience.
\begin{definition}\label{def:preorder equiv star}
Given a V-C preorder $\alphav$, we define the equivalence relation $\equiv_{\alphav^*}$ on $\Graphs\times \Piv^k$ where for all $G,H\in\Graphs$ and $\piv,\tauv\in\Piv^k$, we have $(G,\piv) \equiv_{\alphav^*} (H,\tauv)$ if $\alphav_G(\piv)=\piv$, $\alphav_H(\tauv)=\tauv$ and $(G,\piv) \equiv_{\alphav} (H,\tauv)$.
\end{definition}

Analogous to the combinatorial V-C equivalence relation $\CP$, we define the useful \emph{combinatorial} preorder $\CPv$ on $\Graphs\times\Piv^k\times V^k$ as follows:
First, we define the matrix $\CPv_G(u)\in\R^{k\times k}$ where $\CPv_G(u)_{ij} = 0$ if $u_i = u_j$ and $\CPv_G(u)_{ij} = 1$ if $\{u_i,u_j\}\in E_G$ and $\CPv_G(u)_{ij} = -1$ otherwise.
We then define the preorder $\CPv$ on $\Graphs\times\Piv^k \times V^k$ where $(G,\piv,u) \leqq_{\CPv} (H,\tauv,v)$ if
$([u]_\piv,\CPv_G(u)) \leq_{lex} ([v]_\tauv,\CPv_G(v))$ for all $u,v\in V^k$. 
It is straight-forward to verify that $\CPv$ is a V-C preorder.
Then, as in Definition \ref{def:preorder fn}, we define the function $\CPv_G:V^k\rightarrow V^k$ where $u \leqq_{\CPv_G(\piv)} v$ if $(G,\piv,u)\leqq_{\CPv}(G,\piv,v)$
for all $\piv\in\Piv^k$ and $u,v\in V^k$. We say $\piv$ is $\CPv_G$-stable if $\CPv_G(\piv)=\piv$, 
and we denote by $\CPv_G^*(\piv)$ the fixed point reached by recursively applying $\CPv_G$ to $\piv$.
Also, as in Definition \ref{def:preorder equiv}, we define the $\CPv$ equivalence relation on $\Graphs\times\Piv^k$.
We define   $\CPv_G^k = \CPv_G^*((V^k))$, so $u \leqq_{\CPv^k} v$ if $\CPv_G(u) \leq_{lex} \CPv_G(v)$.
Note that $(G,\piv) \equiv_{\CPv^*}(H,\tauv)$ if and only if $(\piv,\tauv)\leq(\CPv_G^k,\CPv_H^k)$ and $\piv\approx\tauv$.

Analogous to the unordered case, a useful property of pairs ordered partitions is invariance under permutation of tuple components, defined as follows.
Let $\piv\in\Piv^k$.
We define $\sigma(\piv)=(\sigma(\piv_1),...,\sigma(\piv_m))$ for all $\sigma\in\Sym_k$.
Given $\piv,\tauv\in\Piv^k$ where $\piv\approx\tauv$, we say that the pair $(\piv,\tauv)$ is $\Sym_k$-invariant if $(\piv,\tauv)\simeq (\sigma(\piv), \sigma(\tauv))$ for all $\sigma\in\Sym_k$, that is, 
for all $u,v\in V^k$, we have $[u]_\piv = [v]_\tauv\Leftrightarrow [u]_{\sigma(\piv)} = [v]_{\sigma(\tauv)}$  for all $\sigma\in\Sym_k$ (or equivalently $[u]_\piv = [v]_\tauv\Leftrightarrow[\sigma(u)]_\piv = [\sigma(v)]_\tauv$ for all $\sigma\in\Sym_k$).
%We are interested in V-C preorders that preserve $\Sym_k$-invariance for pairs of tuples.

%Also, without loss of generality, for the initial partitions $\piv,\tauv\in\Piv^k$, we will assume that $(\piv,\tauv)$ is $\Sym_k$-invariant.
If we are given partitions $\piv,\tauv\in \Piv^k$ such that $(\piv,\tauv)$ is not $\Sym_k$-invariant, then we can construct such initial partitions using a V-C preorder analogous to the symmetric equivalence relations as follows.
Let $\Sigma_k=\{\sigma_1,\sigma_2\}$ be a generating set of $\Sym_k$ as before.
For all $\piv\in\Piv^k$ and $u\in V^k$, we define $\Symv_\piv(u)=([\sigma_1(u)]_\piv,[\sigma_2(u)]_\piv)$.
Analogous to $\Sym$, we define the \emph{symmetric} V-C preorder $\Symv$ on $\Piv^k\times V^k$ where $(\piv,u) \leqq_{\Symv} (\tauv,v)$ if
$([u]_\piv,\Symv_\piv(u)) \leq_{lex} ([v]_\tauv,\Symv_\tauv(v))$ for all $u,v\in V^k$. 
The preorder $\Symv$ induces a preorder on $\Graphs\times\Piv^k\times V^k$ where we define $(G,\piv,u)\leqq_{\Symv}(G,\tauv,v)$ if $(\piv,u)\leqq_{\Symv}(\tauv,v)$
for all $G,H\in\Graphs$, $\piv,\tauv\in\Piv^r$ and $u,v\in V^k$.
It is straight-forward to verify that $\Symv$ is a V-C preorder.
Then, as in Definition \ref{def:preorder fn}, we define the function $\Symv:V^k\rightarrow V^k$, the concept of $\Symv$-stable, and the partition 
$\Symv^*(\piv)$.
Also, as in Definition \ref{def:preorder equiv}, we define the $\Symv$ equivalence relation on $\Piv^k$ inducing an equivalence relation on $\Graphs\times\Piv^k$.
%Then, we have that $\Sym^*(\piv)\not\equiv_\Sym \Sym^*(\tauv)$ implies that $ISO(G,\piv, H, \tauv)=\emptyset$ for all $G,H\in\Graphs$.
Furthermore, it is straight-forward to show that for all $\piv,\tauv\in\Piv^k$ where $\piv\approx\tauv$, the pair $(\piv,\tauv)$ is $\Sym_k$-invariant if and only if $\piv\equiv_{\Symv^*}\tauv$.

\subsubsection{Comparing and Combining V-C preorders}
\label{sec:vcv compare}
Next, we consider the general case of comparing and combining two different vertex classification preorders $\alphav$ and $\kappav$.
First, Lemma \ref{lem:vcv equiv} shows when 
the $k$-dim $\alpha$-V-C algorithm is at least as strong as the $k$-dim $\kappa$-V-C algorithm, in which case, we say that $\alphav$ implies $\kappav$.
\begin{lemma}\label{lem:vcv equiv}
Let $G,H\in\Graphs$.
Let $\alphav$ and $\kappav$ be two V-C preorders where
for all $\piv,\tauv\in\Piv^k$ where $(G,\piv)\equiv_{\alphav^*}(H,\tauv)$,
we have $(G,\piv,u)\equiv_{\alphav}(H,\tauv,v)$ implies $(G,\piv,u)\equiv_{\kappav}(H,\tauv,v)$ for all $u,v\in V^k$.
Then, for all $\piv,\tauv\in\Piv^k$ where $\piv\approx\tauv$, we have
$(G,\piv)\equiv_{\alphav^*}(H,\tauv)$ implies $(G,\piv)\equiv_{\kappav^*}(H,\tauv)$, and
$(G,\alphav_G^*(\piv))\equiv_{\alphav}(H,\alphav_H^*(\tauv))$ implies 
$(G,\kappav_G^*(\piv))\equiv_{\kappav}(H,\kappav_H^*(\tauv))$  and $(\alphav_G^*(\piv),\alphav_H^*(\tauv))\leq(\kappav_G^*(\piv),\kappav_H^*(\tauv))$.
\end{lemma}
\begin{proof}
Assume, $(G,\piv)\equiv_{\alphav^*}(G,\tauv)$.
Then, from Lemma \ref{lem:vc equiv}, we have $\kappav_G(\piv)=\piv$ and $\kappav_H(\tauv)=\tauv$.
Also, since $(G,\piv)\equiv_{\alphav}(G,\tauv)$, there exists a tuple bijection $\gamma:V^k\rightarrow V^k$ such that $(G,\piv,u)\equiv_\alphav(H,\tauv,\gamma(u))$ for all $u\in V^k$, 
implying that $(G,\piv,u)\equiv_\kappav(H,\tauv,\gamma(u))$ for all $u\in V^k$, and thus, $(G,\piv)\equiv_{\kappav}(G,\tauv)$ as required.

Assume $(G,\alphav_G^*(\piv))\equiv_{\alphav}(H,\alphav_H^*(\tauv))$.
Using Corollary \ref{cor:equiv for all r}, we have $(\alphav_G^*(\piv),\alphav_H^*(\tauv))\leq (\piv,\tauv)$.
From part (1), we have that $(G,\alphav_G^*(\piv))\equiv_{\kappav}(H,\alphav_H^*(\tauv))$.
Then, using Lemma \ref{lem:equiv for all r}, we have 
$(G,\kappav_G^*(\piv))\equiv_{\kappav}(H,\kappav_H^*(\tauv))$  and 
$(\alphav_G^*(\piv),\alphav_H^*(\tauv))\leq(\kappav_G^*(\piv),\kappav_H^*(\tauv))$ as required.
\end{proof}

A key point here is that we do \emph{not} need the stronger condition that $(G,\piv,u)\leqq_{\alphav}(H,\tauv,v)$ implies $(G,\piv,u)\leqq_{\kappav}(H,\tauv,v)$.
The preorder is not important; only the implied equivalence relation is important, but the preorder is necessary in order to compare different partitions.

Moreover, if $\alphav$ implies $\kappav$, then $\alpha$ implies $\kappa$ where $\alpha$ and $\kappa$ are the induced V-C equivalence relations on $\Graphs\times\Pi^k\times V^k$ for $\alphav$ and $\kappav$ respectively.
This follows since $(G,\piv,u)\equiv_{\alphav}(G,\piv,v)$ if and only if $(G,\pi,u)\equiv_{\alpha}(G,\pi,v)$ where $\piv\simeq\pi$
and similarly for $\kappav$, and thus, if $(G,\piv,u)\equiv_{\alphav}(H,\tauv,v)$ implies $(G,\piv,u)\equiv_{\kappav}(H,\tauv,v)$, then 
$(G,\pi,u)\equiv_{\alpha}(G,\pi,v)$ implies $(G,\pi,u)\equiv_{\kappa}(G,\pi,v)$.

As in the automorphism case, we also consider when we wish to combine $\alphav$ and $\kappav$ in order to construct a stronger V-C algorithm that implies $\alphav$ and $\kappav$; there are two ways we can combine the preorders.
(1) First, we can define a new equivalence relation $\lambdav$ on $\Graphs\times\Pi^k\times V^k$ where $(G,\piv,u)\leqq_{\lambdav}(G,\piv,v)$ if $(G,\piv,u)\lneqq_{\alphav}(G,\piv,v)$ or 
$(G,\piv,u)\equiv_{\alphav}(G,\piv,v)$ and $(G,\piv,u)\leqq_{\kappav}(G,\piv,v)$
for all $G,H\in\Graphs$, $\piv,\tauv\in\Piv^k$ and $u,v\in V^k$.
Then, clearly $\lambdav$ is a V-C preorder, and by Lemma \ref{lem:vcv equiv} above, $\lambdav$ implies $\alphav$ and $\kappav$, that is,
for all $G,H\in\Graphs$, $\piv,\tauv\in\Piv^k$ where $\piv\approx\tauv$,
 we have $(G,\piv)\equiv_{\lambdav^*}(H,\tauv)$ implies $(G,\piv)\equiv_{\alphav^*}(H,\tauv)$ and 
$(G,\piv)\equiv_{\kappav^*}(H,\tauv)$, and $(G,\lambdav_G^*(\piv))\equiv_{\lambdav}(H,\lambdav_H^*(\tauv))$ implies
$(G,\alphav_G^*(\piv))\equiv_{\alphav}(H,\alphav_H^*(\tauv))$,
$(G,\kappav_G^*(\piv))\equiv_{\kappav}(H,\kappav_H^*(\tauv))$ and 
$(\lambdav_G^*(\piv),\lambdav_H^*(\tauv))\leq(\alphav_G^*(\piv),\alphav_H^*(\tauv)),(\kappav_G^*(\piv),\kappav_H^*(\tauv))$.
(2) Second, we first apply $\kappav$ and then $\alphav$, that is, given $G,H\in\Graphs$ and $\piv,\tauv\in\Piv^k$ where $\piv\approx\tauv$, 
we check if $(G,\kappav_G^*(\piv))\equiv_{\kappav}(H,\kappav_H^*(\tauv))$ and if
$(G,\alphav_G^*(\kappav_G^*(\piv))) \equiv_{\alphav}(H,\alphav_H^*(\kappav^*(\tauv)))$. 
We do not necessarily have $(G,\alphav_G^*(\kappav_G^*(\piv))) \equiv_{\kappav}(H,\alphav_H^*(\kappav_H^*(\tauv)))$ unless $\alphav$ is $\kappav$-stable as defined below.
\begin{definition}
Let $\alphav,\kappav$ be two V-C equivalence relations. 
We say that $\alphav$ is $\kappav$-stable if for all $G,H\in\Graphs$ and $\piv,\tauv\in\Pi$ where $(G,\piv)\equiv_{\kappav^*}(H,\tauv)$,
we have $(G,\alphav_G(\piv))\equiv_{\kappav^*}(H,\alphav_H(\tauv))$.
\end{definition}
Thus, if $\alphav$ is $\kappav$-stable, then 
$(G,\alphav_G^*(\kappav_G^*(\piv))) \equiv_{\kappav^*}(H,\alphav_H^*(\kappav_H^*(\tauv)))$,
and moreover, the two approaches are equivalent meaning that for all $G,H\in\Graphs$ and $\piv,\tauv\in\Piv^k$ where $\piv\approx\tauv$, we have 
$(G,\lambdav_G^*(\piv))\equiv_{\lambdav}(H,\lambdav_H^*(\tauv))$ if and only if
$(G,\kappav_G^*(\piv))\equiv_{\kappav}(H,\kappav_H^*(\tauv))$ and
$(G,\alphav_G^*(\kappav_G^*(\piv))) \equiv_{\alphav}(H,\alphav_H^*(\kappav^*(\tauv)))$,
and furthermore, if $(G,\lambdav_G^*(\piv))\equiv_{\lambdav}(H,\lambdav_H^*(\tauv))$, then
$(\lambdav_G^*(\piv),\lambdav_H^*(\tauv))\simeq (\alphav_G^*(\kappav_G^*(\piv)),\alphav_H^*(\kappav_H^*(\tauv)))$.
We show this as follows:
Let $\piv'=\kappav_G^*(\piv)$ and $\tauv' =\kappav_H^*(\tauv)$.
First, assume $(G,\lambdav_G^*(\piv))\equiv_{\lambdav}(H,\lambdav_H^*(\tauv))$.
Then, $(G,\piv')\equiv_{\kappav}(H,\tauv')$, $(G,\lambdav_G^*(\piv))\equiv_{\alphav^*}(H,\lambdav_H^*(\tauv))$ and
$(\lambdav_G^*(\piv),\lambdav_H^*(\tauv))\leq(\piv',\tauv')$ from above.
Then, Lemma \ref{lem:equiv for all r}, we have 
$(G,\alphav_G^*(\piv')) \equiv_{\alphav}(H,\alphav_H^*(\tauv'))$
and 
$(\lambdav_G^*(\piv),\lambdav_H^*(\tauv))\leq(\alphav_G^*(\piv'),\alphav_H^*(\tauv'))$.
Second, assume $(G,\piv')\equiv_{\kappav}(H,\tauv')$ and
$(G,\alphav_G^*(\piv')) \equiv_{\alphav}(H,\alphav_H^*(\tauv'))$ implying 
$(G,\alphav_G^*(\piv')) \equiv_{\kappav^*}(H,\alphav_H^*(\tauv'))$ since $\alphav$ is $\kappav$-stable.
Thus,
$(G,\alphav_G^*(\piv')) \equiv_{\lambdav^*}(H,\alphav_H^*(\tauv'))$ since $\alphav$ is $\kappav$-stable.
Then, since $(\alphav_G^*(\piv'),\alphav_H^*(\tauv'))\leq (\piv,\tauv)$,
we have $(G,\lambdav_G^*(\piv))\equiv_{\lambdav}(H,\lambdav_H^*(\tauv))$ and 
$(\alphav_G^*(\piv'),\alphav_H^*(\tauv'))\leq(\lambdav_G^*(\piv),\lambdav_H^*(\tauv))$  by Lemma \ref{lem:equiv for all r} as required.

In this paper, in particular, we are interested in V-C preorders that imply $\CPv$ and $\Symv$.
If a V-C preorder $\alphav$ that does not have this property, then we will combine it with $\CPv$ and $\Symv$.
Moreover, the V-C preorders we are interested in are $\CPv$-stable and $\Symv$-stable, so we can apply either of the two approaches above.
In fact, every V-C equivalence relation is $\CPv$-stable, since $(G,\piv)\equiv_{\CPv^*}(H,\tauv)$ if and only if $(\piv,\tauv)\leq(\CP_G^k,\CP_H^k)$ and $\piv\approx\tauv$, and moreover, $(\piv,\tauv)\leq(\CP_G^k,\CP_H^k)$ and $(G,\piv)\equiv_\alphav(H,\tauv)$ imply $(\alpha_G(\piv),\alpha_H(\tauv))\leq(\CP_G^k,\CP_H^k)$ and $\alphav_G(\piv)\approx\alphav_H(\tauv)$ as required.
We will use the first approach for combining $\alphav$ with $\CPv$ and $\Symv$ since we find it leads to a more concise exposition.

\subsection{W-L Algorithm}
\label{sec:WL}

In this section, we rigorously introduce the $k$-dimensional Weisfeiler-Lehman Method where $k>1$.  This algorithm is ubiquitous in literature (see e.g., \cite{Weisfeiler76, WeisfeilerLehman68, CaiFurer1992}).  %This method is only defined for $k>1$.
%Next, we define a vertex classification equivalence relation $\wl$ for the $k$-dim W-L algorithm.
First, we need the following definitions.
\begin{definition}
Given $u\in V^k$ and $v\in V$, we define $\phi_i(u,v) = (u_1,...,u_{i-1},v,u_{i+1},...,u_k).$
\end{definition}
We define the equivalence relation $\wl'$ on $\Graphs\times\Pi^k\times V^k$ where 
$(G,\pi,u) \equiv_{\wl'} (G,\pi,v)$ if $u \equiv_{\pi} v$ and there exists a bijection
$\psi: V\rightarrow V$ such that $\phi_i(u,w)\equiv_{\pi} \phi_i(v,\psi(w))$ for all
$1\leq i \leq k$ and for all $w\in V$.
%We define the function $\wl:\Pi^k\rightarrow\Pi^k$ where $u \equiv_{\wl(\pi)} v$ if $(\pi,u)\equiv_
It is straight-forward to verify that $\wl'$ is a V-C equivalence relation.
We want a V-C equivalence relation that implies $\CP$ and $\Sym$, so we combine $\wl'$ with $\Sym$ and $\CP$ as follows
(indeed, $\wl'$ does not explicitly use the properties of $G$).

\begin{definition}
We define the equivalence relation $\wl$ on $\Graphs\times\Pi^k\times V^k$ where 
$(G,\pi,u) \equiv_\wl (G,\pi,v)$ if $(G,\pi,u) \equiv_{\wl'} (G,\pi,v)$, $(G,\pi,u) \equiv_\CP(G,\pi,v)$ and $(\pi,u) \equiv_\Sym(\pi,v)$
for all $G\in\Graphs$, $\pi\in\Pi^k$ and $u,v\in V^k$.
\end{definition}
%It is straight-forward to verify that $\wl$ is a V-C equivalence relation.
%In the following, we will treat $\wl$ as an equivalence relation on $\Pi^k\times V^k$ and implicitly as a vertex classification equivalence relation on $\Graphs\times\Pi^k\times V^k$.
Then, as in Definition \ref{def:equiv fn}, we define the function $\wl_G:\Pi^k\rightarrow \Pi^k$, and a partition $\pi\in\Pi^k$ is $\wl_G$-stable if $\wl_G(\pi)=\pi$, and we denote by $\wl_G^*(\pi)$ the unique coarsest $\wl$-stable partition that is finer than $\pi$.
Also, we define $\wl_G^k$ as the unique coarsest $\wl_G$-stable partition, that is, $\wl_G^k=\wl_G^*(\{V^k\})$.
%%Since $\wl$ does not explicitly use the properties of $G$,
%%it is crucial that the initial partition of the $k$-dim $\wl$-V-C automorphism algorithm is a refinement of $\SGP{G}^k$.
Given a partition $\pi\in\Pi^k$, the $k$-dim $\wl$-V-C automorphism algorithm is thus to compute $\wl_G^*(\pi)$, which if complete, implies
that $|AUT(G,\pi)|=1$.

The $k$-dim W-L algorithm (see e.g., \cite{CaiFurer1992}) is essentially to compute $\wl'^*_G(\CP_G^k)$.
It is straight-forward to show that $\wl'$ is both $\CP$-stable and $\Sym$-stable (see Definition \ref{def:vc stable}); therefore, $\wl'^*_G(\CP_G^k) = \wl_G^k$
(see Section \ref{sec:vc compare}), and
so the $k$-dim W-L algorithm is equivalent to the $\wl$-V-C automorphism algorithm; however, we consider the more general case of arbitrary starting partitions.

Below, we define an ordered vertex classification relation $\wlv$ for the $k$-dim W-L algorithm.
We define $$\wlv_\piv(u,w) = ([\phi_1(u,w)]_\piv, [\phi_2(u,w)]_\piv, \cdots, [\phi_k(u,w)]_\piv),$$
and we define $\wlv_\piv(u) = (\wlv_\piv(u,w_1),...,\wlv_\piv(u,w_n))$
where $\{w_1,...,w_n\}=V$ and $\wlv_\piv(u,w_i)\leq_{lex}\wlv_\piv(u,w_j)$ for all $i<j$. %, which is an sequence of $k$-tuples of indices of equivalence classes of $\piv$.
Then, we define $\wlv'$ a preorder on $\Graphs\times\Piv^k\times V^k$ where $(G,\piv,u) \leqq_{\wlv'} (H,\tauv,v)$ if 
$([u]_\piv,\wlv_\piv(u)) \leq_{lex} ([v]_\tauv,\wlv_\tauv(v)).$
It is straight-forward to verify that $\wlv$ is a V-C preorder.
In other words, $\wlv_\piv(u) = \wlv_\tauv(v)$ means that $[u]_\piv=[v]_\tauv$ and there exists a bijection 
$\psi: V\rightarrow V$ such that $[\phi_i(u,w)]_\piv = [\phi_i(v,\psi(w))]_\tauv$ for all $1 \leq i \leq k$ and for all $w\in V$.
%We can induce a preorder $\wlv'$ on $\G\times\Piv^k\times V^k$ where $(G,\piv,u) \leqq_{\wlv'} (H,\tauv,v)$ if $(\piv,u) \leqq_\wlv(\tauv,v)$. 
We want a V-C preorder that implies $\CPv$ and $\Symv$, so we combine $\wlv'$ with $\Symv$ and $\CPv$ as follows
\begin{definition}
We define $\wlv$ a preorder on $\Graphs\times\Piv^k\times V^k$ where we have $(G,\piv,u) \leqq_{\wlv} (H,\tauv,v)$ if we have
$([u]_\piv,\Symv_\piv(v),\CPv_\piv(v),\wlv_\piv(u)) \leq_{lex} ([v]_\tauv,\Symv_\tauv(v),\CPv_\tauv(v),\wlv_\tauv(v)).$
\end{definition}
%It is straight-forward to verify that $\wlv$ is a V-C preorder.
%In the following, we will treat $\wlv$ as a preorder on $\Piv^k\times V^k$ and implicity as a vertex classification preorder on $\Graphs\times\Piv^k\times V^k$.
%Again, since $\wlv$ does not use the properties of graphs, 
%it is crucial to the $k$-dim $\wlv$-V-C isomorphism algorithm that the algorithm starts from partitions $\piv,\tauv\in\Piv^k$ such that 
%$(G,\piv)\equiv_{\CPv^*}(H,\tauv)$.
As in Definition \ref{def:preorder fn}, we define the function $\wlv_G:\Piv^k\rightarrow\Piv^k$, and we say that a partition $\piv\in\Piv^k$ is $\wlv_G$-stable
if $\wlv_G(\piv)=\piv$. Also, we define $\wlv_G^*(\piv)$ as the fixed point reach by recursively applying $\wlv_G$, and specifically, we define
$\wlv_G^k=\wlv^*_G((V^k))$.
Also, analogous to Definition \ref{def:preorder equiv}, we define the equivalence relation $\wlv$ on $\Graphs\times\Piv^k$.

Given $\piv,\tauv\in\Piv^k$ where $\piv\approx\tauv$,
the $k$-dim $\wlv$-V-C algorithm is thus to compute $\wlv_G^*(\piv)$ and $\wlv_H^*(\tauv)$ and then to check whether $(G,\wlv_G^*(\piv)) \equiv_\wlv (H,\wlv_H^*(\tauv))$.
The $k$-dim W-L isomorphism algorithm as presented in \cite{CaiFurer1992} is essentially to check whether $(G,\CPv_G^k)\equiv_{\CPv}(H,\CPv_H^k)$,
and then to check whether $(G,\wlv'^*_G(\CPv_G^k)) \equiv_{\wlv'} (H,\wlv'^*_H(\CPv_H^k))$.
Since $\wlv'$ is $\Symv$-stable and $\CPv$-stable, which is straight-forward to verify, we have $(G,\CPv_G^k)\equiv_{\CPv}(H,\CPv_H^k)$,
and $(G,\wlv'^*_G(\CPv_G^k)) \equiv_{\wlv'} (H,\wlv'^*_H(\CPv_H^k))$ if and only if $(G,\wlv_G^*(\piv)) \equiv_\wlv (H,\wlv_H^*(\tauv))$
(see Section \ref{sec:vcv compare}), and so,
the $k$-dim W-L isomorphism algorithm is equivalent to the $k$-dim $\wlv$-V-C algorithm;
however, we consider the more general case of arbitrary starting ordered partitions.

\begin{example}
We demonstrate the $\wlv$-V-C algorithm for $k=2$.
Let $G$ be the graph as in the previous example with vertex set $\{1,2,3,4\}$ and edge set $E(G) = \{ \{1,2\},\{2,3\},\{3,4\},\{2,4\} \}$.
Starting from the trivial ordered partition $\piv^0=(V^2)$, the first iteration computes $\piv^1 = \SGPv{G}^2$, which consists of three sets:
\begin{align*}
\piv^1_1 &= \{(1,1),(2,2),(3,3),(4,4)\}, \\
\piv^1_2 &= \{(1,2),(2,1),(2,3),(3,2),(2,4),(4,2),(3,4),(4,3)\}, \\
\piv^1_3 &= \{(1,3),(3,1),(1,4),(4,1)\}.
\end{align*}
Below we list all $\phi$ values for the pair $u=(1,1)\in V_1^2$.
\begin{align*}
\phi_1(u,1) = (1,1), &\qquad  \phi_1(u,2) = (2,1), & \phi_1(u,3) = (3,1), & \qquad \phi_1(u,4) = (4,1) \\
\phi_2(u,1) = (1,1), & \qquad \phi_2(u,2) = (1,2), & \phi_2(u,3) = (1,3), & \qquad \phi_2(u,4) = (1,4)
\end{align*}
and hence $\wlv_\piv(1,1) = ((1,1),(2,2),(3,3),(3,3))$.
We repeat this procedure for all possible pairs in $V_1^2$ and obtain
\begin{align*}
\wlv_\piv((1,1)) &= ((1,1),(2,2),(3,3),(3,3)), \ & \wlv_\piv((2,2)) &= ((1,1),(2,2),(2,2),(2,2)), \\
\wlv_\piv((3,3)) &= ((1,1),(2,2),(2,2),(3,3)), \ & \wlv_\piv((4,4)) &= ((1,1),(2,2),(2,2),(3,3)).
\end{align*}
and hence $(1,1)$ and $(2,2)$ are in new cells by themselves, while $(3,3),(4,4)$ remain together.
Continuing in this fashion, we obtain the partition $\piv=(\piv_1,..,\piv_{10})$ where
\begin{align*}
\piv_1 &= \{(2,2)\},& \piv_{2} &= \{(3,3),(4,4)\},& \piv_3 &= \{(1,1)\},\\
\piv_4 &= \{(3,4),(4,3)\},& \piv_5 &= \{(3,2),(4,2))\},& \piv_6 &= \{(2,3),(2,4)\}, \\
\piv_7 &= \{(2,1)\},& \piv_{8}^2 &= \{(1,2)\},& \piv_9 &= \{(1,3),(1,4)\},\\
\piv_{10} &= \{(3,1),(4,1)\}.
\end{align*}
This is indeed the unique coarsest $\wlv$-stable partition $\wlv_G^2$.
\end{example}

\subsection{The $\delta$-Vertex Classification Algorithm}
\label{sec:VC}

In this section, we give a detailed and precise description of the $k$-dimensional C-V-C algorithm for the isomorphism and automorphism problem.
\begin{definition}
For $u\in V^k$, we define $\delta^i_G(u) = \{\phi_i(u,w) : w \in \delta_G(u) \}$.
\end{definition}

For graph automorphism, we define the vertex classification equivalence relation $\delta'$ where
$(G,\pi,u) \equiv_{\delta'} (G,\pi,v)$ if and only if $u \equiv_\pi v$ and $|\delta^i_G(u)\cap \pi_s|=|\delta^i_G(v)\cap \pi_s|$ and $|\deltabar^i_G(u)\cap \pi_s|=|\deltabar^i_G(v)\cap \pi_s|$ for all $\pi_s\in\pi$ and $1\leq i \leq k$.
It is straight-forward to show that the relation $\delta$ on $\Graphs\times \Pi^k\times V^k$ is a V-C equivalence relation that is $\CP$-stable and $\Sym$-stable.
We combine $\delta'$ with $\CP$ and $\Sym$ as follows.
\begin{definition}
We define the equivalence relation $\delta$ on $\Graphs\times\Pi^k\times V^k$ where 
$(G,\pi,u) \equiv_\delta(G,\pi,v)$ if $(G,\pi,u) \equiv_{\delta'} (G,\pi,v)$, $(G,\pi,u) \equiv_\CP(G,\pi,v)$ and $(\pi,u) \equiv_\Sym(\pi,v)$
for all $G\in\Graphs$, $\pi\in\Pi^k$ and $u,v\in V^k$.
\end{definition}
Then, as in Definition \ref{def:equiv fn}, we define the function $\delta_G:\Pi^k\rightarrow \Pi^k$, the concept of $\delta_G$-stable, and the partitions $\delta_{G}^*(\pi)$ and $\delta_G^k$.
In the $1$-dim case, for all $\pi\in\Pi$, we trivially have that $\Sym(\pi)=\pi$ and $\CP(\pi)=\pi$, and moreover, $|\delta^i_G(u)\cap \pi_s|=|\delta^i_G(v)\cap \pi_s|$ implies that $|\deltabar^i_G(u)\cap \pi_s|=|\deltabar^i_G(v)\cap \pi_s|$ for all $\pi_s\in\pi$ and $1\leq i \leq k$.
Thus, in this case, the $\delta$-V-C automorphism algorithm is the same as the C-V-C algorithm in the literature where
a $\delta_G$-stable partition is referred to as an \emph{equitable} partition with respect to $G$ (in e.g., \cite{McKay1980}).

We now consider graph isomorphism.
First, we need the following definitions:
For all $\piv=(\piv_1,...,\piv_m)\in\Piv^k$, $1\leq i \leq k$, and $u\in V^k$, we define
$$\deltav^i_{G,\piv}(u) = (|\delta^i_G(u)\cap \piv_1|,...,|\delta^i_G(u)\cap\piv_m|,|\deltabar^i_G(u)\cap \piv_1|,...,|\deltabar^i_G(u)\cap \piv_m|),$$
and also, we define $\deltav_{G,\piv}(u) = (\deltav^1_{G,\piv}(u),...,\deltav^k_{G,\piv}(u)).$
We define the preorder $\deltav'$ where $(G,\piv,u)\leqq_{\deltav'}(H,\tauv,v)$ if 
$([u]_\piv,\deltav_{G,\piv}(u)) \leq_{lex} ([v]_\tauv,\deltav_{H,\tauv}(v))$
for all $G,H\in\Graphs$, $\piv,\tauv\in\Piv^k$ and $u,v\in V^k$.
It is straight-forward to check that $\deltav'$ satisfies the properties of a vertex classification preorder, which is $\CPv$-stable and $\Symv$-stable.
We then combine $\deltav'$ with $\CPv$ and $\Symv$ as follows.
\begin{definition}
We define $\deltav$ a preorder on $\Graphs\times\Piv^k\times V^k$ where we have $(G,\piv,u) \leqq_{\deltav} (H,\tauv,v)$ if we have
$([u]_\piv,\Symv_\piv(v),\CPv_\piv(v),\deltav_{G,\piv}(u)) \leq_{lex} ([v]_\tauv,\Symv_\tauv(v),\CPv_\tauv(v),\deltav_{H,\tauv}(v)).$
\end{definition}
Then, as in Definition \ref{def:preorder fn}, we define the function $\deltav_G:V^k\rightarrow V^k$, the concept of $\deltav_G$-stable, and the partitions 
$\deltav^*_G(\piv)$ and $\deltav_G^k$.
Also, as in Definition \ref{def:preorder equiv}, we define the $\deltav$ equivalence relation on $\Graphs\times \Piv^k$.
%We thus arrive at an algorithm for graph isomorphism; given $\piv,\tauv\in\Piv^k$ where $\piv\approx\tauv$, we have
%$ISO(G,\piv,H,\tauv)=\emptyset$ if $(G,\deltav^*_G(\piv))\not\equiv_{\deltav}(H,\deltav^*_H(\tauv))$.

Again, in the $1$-dim case, the $\deltav$-V-C algorithm simplifies to the C-V-C algorithm for isomorphism.

\subsection{$\Delta$-Vertex Classification Algorithm}
\label{sec:DVC}
In this section, we describe the $k$-dim $\Delta$-V-C algorithm for the automorphism problem
and the $k$-dim $\Deltav$-V-C algorithm for the isomorphism problem.
We will see that the $k$-dim $\Delta$ and $\Deltav$ V-C algorithms are strongly
related to the $k$-dim W-L algorithm and thus useful in proving results about
the $k$-dim W-L algorithm. 

\begin{definition}
For $u\in V^k$, we define $\Delta^i(u) = \{\phi_i(u,w) : w \in V\}$.
\end{definition}

For graph automorphism, we define the equivalence relation $\Delta'$ on $\Graphs\times\Pi^k \times V^k$ where
$(\pi,u) \equiv_{\Delta'} (\pi,v)$ if and only if $u \equiv_\pi v$ and $|\Delta^i(u)\cap \pi_s|=|\Delta^i(v)\cap \pi_s|$ for all $\pi_s\in\pi$ and $1\leq i \leq k$.
It is straight-forward to verify that $\Delta'$ is vertex classification equivalence relation that is $\CP$-stable and $\Sym$-stable.
We combine $\delta'$ with $\CP$ and $\Sym$ as follows.
\begin{definition}
We define the equivalence relation $\Delta$ on $\Graphs\times\Pi^k\times V^k$ where 
$(G,\pi,u) \equiv_\Delta(G,\pi,v)$ if $(G,\pi,u) \equiv_{\Delta'} (G,\pi,v)$, $(G,\pi,u) \equiv_\CP(G,\pi,v)$ and $(\pi,u) \equiv_\Sym(\pi,v)$
for all $G\in\Graphs$, $\pi\in\Pi^k$ and $u,v\in V^k$.
\end{definition}
As in Definition \ref{def:equiv fn}, we define the function $\Delta_G:\Pi^k\rightarrow \Pi^k$, the concept of $\Delta_G$-stable, and the partitions $\Delta_{G}^*(\pi)$ and $\Delta_G^k$.

For graph isomorphism, we define the vertex classification preorder $\Deltav$.
First, we need the following definitions:
For all $\piv=(\piv_1,...,\piv_m)\in\Piv^k$, $1\leq i \leq k$, and $u\in V^k$, we define
$$\Deltav^i_\piv(u) = (|\Delta^i(u)\cap \piv_1|,...,|\Delta^i(u)\cap \piv_m|),$$
and also, we define $\Deltav_{\piv}(u) = (\Deltav^1_{\piv}(u),...,\Deltav^k_{\piv}(u)).$
Then, we can define the preorder $\Deltav'$ on $\Piv^k\times V^k$ where $(\piv,u)\leqq_{\Deltav'}(\tauv,v)$ if
$([u]_\piv,\Deltav_\piv(u)) \leq_{lex} ([v]_\tauv,\Deltav_\tauv(v))$.
It is straight-forward to check that $\Deltav'$ is a V-C preorder that is $\CPv$-stable and $\Symv$-stable.
We next combine it with $\CPv$ and $\Symv$.
\begin{definition}
We define $\Deltav$ a preorder on $\Graphs\times\Piv^k\times V^k$ where we have $(G,\piv,u) \leqq_{\Deltav} (H,\tauv,v)$ if we have
$([u]_\piv,\Symv_\piv(v),\CPv_\piv(v),\Deltav_{\piv}(u)) \leq_{lex} ([v]_\tauv,\Symv_\tauv(v),\CPv_\tauv(v),\Deltav_{\tauv}(v)).$
\end{definition}
Then, we define the function $\Deltav_G:V^k\rightarrow V^k$, the concept of $\Deltav_G$-stable, and the partitions $\Deltav_G^*(\piv)$ and $\Deltav_G^k$ as in Definition \ref{def:preorder fn}.
Also, we define the $\Deltav$ equivalence relation on $\Graphs\times \Piv^k$ as in Definition \ref{def:preorder equiv}.

\subsection{Comparison of Combinatorial Approaches}
\label{sec:VC+WL}
In this section, we show that $\wlv$ implies $\deltav$, which implies $\Deltav$, and subsequently, $\wl$ implies $\delta$, which implies $\Delta$.
\begin{lemma}\label{lem:delta+Delta}
Let $G,H\in\Graphs$, and let $\piv,\tauv\in\Piv^k$ where  $(G,\piv)\equiv_{\deltav^*}(H,\tauv)$.
We have $(G,\piv,u)\equiv_{\deltav}(H,\tauv,v)$ implies $(G,\piv,u)\equiv_\Deltav(H,\tauv,v)$ for all $u,v\in V^k$.
\end{lemma}
\begin{proof}
The result follows from the fact that $\Delta^i(u)=\delta_G^i(u)\cup\deltabar_G^i(u)$.
\end{proof}
\begin{lemma}\label{lem:xi+delta}
Let $G,H\in\Graphs$, and let $\piv,\tauv\in\Piv^k$ where $(G,\piv)\equiv_{\wlv^*}(H,\tauv)$.
We have $(G,\piv,u)\equiv_\wlv(H,\tauv,v)$ implies $(G,\piv,u)\equiv_{\deltav}(H,\tauv,v)$ for all $u,v\in V^k$.
\end{lemma}
\begin{proof}
Let $\piv=(\piv_1,...,\piv_m)$ and $\tauv=(\tauv_1,...,\tauv_m)$, and let $u,v\in V^k$.
Assume $(G,\piv,u)\equiv_\wlv(H,\tauv,v)$, so by assumption, there exists a bijection $\psi: V \rightarrow V$ such that $[\phi_j(u,w)]_\piv = [\phi_j(v,\psi(w))]_\tauv$ for all $1\leq j \leq k$ and for all $w\in V$.
Let $1\leq t\leq m$, $1\leq i \leq k$ and $w\in V$.
We show that $\phi_i(u,w)\in \delta^i_G(u)\cap \piv_t \Leftrightarrow \phi_i(v,\psi(w))\in \delta^i_H(v)\cap \tauv_t$ and that 
$\phi_i(u,w)\in \deltabar^i_G(u)\cap \piv_t \Leftrightarrow \phi_i(v,\psi(w))\in \deltabar^i_H(v)\cap \tauv_t$
implying that $|\delta^i_G(u)\cap \piv_t| = |\delta^i_H(v)\cap \tauv_t|$ and 
$|\deltabar^i_G(u)\cap \piv_t| = |\deltabar^i_H(v)\cap \tauv_t|$
since $\psi$ is a bijection, thereby proving that $(G,\piv,u)\equiv_{\delta}(H,\tauv,v)$.

First, $\phi_i(u,w)\in \piv_t \Leftrightarrow \phi_i(v,\psi(w))\in \tauv_t$ since $[\phi_i(u,w)]_\piv = [\phi_i(v,\psi(w))]_\tauv$ by assumption.
Let $j\ne i$.
Then, since $[\phi_j(u,w)]_\piv = [\phi_j(v,\psi(w))]_\tauv$ and $(G,\piv)\equiv_{\CPv^*}(H,\tauv)$, we have $(G,\phi_j(u,w))\equiv_\CP(H,\phi_j(v,\psi(w)))$.
Thus, $w\in \delta_G(u_i)\Leftrightarrow \psi(w)\in\delta_H(v_i)$ implying 
$\phi_i(u,w)\in \delta^i_G(u)\Leftrightarrow \phi_i(v,\psi(w))\in \delta^i_H(v)$ and
$\phi_i(u,w)\in \deltabar^i_G(u) \Leftrightarrow\phi_i(v,\psi(w))\in \deltabar^i_H(v)$ as required.
\end{proof}
Lemmas \ref{lem:xi+delta} and \ref{lem:delta+Delta} together with Lemma \ref{lem:vc equiv} means that the $\wl$ implies $\delta$, which implies $\Delta$ as stated in Lemma \ref{lem:aut comb comparison} below.

\begin{corollary}\label{lem:aut comb comparison}
Let $G\in\Graphs$, and let $\pi\in\Pi^k$.
We have $\OP{G}^k(\pi) \leq \wl_G^*(\pi) \leq \delta^*_G(\pi) \leq \Delta_G^*(\pi) \leq \pi$.
\end{corollary}

Lemmas \ref{lem:xi+delta} and \ref{lem:delta+Delta} together with Lemma \ref{lem:vcv equiv} mean that the $\wlv$ implies $\deltav$, which implies $\Deltav$ as stated in Lemma \ref{lem:aut comb comparison} below.

\begin{lemma}\label{lem:iso comb comparison}
Let $\piv,\tauv\in\Piv^k$ where $\piv\approx\tauv$.
We have $(G,\wlv_G^*(\piv)) \equiv_\wlv (H,\wlv_H^*(\tauv)) \Rightarrow
(G,\deltav^*_G(\piv)) \equiv_{\deltav} (H,\deltav^*_H(\tauv))
\Rightarrow (G,\Deltav_G^*(\piv)) \equiv_\Deltav (H,\Deltav_H^*(\tauv))$
and $(\wlv_G^*(\piv),\wlv_H^*(\tauv))\leq(\deltav^*_G(\piv), \deltav^*_H(\tauv))\leq (\Deltav_G^*(\piv),\Deltav_H^*(\tauv))$.% \leq(\SGPv{G}^k,\SGPv{H}^k)$.
\end{lemma}

Next, we next show that $k$-dim $\Delta$ implies $(k-1)$-dim $\wl$.
First, we need a way of mapping a $k$-tuple onto an $(k-1)$-tuple and a $(k-1)$-tuple onto a $k$ tuple.
Also, we need a way of mapping a $k$-dimensional partition onto a $(k-1)$-dimensional partition.
\begin{definition}
We define the map $\rho: V^k \rightarrow V^{k-1}$ where $\rho(u)=(u_1,...,u_{k-1})$ for all $u\in V^k$,
and we define the map $\nu:V^{k-1} \rightarrow V^k$ where $\nu(u') = (u'_1,...,u'_{k-1},u'_{k-1})$ for all $u'\in V^{k-1}$.
We define the map $\rho: \Pi^k \rightarrow \Pi^{k-1}$ such that for $\pi \in \Pi^k$ and for $u',v'\in V^{k-1}$, we have $u'\equiv_{\rho(\pi)} v'$ if and only if $\nu(u') \equiv_{\pi} \nu(v')$.
\end{definition}
%Importantly, $\rho(\SGP{G}^k) = \SGP{G}^{k-1}$, and $\Sym(\rho(\pi))=\rho(\pi)$ if $\Sym(\pi)=\pi$, and also, $\pi\leq\tau$ implies $\rho(\pi)\leq\rho(\tau)$.
Next, we establish that $\Delta$ equivalence implies $\wl$ equivalence.
\begin{lemma} \label{lem:rho}
Let $G\in\Graphs$. Let $\pi\in\Pi^k$ where $\Delta_G(\pi)=\pi$, and let $u,v\in V^k$.
If $(G,\pi,u)\equiv_\Delta(G,\pi,v)$, then 
$(G,\rho(\pi),\rho(u))\equiv_\wl(G,\rho(\pi),\rho(v))$.
\end{lemma}
For the proof of this, see Lemma \ref{lem:rhov} and its proof, where we prove the more general result for V-C preorders.
Now, we can show the result that $k$-dim $\Delta$ implies $(k-1)$-dim $\wl$.
\begin{lemma}\label{lem:rho2}
Let $G\in\Graphs$. Let $\pi\in\Pi^k$.
If $\Delta_G(\pi)=\pi$, then $\wl_G(\rho(\pi))= \rho(\pi)$.
Moreover, we have $\rho(\Delta_G^*(\pi)) \leq \wl_G^*(\rho(\pi))$.
\end{lemma}
\begin{proof}
Let $\pi'= \rho(\pi)$.
Assume $\Delta_G(\pi)=\pi$. Let $u',v\in V^{k-1}$ where $u'\equiv_{\pi'} v'$.
Let $u= \nu(u')$ and $v=\nu(v')$. 
Then, $u\equiv_{\pi} v$ by definition of $\rho$, and thus, $(\pi,u)\equiv_{\Delta} (\pi,v)$ since $\Delta_G(\pi)=\pi$.
Then, by Lemma \ref{lem:rho}, we have $(\pi',u') \equiv_{\wl'} (\pi',v')$.
Thus, $\wl_G(\pi') = \pi'$ as required.
Next, the partition $\rho(\Delta_G^*(\pi))$ is $\Delta$-stable and thus $\wl$-stable from above.
Then, since $\rho(\Delta_G^*(\pi))$ is a $\wl$-stable subpartition of $\pi'$ and $\wl^*(\pi')$ is the unique maximal $\wl$-stable subpartition of $\pi'$,
we must have $\rho(\Delta_G^*(\pi)) \leq \wl^*(\pi')$ as required.
\end{proof}
Specifically, we have that $\rho(\Delta_G^k) \leq \wl_G^{k-1}$.
Next, we show that the $k$-dim $\Deltav$ implies $(k-1)$-dim $\wlv$.
First, we need an analogue of $\rho$ for ordered partitions.
\begin{definition}
We define the map $\rhov: \Piv^k \rightarrow \Piv^{k-1}$ such that for $\piv \in \Piv^k$ and for $u',v'\in V^{k-1}$, we have $u'\leqq_{\rhov(\piv)} v'$ if and only if $\nu(u') \leqq_\piv \nu(v')$.
\end{definition}
%Then, by construction $(G,\piv)\equiv_{\CPv^*}(H,\tauv)$ implies $(G,\rhov(\piv))\equiv_{\CPv^*}(H,\rhov(\tauv))$,
%and also, $(G,\piv)\equiv_{\Symv^*}(H,\tauv)$ implies $(G,\rhov(\piv))\equiv_{\Symv^*}(H,\rhov(\tauv))$.
%Before presenting the lemma, we note that if $(G,\piv)\equiv_{\Deltav^*}(H,\tauv)$ and $(G,\piv,u)\equiv_{\Deltav}(H,\tauv,v)$, then we have $[\rho(u)]_\piv=[\rho(v)]_\tauv$ since from Lemma \ref{lem:tech}, we have $[\phi_k(u,u_{k-1})]_\piv = [\phi_k(v,v_{k-1})]_\tauv$, and then, since $(G,\piv)\equiv_{\CPv^*}(H,\tauv)$ and by construction of $\rhov$, we have $[u']_{\piv'} = [v']_{\tauv'}$ since $\rho(\phi_k(u,u_{k-1}))=u'$ and $\rho(\phi_k(v,v_{k-1})) = v'$.
The next lemma is needed to establish that $k$-dim $\Deltav$ implies $(k-1)$-dim $\wlv$.
\begin{lemma}\label{lem:rhov}
Let $G,H\in\Graphs$. Let $\piv,\tauv\in\Piv^k$ where $(G,\piv)\equiv_{\Deltav^*}(H,\tauv)$, and let $u,v\in V^k$.
If $(G,\piv,u)\equiv_\Deltav(H,\tauv,v)$, then $(G,\rhov(\piv),\rho(u))\equiv_\wlv(H,\rhov(\tauv),\rho(v))$.
\end{lemma}
\begin{proof}
Let $m=|\piv|$. Let $\piv' = \rhov(\piv)$, $\tauv'=\rhov(\tauv)$, and let $u'=\rho(u)$ and $v'=\rho(v)$.
Assume $(G,\piv,u)\equiv_\Deltav(H,\tauv,v)$; so, we have $|\Delta^i(u)\cap \piv_t| = |\Delta^i(v)\cap \tauv_t|$ for all $1\leq t\leq m$ and $1\leq i \leq k$.
Then, $[u']_{\piv'} = [v']_{\tauv'}$ by Lemma \ref{lem:rhov props},
and also, $(G,\piv',u')\equiv_{\CPv}(H,\tauv',v')$ and $(\piv',u')\equiv_{\Symv}(\tauv',v')$
since $(G,\piv')\equiv_{\CPv^*}(H,\tauv')$ and $\piv'\equiv_{\Symv^*}\tauv'$ by Lemma \ref{lem:rhov props}.
%Then, since $(G,\piv)\equiv_{\CPv^*}(H,\tauv)$, we have $[u']_{\piv'} = [v']_{\tauv'}$ since $\rho(\phi_k(u,u_{k-1}))=u'$
%and $\rho(\phi_k(v,v_{k-1})) = v'$.
It remains to show that there exists a bijection $\psi:V\rightarrow V$ such that 
$[\phi_i(u',w)]_{\piv'} = [\phi_i(v',\psi(w))]_{\tauv'}$ for all $w\in V$ and $1\leq i\leq k-1$.

Let $U_t=\{w\in V: \phi_k(u,w)\in \piv_t\}$
and let $W_t= \{w\in V: \phi_k(v,w)\in \tauv_t\}$ for all $1\leq t \leq m$.
Note that $\{U_1,...,U_m\}$ and $\{W_1,...,W_m\}$ are partitions of $V$
such that $|\Delta^k(u)\cap \piv_t|=|U_t|=|W_t|=|\Delta^k(v)\cap \tauv_t|$ for all $1\leq  t\leq m$.
Let $\psi:V\rightarrow V$ be a bijection such that $\psi(U_t)=W_t$  for all $1\leq  t\leq m$.

Let $w\in U_t$ and let $w'=\psi(w)$.
Then, by construction,  $[\phi_k(u,w)]_\piv = [\phi_k(v,w')]_\tauv$, and since $(G,\piv)\equiv_{\Deltav^*}(H,\tauv)$,
we also have $(G,\piv, \phi_k(u,w)) \equiv_{\Deltav} (H,\tauv, \phi_k(v,w'))$.
From Lemma \ref{lem:tech}, we have $[\phi_i(\phi_k(u,w),w)]_\piv = [\phi_i(\phi_k(v,w'),w')]_\tauv$
and moreover $(G,\piv,\phi_i(\phi_k(u,w),w)) \equiv_{\Deltav} (H,\tauv,\phi_i(\phi_k(v,w'),w'))$ for all $1\leq i \leq k-1$.
Thus, $[\phi_i(u',w)]_{\piv'} = [\phi_i(v',w')]_{\tauv'}$ for all $1\leq i\leq k-1$
since $\rho(\phi_i(\phi_k(u,w),w))=\phi_i(u',w)$ and $\rho(\phi_i(\phi_k(v,w'),w'))= \phi_i(v',w')$.
\end{proof}
Then, using the lemma above, we arrive at the desired result.
\begin{lemma}\label{lem:Deltav to xiv}
Let $G,H\in\Graphs$. Let $\piv,\tauv\in\Piv^k$ where $\piv\approx\tauv$.
Then,
$(G,\piv) \equiv_{\Deltav^*} (H,\tauv)$ implies $(G,\rhov(\piv)) \equiv_{\wlv^*} (H,\rhov(\tauv))$.
Furthermore, we have $(G,\Deltav_G^*(\piv)) \equiv_\Deltav (H,\Deltav_H^*(\tauv))$ implies
$(G,\wlv^*(\rho(\piv))) \equiv_\wlv (H,\wlv^*(\rho(\tauv)))$ and $(\rhov(\Deltav_G^*(\piv)),\rhov(\Deltav_H^*(\tauv)))  \leq (\wlv_G^*(\rhov(\piv)),\wlv_H^*(\rhov(\tauv)))$.
\end{lemma}
\begin{proof}
Let $\piv' = \rhov(\piv)$ and $\tauv'=\rhov(\tauv)$.
Assume $(G,\piv)\equiv_{\Deltav^*}(H,\tauv)$.
First, from Lemma \ref{lem:rho}, we have $\wlv_G(\piv')=\piv'$ and $\wlv_H(\tauv')=\tauv'$.
Then, there exists a tuple bijection $\gamma:V^k\rightarrow V^k$ such that $(G,\piv,u)\equiv_\Deltav(H,\tauv,\gamma(u))$ for all $u\in V^k$.
Now, we define bijection $\gamma':V^{k-1}\rightarrow V^{k-1}$ where for all $u'\in V^{k-1}$, we have $\gamma'(u') = \rho(\gamma(\nu(u')))$.
Let $u'\in V^{k-1}$ and let $u=\nu(u')$.
Then, we have $(G,\piv,u)\equiv_\Deltav(H,\tauv,\gamma(u))$ implying $(G,\piv',\rho(u))\equiv_\wlv(H,\tauv',\rho(\gamma(u)))$ by Lemma \ref{lem:rhov}
or equivalently $(G,\piv',u')\equiv_\wlv(H,\tauv',\gamma'(u'))$, and thus, $(G,\piv')\equiv_\wlv(H,\tauv')$.

Assume $(G,\Deltav_G^*(\piv))\equiv_{\Deltav}(H,\Deltav_H^*(\tauv))$.
Let $\piv'' = \rhov(\Deltav_G^*(\piv))$ and $\tauv'' = \rho(\Deltav_H^*(\tauv))$.
From Corollary \ref{cor:equiv for all r}, we have $(\Deltav_G^*(\piv),\Deltav_H^*(\tauv))\leq (\piv,\tauv)$ implying 
$(\piv'',\tauv'')\leq(\piv',\tauv')$.
Then, from above, $(G,\piv'')\equiv_{\wlv^*}(H,\tauv'')$,
and by Lemma \ref{lem:equiv for all r}, we have 
$(G,\wlv^*(\piv')) \equiv_\wlv (H,\wlv^*(\tauv'))$ and $(\piv'',\tauv'')  \leq (\wlv_G^*(\piv'),\wlv_H^*(\tauv'))$ as required.
\end{proof}

Next, we show that $k$-dim $\wl$ where $k>1$ implies $(k+1)$-dim $\Delta$.
First, we need a way of mapping a $k$-dimensional partitions to a $(k+1)$-dimensional partition.
\begin{definition}\label{def:nu}
We define the map $\nu: \Pi^k \rightarrow \Pi^{k+1}$ such that for $\pi \in \Pi^k$ and for $u,v\in V^{k+1}$, we have $u\equiv_{\nu(\pi)} v$ if $\rho(u)\equiv_{\pi}\rho(v)$ and $\phi_i(\rho(u),u_{k+1})\equiv_{\pi} \phi_i(\rho(v),v_{k+1})$ for all $1\leq i\leq k$.
\end{definition}
%Crucially (see Corollary \ref{lem:nu props}), if $\Sym(\pi)=\pi$, then $\Sym(\nu(\pi))=\nu(\pi)$, so the map preserves symmetry,
%and also, we have $\nu(\SGP{G}^k) = \SGP{G}^{k+1}$, and furthermore, $\pi\leq\tau$ implies 
%$\nu(\pi)\leq\nu(\tau)$ for all $\tau\in\Pi^k$, and thus, $\pi\leq\SGP{G}^k$ implies $\nu(\pi)\leq\SGP{G}^{k+1}$.

For the proof of the following lemma, see Lemma \ref{lem:nuv} and its proof, where we prove the more general result for V-C preorders.
\begin{lemma} \label{lem:nu}
Let $G\in\Graphs$. Let $k>1$, and let $\pi\in\Pi^k$ where $\wl_G(\pi)=\pi$. 
For all $u,v\in V^{k+1}$, if $u \equiv_{\nu(\pi)} v$ and $(G,\pi,\rho(u))\equiv_\wl(G,\pi,\rho(v))$, then $(G,\nu(\pi),u)\equiv_\Delta (H,\nu(\pi),v)$.
\end{lemma}
The next lemma shows that the $k$-dim $\wl$ implies $(k+1)$-dim $\Delta$.
\begin{lemma} \label{lem:nu2}
Let $k>1$, and let $\pi\in \Pi^k$.
If $\wl_G(\pi)=\pi$, then $\Delta_G(\nu(\pi)) = \nu(\pi)$.
Moreover, $\nu(\wl_G^*(\pi)) \leq \Delta_G^*(\nu(\pi))$.
\end{lemma}
\begin{proof}
Assume $\wl_G(\pi)=\pi$, and let $u,v\in V^{k+1}$. We show $\Delta_G(\nu(\pi)) = \nu(\pi)$.
If $u \equiv_{\nu(\pi)} v$, then $\rho(u)\equiv_{\pi}\rho(v)$ and $(G,\pi,\rho(u))\equiv_\wl(G,\pi,\rho(v))$ since $\wl_G(\pi)=\pi$,
and thus by Lemma \ref{lem:nu}, $(G,\nu(\pi),u)\equiv_\Delta(G,\nu(\pi),v)$ as required.
Also, we have $\nu(\wl_G^*(\pi))$ is a $\Delta_G$-stable subpartition of $\nu(\pi)$ since $\wl_G^*(\pi)$ is $\wl_G$-stable.
Thus, since $\Delta_G^*(\nu(\pi))$ is the unique maximal subpartition of $\nu(\pi)$,
we have $\nu(\wl_G^*(\pi)) \leq \Delta_G^*(\nu(\pi))$.
\end{proof}
Specifically, we have $\nu(\wl_G^k) \leq \Delta_G^{k+1}$ since $\nu(\{V^k\})=\{V^{k+1}\}$.
Next, we establish the equivalence between the $k$-dim $\wl$ and $(k+1)$-dim $\Delta$ preorders.
\begin{corollary}\label{cor:xi=Delta}
Let $k>1$ and let $G\in\Graphs$. Let $\pi\in\Pi^k$.
Then, $\rho(\Delta_G^*(\nu(\pi))) = \wl_G^*(\pi)$.
\end{corollary}
\begin{proof}
By Lemma \ref{lem:nu2}, we have $\nu(\wl^*(\pi)) \leq \Delta_G^*(\nu(\pi))$, and since $\wl^*(\pi)$ is $\Delta$-stable,
we have $\wl^*(\pi) \leq \rho(\Delta_G^*(\nu(\pi))$ since $\rho(\nu(\wl^*(\pi)))=\wl^*(\pi)$ by Corollary \ref{cor:id}.
By lemma \ref{lem:rho2}, we have $\wl^*(\rho(\nu(\pi))) \geq \rho(\Delta_G^*(\nu(\pi))$,
but $\wl^*(\rho(\nu(\pi))) \leq \wl^*(\pi)$ since $\rho(\nu(\pi)) \leq \pi$ from Corollary \ref{cor:id}.
Therefore, $\wl^*(\pi) = \rho(\Delta_G^*(\nu(\pi))$.
\end{proof}
This implies that for all $\pi\in\Pi^k$, we have $\wl_G^*(\pi)$ is a complete partition if and only if  $\Delta_G^*(\nu(\pi))$ is a complete partition,
and so, the algorithms are essentially equivalent.
Specifically, 
$\wl_G^k$ is a complete partition if and only if  $\Delta_G^{k+1}$ is a complete partition,
and moreover, $\rho(\Delta_G^{k+1}) = \wl_G^k$.

Next, analogously, we show the equivalence of the $k$-dim $\Deltav$ and $(k-1)$-dim $\wlv$.
First, we need a map from ordered partitions of $k$-tuples to ordered partitions of $(k+1)$-tuples.
\begin{definition}\label{def:nuv}
Define $\nuv :\Piv^k \rightarrow \Piv^{k+1}$ where for all $\piv\in \Piv^k$ and $u,v \in V^{k+1}$, 
we have $u \leqq_{\nuv(\piv)} v$ if
$\rho(u) \lneqq_\piv \rho(v)$ or $\rho(u) \equiv_\piv \rho(v)$ and $\wlv_\piv(\rho(u),u_{k+1}) \leq_{lex} \wlv_\piv(\rho(v),v_{k+1})$.
\end{definition}
The next lemma is used to establish that $k$-dim $\wlv$ implies $(k+1)$-dim $\Deltav$.
\begin{lemma}\label{lem:nuv}
Let $G,H\in\Graphs$ and let $k>1$.
Let $\piv,\tauv\in\Piv^k$ where $(G,\piv)\equiv_{\wlv^*} (H,\tauv)$.
Let $u,v\in V^{k+1}$. 
If $[u]_{\nuv(\piv)} = [v]_{\nuv(\tauv)}$ and $(G,\piv,\rho(u))\equiv_\wlv(H,\tauv,\rho(v))$, then $(G,\nuv(\piv),u)\equiv_\Deltav(H,\nuv(\tauv),v)$.
\end{lemma}
\begin{proof}
Let $\piv' = \nuv(\piv)$ and $\tauv' = \nuv(\tauv)$. Let $u,v\in V^k$ and let $u'=\rho(u)$ and $v'=\rho(v)$.
Assume $[u]_{\piv'} = [v]_{\tauv'}$ and $(G,\piv,u')\equiv_\wlv(H,\tauv,v')$.
First, from Lemma \ref{lem:nuv props}, we have $(G,\piv')\equiv_{\CPv^*} (H,\tauv')$ and $\piv'\equiv_{\Symv^*}\tauv'$,
and thus, $(G,\piv',u)\equiv_\CPv(H,\tauv',v)$ and $(\piv',u)\equiv_\Symv(\tauv',v)$.
Since $(G,\piv,u')\equiv_\wlv (H,\tauv,v')$, we have $[u']_{\piv'} = [v']_{\tauv'}$ and 
there exists a bijection $\psi:V\rightarrow V$ such that 
$\wlv_\piv(u',w) = \wlv_\tauv(v',\psi(w))$ for all $w\in V$.

Next, we show that $|\Deltav_G^{k+1}(u)\cap \piv'_t|=|\Deltav_H^{k+1}(v)\cap \tauv'_t|$ for all $1\leq t \leq m $ where $m=|\piv'|=|\tauv'|$.
Let $U_t=\{w\in V: \phi_k(u,w)\in \piv'_t\}$ and $W_t= \{w\in V: \phi_k(v,w)\in \tauv'_t\}$.
We will show that $\psi(U_t)=W_t$ implying that $|\Deltav_G^{k+1}(u)\cap \piv'_t|=|U_t|=|W_t|=|\Deltav_H^{k+1}(v)\cap \tauv'_t|$ as required.
Now, for all $w\in U_t$, we have $[u']_\piv = [v']_\tauv$ and $\wlv_\piv(u',w) = \wlv_\tauv(v',\psi(w))$,
which by Lemma \ref{lem:nuv props} implies that $[\phi_{k+1}(u,w)]_{\piv'} = [\phi_{k+1}(v,\psi(w))]_{\tauv'}$, and thus, $\psi(w)\in W_t$.
Thus, $\psi(U_t)\subseteq W_t$, and analogously by symmetry, $\psi(W_t)\subseteq U_t$ as required.

Lastly, since $\piv\equiv_{\Symv^*} \tauv$ and $\piv'\equiv_{\Symv^*} \tauv'$, for all $\sigma\in\Sym_{k+1}$, we have $[\sigma(u)]_{\piv'} = [\sigma(v)]_{\tauv'}$ and 
$(G,\piv',\rho(\sigma(u)))\equiv_\wlv(H,\tauv',\rho(\sigma(v)))$.
Therefore, $|\Delta^i(u)\cap \piv'_t|=|\Delta^i(v)\cap \tauv'_t|$ for all $1\leq i \leq k$ and all $1\leq t\leq m$ as required.
\end{proof}
So, we arrive at the next lemma shows that $k$-dim $\wlv$ implies $(k+1)$-dim $\Deltav$.
\begin{lemma}\label{lem:xiv to Deltav}
Let $G,H\in\Graphs$. Let $k>1$ and let $\piv,\tauv\in\Piv^k$ where $\piv\approx\tauv$.
Then,
$(G,\piv) \equiv_{\wlv^*} (H,\tauv)$ implies $(G,\nuv(\piv)) \equiv_{\Deltav^*} (H,\nuv(\tauv))$, and
$(G,\wlv_G^*(\piv)) \equiv_\wlv (H,\wlv_H^*(\tauv))$ implies
$(G,\Deltav_G^*(\nuv(\piv)) \equiv_\Deltav (H,\Deltav_H^*(\nuv(\tauv)))$ and
$(\nuv(\wlv_G^*(\piv)),\nuv(\wlv_H^*(\tauv)))  \leq (\Deltav_G^*(\nuv(\piv)),\Deltav_H^*(\nuv(\tauv)))$.
\end{lemma}
\begin{proof}
Let $\piv' = \nuv(\piv)$ and $\tauv' = \nuv(\tauv)$.
Assume, $(G,\piv)\equiv_{\wlv^*}(H,\tauv)$.
First, from Lemma \ref{lem:nu2}, we have $\Deltav_G(\piv')=\piv'$ and $\Deltav_H(\tauv')=\tauv'$.
We need to show $(G,\piv')\equiv_{\Deltav}(H,\tauv')$.
Since $(G,\piv)\equiv_{\wlv^*}(H,\tauv)$, there exists a tuple bijection $\gamma':V^k\rightarrow V^k$ such that
$(G,\piv,u')\equiv_\wlv(H,\tauv,\gamma'(u'))$ for all $u'\in V^k$.
First, since $(G,\piv,u')\equiv_\wlv (H,\piv,v')$, there must exist a bijection $\psi:V\rightarrow V$ such that 
$\wlv_\piv(u',w) = \wlv_\tauv(\gamma'(u'),\psi(w))$ for all $w\in V$.
Now, we define the bijection $\gamma:V^{k+1}\rightarrow V^{k+1}$ where for all $u\in V^{k+1}$, we have $\gamma(u) = (\gamma(u')_1,...,\gamma(u')_k,\psi(u_{k+1}))$.
Let $u\in V^{k+1}$ and let $u'=\rho(u)$.
Then, we have $(G,\piv,u')\equiv_\wlv(H,\tauv,\gamma(u'))$ and $\wlv_\piv(u',u_{k+1}) = \wlv_\tauv(\gamma'(u'),\psi(u_{k+1}))$ implying $(G,\piv',u)\equiv_\Deltav(H,\tauv',\gamma(u))$
by Lemma \ref{lem:nuv},
and thus, $(G,\piv')\equiv_{\Deltav}(H,\tauv')$.

Secondly, assume $(G,\wlv_G^*(\piv))\equiv_\wlv(H,\wlv_H^*(\tauv))$.
Let $\piv'' = \nuv(\wlv_G^*(\piv))$ and $\tauv'' = \nuv(\wlv_H^*(\tauv))$.
From above, $(G,\piv'')\equiv_{\Deltav^*}(H,\tauv'')$.
By Corollary \ref{cor:equiv for all r}, $(\wlv_G^*(\piv),\wlv_H^*(\tauv))\leq(\piv,\tauv)$,
and then, from Lemma \ref{lem:nuv props}, $(\piv'',\tauv'')\leq(\piv',\tauv')$.
Therefore, by Lemma \ref{lem:equiv for all r}, $(\piv'',\tauv'')  \leq (\Deltav_G^*(\piv'),\Deltav_H^*(\tauv'))$ and
$(G,\Deltav_G^*(\piv') \equiv_\Deltav (H,\Deltav_H^*(\tauv'))$ as required.
\end{proof}
Combining Lemmas \ref{lem:nuv} and \ref{lem:xiv to Deltav}, we arrive at the following result stating the essential equivalence of the $k$-dim $\wlv$ and $(k+1)$-dim $\Deltav$ preorders.
\begin{corollary} \label{cor:xiv=Deltav}
Let $G,H\in\Graphs$. Let $k>1$ and let $\piv,\tauv\in\Piv^k$ where $\piv\approx\tauv$ and $\nuv(\piv) \approx \nuv(\tauv)$.
We have $(G,\wlv_G^*(\piv)) \equiv_\wlv (H,\wlv_H^*(\tauv))$ if and only if $(G,\Deltav_G^*(\nuv(\piv))) \equiv_\Deltav (H,\Deltav_H^*(\nuv(\tauv)))$.
Moreover, if $(G,\wlv_G^*(\piv)) \equiv_\wlv (H,\wlv_H^*(\tauv))$, then
$(\wlv_G^*(\piv),\wlv_H^*(\tauv)) \simeq (\rhov(\Deltav_G^*(\nuv(\piv))),\rhov(\Deltav_H^*(\nuv(\tauv))))$.
\end{corollary}
\begin{proof}
Let $\piv'=\nuv(\piv)$ and $\tauv'=\nuv(\tauv)$.
Assume $(G,\wlv_G^*(\piv)) \equiv_\wlv (H,\wlv_H^*(\tauv))$.
From Lemma \ref{lem:xiv to Deltav}, we have $(G,\Deltav_G^*(\piv')) \equiv_\Deltav (H,\Deltav_H^*(\tauv'))$.
Also, from Lemma \ref{lem:xiv to Deltav}, we have $(\nuv(\wlv_G^*(\piv)),\nuv(\wlv_H^*(\tauv)))  \leq (\Deltav_G^*(\piv'),\Deltav_H^*(\tauv'))$,
which implies $(\wlv_G^*(\piv),\wlv_H^*(\tauv))  \leq (\rhov(\Deltav_G^*(\piv')),\rhov(\Deltav_H^*(\tauv')))$ 
since $\rhov(\nuv(\wlv_G^*(\piv)))=\wlv_G^*(\piv)$ and 
$\rhov(\nuv(\wlv_H^*(\tauv)))=\wlv_H^*(\tauv)$ from Lemma \ref{lem:idv} since $\wlv_G^*(\piv)$ and $\wlv_G^*(\piv)$ are $\Deltav$-stable. % and since $(\rho(\Deltav_H^*(\tauv)),\rho(\Deltav_H^*(\tauv)))$

Assume $(G,\Deltav_G^*(\piv')) \equiv_\Deltav (H,\Deltav_H^*(\tauv'))$.
Let $\piv''=\rhov(\piv')$ and $\tauv''=\rhov(\tauv')$. 
Then, since $\piv'\approx\tauv'$, we have $(G,\piv') \equiv_\Deltav (H,\tauv')$ by Corollary \ref{cor:equiv for all r}, which implies that $\piv''\approx\tauv''$.
Then, since $\piv\approx\tauv$ and $\piv''\approx\tauv''$, and also, $\piv''\preceq\piv$ and $\tauv''\preceq\tauv$ from Lemma \ref{lem:idv}, we have $(\piv'',\tauv'') \leq (\piv,\tauv)$.
Now, from Lemma \ref{lem:Deltav to xiv}, we have $(G,\wlv_G^*(\piv''))\equiv_\wlv(H,\wlv_H^*(\tauv''))$ and
$(\rhov(\Deltav_G^*(\piv')),\rhov(\Deltav_H^*(\tauv'))) \leq (\wlv_G^*(\piv''),\wlv_H^*(\tauv''))$,
and since $(\wlv_G^*(\piv''),\wlv_H^*(\tauv''))\leq(\piv'',\tauv'') \leq (\piv,\tauv)$ by Corollary \ref{cor:equiv for all r}, 
$(G,\wlv_G^*(\piv))\equiv_\wlv(H,\wlv_H^*(\tauv))$ and 
$(\wlv_G^*(\piv''),\wlv_H^*(\tauv'')) \leq (\wlv_G^*(\piv),\wlv_H^*(\tauv))$ by Lemma \ref{lem:equiv for all r} implying 
$(\rhov(\Deltav_G^*(\piv')),\rhov(\Deltav_H^*(\tauv'))) \leq (\wlv_G^*(\piv),\wlv_H^*(\tauv))$,
as required.
\end{proof}
Specifically, we have $(G,\wlv_G^k) \equiv_\wlv (H,\wlv_H^k)$ if and only if 
$(G,\Deltav_G^{k+1}) \equiv_\Deltav (H,\Deltav_H^{k+1})$.
Moreover, $(G,\wlv_G^k) \equiv_\wlv (H,\wlv_H^k)$ or
$(G,\Deltav_G^{k+1}) \equiv_\Deltav (H,\Deltav_H^{k+1})$ implies
$(\wlv_G^k,\wlv_H^k) \simeq (\rhov(\Deltav_G^{k+1}),\rhov(\Deltav_H^{k+1}))$.

\section{Sherali-Adams relaxations}
\label{sec:SA}
In this section, we give explicit descriptions of the Sherali-Adams relaxations for different graph automorphism and isomorphism polytopes.

\subsection{Birkhoff Polytope}
In this section, we give an explicit description of the Sherali-Adams relaxations of the Birkhoff polytope.
The Birkhoff polytope (for a given positive integer $n$) is the polytope whose integer points are all $n \times n$ permutation matrices:
$$B = \{X \in [0,1]^{n\times n}: Xe=X^Te=e\}.$$
The Birkhoff polytope is precisely $\T_G$ when $G$ is the complete graph $K_n$ or a stable graph.
It is well-known that the Birkhoff polytope is integral (its extreme points are integer),
so the Sherali-Adams relaxations of the Birkhoff polytope are not interesting in themselves. 
However, we use the Birkhoff polytope to illustrate the Sherali-Adams approach explicitly, 
as this serves as the foundation for Sherali-Adams relaxations of polytopes we are
interested in.  We also prove some minor results that are useful later in the paper.  We begin 
with a more explicit formulation of the Birkhoff polytope.  This is given by the inequalities below:
\begin{align*}
\sum_{w\in V} X_{uw} - 1 &= 0 \ \ \  \forall u\in V, \\
\sum_{w\in V} X_{wv} - 1 &= 0 \ \ \  \forall v\in V, \\
0 \le X_{uv}& \le 1 \ \ \  \forall u,v \in V.
\end{align*}

The first step in computing the $k^{th}$ Sherali-Adams relaxation is to multiply each equation above with monomials $\prod_{(r,s)\in I} X_{rs}$ for all $I\subseteq V^2=\{(u,v):u,v\in V\}$ where $|I|\leq k-1$.  
Moreover, we must introduce the inequalities $\prod_{(u,v)\in I} X_{uv}\prod_{(u,v)\in J\setminus I} (1-X_{uv}) \ge 0$ for all $I\subseteq J \subseteq V^2$ where $|J|\leq k$.  We do so as follows:
\begin{align*}
\prod_{(r,s)\in I} X_{rs} \left(\sum_{w\in V} X_{uw}\right)  - \prod_{(r,s)\in I} X_{rs} &= 0
\ \ \ \forall I\subseteq V^2, |I|\le k-1, \forall u \in V, \\
\prod_{(r,s)\in I} X_{rs} \left(\sum_{w\in V} X_{wv}\right) - \prod_{(r,s)\in I} X_{rs} &= 0 
\ \ \  \forall I\subseteq V^2, |I|\le k-1,\forall v\in V, \\
\prod_{(u,v)\in I} X_{uv}\prod_{(u,v)\in J\setminus I} (1-X_{uv}) &\ge 0
\ \ \ \forall I\subseteq J\subseteq V^2, |J|\le k.
\end{align*}

Next, we linearize the above equations to produce the $k^{th}$-degree Sherali Adams extended formulation relaxation,
$\hat{B}^k$.
Here, we have replaced $X_{uv}^2$ with $X_{uv}$ and replaced the monomial $\prod_{(u,v)\in I}X_{uv}$ where $I\subseteq V^2$ with the variable $Y_I$.
\begin{align}
\sum_{w\in V} Y_{I\cup\{(u,w)\}} - Y_I &= 0
\ \ \  \forall I\subseteq V^2, |I| \le k-1, \forall u\in V, \label{equali} \\
\sum_{w\in V} Y_{I\cup\{(w,v)\}} - Y_I &= 0
\ \ \  \forall I\subseteq V^2, |I| \le k-1, \forall v\in V, \label{equalj} \\
\sum_{I \subseteq K \subseteq J} (-1)^{|K\setminus I|} Y_{K} &\ge 0 
\ \ \ \forall I\subseteq J\subseteq V^2, |J|\le k,
\label{ineq}  \\
Y_\emptyset &= 1.
\end{align}

Recall that the integer points in $B$ are the set of $n \times n$ permutation matrices.
The integers points in $\hat{B}^k$ are in bijection with the integer points in $B$.
Indeed, a permutation matrix $P \in \{0,1\}^{n \times n}$ is in one-to-one correspondence
with the point $Y_I \in \hat{B}^k$ if
\[
Y_I = 1 \Leftrightarrow X_{uv} = 1 \ \ \forall \ (u,v) \in I, \ \ \ Y_I = 0 \ \ \mbox{otherwise}.
\]

Many of the inequalities in (\ref{ineq}) above are redundant.
In fact, the only inequalities we need to keep are $Y_I\ge 0$ for all $I\subseteq V^2$.
We show that the other inequality constraints $\sum_{I \subseteq K \subseteq J} (-1)^{|K\setminus I|} Y_{K} \ge 0$  for all 
$I\subseteq J\subseteq V^2, |J|\le k$
are implied by the constraints (\ref{equali}), (\ref{equalj}) and the condition $Y_I\ge 0$, for all $I\subseteq V^2$.
This is established by induction on $|J\setminus I|$. This is trivially true when $|J \setminus I|=0$,
so assume it is true for $|J \setminus I|\leq l$.
Let $(u,v) \in J\setminus I$, with $|J \setminus I| \leq l+1$.
Then,
\begin{align}
 & \ \ \ \ \sum_{I \subseteq K \subseteq J} (-1)^{|K\setminus I|} Y_{K} \notag \\
\label{redundent1}
 & = \sum_{I \subseteq K \subseteq (J\setminus\{(u,v)\})}
\left((-1)^{|K\setminus I|} Y_{K} - (-1)^{|K\setminus I|}  Y_{K\cup \{(u,v)\}}\right) \\
\label{redundent2}
&= \sum_{I \subseteq K \subseteq (J\setminus\{(u,v)\})} \left((-1)^{|K\setminus I|} Y_{K} - 
(-1)^{|K\setminus I|}\left(Y_{K} -
\sum_{w\in V\setminus \{v\}} Y_{K\cup\{(u,w)\}}\right)\right)\\
\label{redundent3}
&= \sum_{I \subseteq K \subseteq (J\setminus\{(u,v)\})} 
\left(\sum_{w\in V\setminus \{v\}} (-1)^{|K\setminus I|} Y_{K\cup\{(u,w)\}}\right)\\
&= \sum_{w\in V\setminus \{v\}}  \left(\sum_{I \subseteq K \subseteq (J\setminus\{(u,v)\})}
(-1)^{|K\setminus I|} Y_{K\cup\{(u,w)\})} \right)\\
\label{redundent4}
&= \sum_{w\in V\setminus \{v\}} \left( \sum_{(I\cup\{(u,w)\}) \subseteq K \subseteq (J\setminus\{(u,v)\})} 
(-1)^{|K\setminus (I\cup\{(u,w)\})|} Y_{K}  \right)\\
&\ge 0.
\end{align}
To progress from (\ref{redundent1}) to (\ref{redundent2}), we used equation
(\ref{equali}).  To show (\ref{redundent4}) is non-negative, we have that $\sum_{\bar{I} \subseteq K \subseteq \bar{J}} 
(-1)^{|K\setminus \bar{I}|} Y_{K}\geq 0$ where $\bar{I}=I\cup\{(u,w)\}$ and $\bar{J}=J\setminus\{(u,v)\}$ for all $w\in V\setminus \{v\}$ by
assumption since $|\bar{J}\setminus\bar{I}|\leq l$.

Thus, we arrive at the following description of $\hat{B}^k$.  
There are still redundant inequalities here, but for the purpose of this paper,
this description suffices.
\begin{align}
\sum_{w\in V} Y_{I\cup\{(u,w)\}} - Y_I &= 0
\ \ \  \forall I\subseteq V^2, |I| \le k-1, \forall u\in V,  \label{e2} \\
\sum_{w\in V} Y_{I\cup\{(w,v)\}} - Y_I &= 0
\ \ \  \forall I\subseteq V^2, |I| \le k-1, \forall v\in V, \label{e3} \\
Y_I &\ge 0 \ \ \ \forall I\subseteq V^2, |I|\le k, \\
Y_\emptyset &= 1.
\end{align}

The following sets of equations, which we will find useful later, are implied by
(\ref{e2}) and (\ref{e3}):
\begin{align}
\sum_{v_1,...,v_s\in V}Y_{\{(u_1,v_1),...,(u_s,v_s)\}} = 1 
\ \ \ \forall u_1,...,u_s\in V, \ \forall s\leq k, \label{eqn1:sums}\\
\sum_{u_1,...,u_s\in V}Y_{\{(u_1,v_1),...,(u_s,v_s)\}} = 1 
\ \ \ \forall v_1,...,v_s\in V, \ \forall s \leq k. \label{eqn2:sums}
\end{align}
This set of equations is equivalent to saying that any given sequence of vertices 
$u_1,...,u_s\in V$ must map onto to one and only one other sequence of vertices 
$v_1,...,v_s\in V$. We can prove this by induction on $s$. This is true for $s=1$. So, 
let us assume that it is true for $s$ and we will prove it is true for $s+1$.
\begin{align*}
\sum_{v_1,...,v_{s+1}\in V}Y_{\{(u_1,v_1),...,(u_{s+1},v_{s+1})\}} &= 
\sum_{v_1,...,v_s\in V}Y_{\{(u_1,v_1),...,(u_s,v_s)\}} = 1
 \ \ \ \forall u_1,...,u_{s+1}\in V,\\
\sum_{u_1,...,v_{s+1}\in V}Y_{\{(u_1,v_1),...,(u_{s+1},v_{s+1})\}} &= 
\sum_{u_1,...,u_s\in V}Y_{\{(u_1,v_1),...,(u_s,v_s)\}} = 1
 \ \ \ \forall v_1,...,v_{s+1}\in V.
\end{align*}
In the above equations, the first equality follows from equations 
(\ref{e2}) and (\ref{e3}) and the second follows by assumption.

Another useful set of equations, which are also implied by (\ref{e2}) and (\ref{e3}) are as follows:
\begin{align}
Y_{I \cup \{(u,v_1),(u,v_2)\}} &= 0 \ \ \ \forall u,v_1,v_2\in V, v_1\ne v_2, I\subseteq V^2, |I|\leq k-2, \label{eqn1:comb}\\
Y_{I \cup \{(u_1,v),(u_2,v)\}} &= 0 \ \ \ \forall u_1,u_2,v\in V, u_1\ne u_2, I\subseteq V^2, |I|\leq k-2. \label{eqn2:comb}
\end{align}
These equations say that no bijection from $V$ to $V$ can map $u$ onto $v_1$ and $u$ onto $v_2$, and similarly, no bijection can map $u_1$ onto $v$ and $u_2$ onto $v$.
We prove this as follows:
Let $u,v\in V$ and $I\subseteq V^2$ where $|I|\leq k-2$.
Then, from equation \ref{e2},  we have
$$\sum_{w\in V} Y_{I\cup\{(u,v)(u,w)\}} - Y_{I\cup\{(u,v)\}} = 
\sum_{w\in V\setminus\{v_1\}} Y_{I\cup\{(u,v)(u,w)\}} = 0.$$
This implies that $ Y_{I\cup\{(u,v)(u,w)\}} = 0$ for all $w\in V$ where $w \ne v$.
Similarly, from equation \ref{e3},  we have
$$\sum_{w\in V} Y_{I\cup\{(u,v)(w,v)\}} - Y_{I\cup\{(u,v)\}} = 
\sum_{w\in V\setminus\{u\}} Y_{I\cup\{(u,v)(w,v)\}} = 0.$$
This implies that $Y_{I\cup\{(u,v)(w,v)\}} = 0$ for all $w\in V$ where $w \ne u$ as required.

In order to relate Sherali-Adams relaxations with combinatorial algorithms, we find it more convenient to use a slightly different notation than above for formulating $\hat{B}^k$. 
We use the following $k$-tuple notation.
\begin{definition}
Given $k$-tuples $u,v\in V^k$, we define $\pair{u,v}=\{(u_1,v_1),...,(u_k,v_k)\}$.
\end{definition}
Using this new notation, we can replace $Y_I$ by $\Y{u,v}$ in the formulation of $\hat{B}^k$ for any $I \subseteq V^2$.
The reformulation of $\hat{B}^k$ with this notation is as follows:
\begin{align}
\sum_{w\in \Delta^i(v)} \Y{u,w} - Y_{\pair{u,v}\setminus\{(u_i,v_i)\}} &= 0
\ \ \  \forall u,v\in V^k, 1\leq i \leq k, \label{eqn2} \\
\sum_{w\in \Delta^i(u)} \Y{w,v} - Y_{\pair{u,v}\setminus\{(u_i,v_i)\}} &= 0
\ \ \  \forall u,v\in V^k, 1\leq i \leq k, \label{eqn3} \\
\Y{u,v} &\ge 0 \ \ \  \forall u,v \in V^k, \label{eqn4}  \\
Y_\emptyset &= 1. \label{eqn5}
\end{align}
There is a lot of redundancy in the above formulation of $\hat{B}^k$ since there are many ways of writing $I$ as $\pair{u,v}$ for some $u,v\in V^k$, but we will find the above description very useful.

We can also reformulate (\ref{eqn1:sums}) and (\ref{eqn2:sums}) using tuple notation:
\begin{align}
\sum_{u\in V^k}\Y{u,v} &= 1 \ \ \ \forall u\in V^k, \label{eqn1:tuple sum}\\
\sum_{v\in V^k}\Y{u,v} &= 1 \ \ \ \forall v\in V^k. \label{eqn2:tuple sum}
\end{align}
This set of equations enforces that any given $k$-tuple of vertices $u\in V^k$ must map onto to 
one and only one other $k$-tuple $v\in V^k$.

We can also reformulate (\ref{eqn1:comb}) and (\ref{eqn2:comb}) concisely using tuple notation as follows:
\begin{align}
\Y{u,v} &= 0 \ \ \ \forall u,v\in V^k, u \not\equiv_\CP v. \label{eqn:tuple comb}
\end{align}
Recall that $u \not\equiv_\CP v$ means that $u_i=u_j$ and $v_i\ne v_j$ for some $1\leq i,j\leq k$.

\subsection{Tinhofer Polytope}
In this section, we give an explicit description of the Sherali-Adams relaxations of the graph isomorphism polytope $\T_{G,H}$
for graphs $G$ and $H$.  As described in the introduction, the polyhedron $\T_{G,H}$ is defined explicitly as follows:
$$ \T_{G,H} = \left\{ X \in B: \sum_{w\in\delta_G(v)} X_{uw} - \sum_{w\in\delta_H(u)}X_{wv} = 0
\ \forall u,v \in V \right\}.$$
There is a nice and intuitive interpretation of the set of the constraints (in addition to $X\in B$) for the Tinhofer polytope.
Let $u,v\in V$ and $X\in B\cap\{0,1\}^{n\times n}$, and let $\psi: V\rightarrow V$ be the corresponding bijection to the permutation matrix $X$.
Then, the expression $\sum_{w\in\delta_G(v)} X_{uw}$ equals 1 if $\psi(u)\in\delta_H(v)$ and equals 0 otherwise.
Similarly, the expression $\sum_{w\in\delta_H(u)}X_{wv}$ equals 1 if $\psi^{-1}(v) \in \delta_G(u)$ and equals 0 otherwise.
Thus, $u$ is adjacent to $\psi^{-1}(v)$ in $G$ if and only if $\psi(u)$ is adjacent to $v$ in $H$, or in other words, edges must map onto edges and non-edges must map onto non-edges.

Recall that the last three sets of equations define the Birkhoff polytope, which we generated the Sherali-Adams relaxation for in the last section.   We therefore determine the Sherali-Adams relaxation of the first set of equations.
First, we multiply the equations with monomials $\prod_{(r,s)\in I} X_{rs}$ for all $I\subseteq V^2=\{(u,v):u,v\in V\}$ where $|I|\leq k-1$.  
\begin{align*}
\prod_{(r,s)\in I} X_{rs}
\left(\sum_{w\in\delta_H(v)}X_{uw}-\sum_{w\in\delta_G(u)}X_{wv}\right)&=0
\ \ \  \forall I\subseteq V^2, |I|\le k-1,\forall u,v\in V.
\end{align*}

Next, we linearize the above equations.
\begin{align}
\sum_{w\in\delta_H(v)} Y_{I\cup\{(u,w)\}} - \sum_{w\in\delta_G(u)}Y_{I\cup\{(w,v)\}} &= 0
\ \ \ \forall I\subseteq V^2, |I|\le k-1, \forall u,v\in V. \label{eqn1:delta non-tuple}
\end{align}

In order to relate the relaxation $\TX^k_{G,H}$ with the $k$-dim V-C algorithm, we find it more convenient to use $k$-tuple notation as in the previous section.
\begin{align}
\sum_{w\in\delta^i_H(v)} \Y{u,w} - \sum_{w\in\delta^i_G(u)}\Y{w,v} &= 0
\ \ \ \forall u,v\in V^k, 1\leq i \leq k. \label{eqn1:delta}
\end{align}

Then, combining the above equations with the equations for the $k$th Sherali-Adams relaxation of the Birkhoff polytope gives the $k^{th}$-degree Sherali Adams extended formulation relaxation of $\T_{G,H}$, which we denote by $\TX^k_{G,H}$.
$$\TX^k_{G,H} = \left\{ Y \in \BX^k: \sum_{w\in\delta^i_H(v)} \Y{u,w} - \sum_{w\in\delta^i_G(u)}\Y{w,v} = 0
\ \forall u,v\in V^k, 1\leq i \leq k\right\}.$$

Recall the integer points in $\T_{G,H}$ are in bijection with isomorphisms from $G$ to $H$.  
In an analogous manner, the integer points in $\TX^k_{G,H}$ are in bijection with isomorphisms from $G$ to $H$ as follows.
A point $Y\in \TX^k_{G,H}$ is integer if and only if there exists an isomorphism $\psi \in ISO(G,H)$ such that for all $u,v\subseteq V^k$, we have $\Y{u,v} = 1$ if $\psi(u_i)=v_i$ for all $1\leq i \leq k$ and $\Y{u,v}=0$ otherwise.

The $k$th Sherali-Adams relaxations of the Tinhofer polytopes $\T_G$ and $\T_{G,H}$ are then the polytopes $\T^k_G\subseteq B$ and $\T^k_{G,H}\subseteq B$ respectively defined as the projection of $\TX^k_G$ and $\TX^k_{G,H}$ onto $B$ respectively.
Lastly, if $\T^k_G=\{\I\}$, then $|AUT(G)|=1$ and if $\T^k_{G,H}= \emptyset$, then $ISO(G,H)=\emptyset$.

\subsection{$\Delta$-Polytope}

We define the semi-algebraic se $Q_{G,H}\subseteq B$ as the set of all $X\in B$ such that
\begin{align}
X_{u_1v_1}X_{u_2v_2} &= 0 \ \ \ \forall \{u_1,u_2\}\in E_G, \{v_1,v_2\}\not\in E_H, \label{eqnWL.1}\\
X_{u_1v_1}X_{u_2v_2} &= 0 \ \ \ \forall \{u_1,u_2\}\not\in E_G, \{v_1,v_2\}\in E_H. \label{eqnWL.2}
\end{align}

Note that $Q_{G,H}\cap\{0,1\}^{n\times n} = \T_{G,H}\cap\{0,1\}^{n\times n}$, that is, the set of permutation matrices in bijection with $ISO(G,H)$.   This is because equations (\ref{eqnWL.1}) and (\ref{eqnWL.2}) enforce that edges must map onto edges and non-edges must map onto non-edges.

We now describe the Sherali-Adams relaxations of $Q_{G,H}$ the first two sets of equations.
The first step in computing the $k$th Sherali-Adams relaxation is introducing:
\begin{align*}
\prod_{(r,s)\in I} X_{rs}
\left(X_{u_1v_1}X_{u_2v_2}\right) &= 0 \ \ \ \forall I\subseteq V^2, |I|\le k-1,  \{u_1,u_2\}\in E_G, \{v_1,v_2\}\not\in E_H, \\
\prod_{(r,s)\in I} X_{rs}
\left(X_{u_1v_1}X_{u_2v_2}\right) &= 0 \ \ \ \forall I\subseteq V^2, |I|\le k-1, \{u_1,u_2\}\not\in E_G, \{v_1,v_2\}\in E_H.
\end{align*}
Next, we linearize the above equations.
\begin{align}
Y_{I\cup\{(u_1,v_1),(u_2,v_2)\}}&= 0 \ \ \ \forall I\subseteq V^2, |I|\le k-1,  \{u_1,u_2\}\in E_G, \{v_1,v_2\}\not\in E_H, \label{eqn1:Delta}\\
Y_{I\cup\{(u_1,v_1),(u_2,v_2)\}}&= 0 \ \ \ \forall I\subseteq V^2, |I|\le k-1, \{u_1,u_2\}\not\in E_G, \{v_1,v_2\}\in E_H. \label{eqn2:Delta}
\end{align}
In tuple notation, these equations are as follows:
\begin{align*}
\Y{u,v} &= 0 \ \ \ \forall u,v\in V^{k+1}, \{u_i,u_j\}\in E_G, \{v_i,v_j\}\not\in E_H \text{ for some } 0\leq i<j\leq k+1,\\
\Y{u,v} &= 0 \ \ \ \forall u,v\in V^{k+1}, \{u_i,u_j\}\not\in E_G, \{v_i,v_j\}\in E_H \text{ for some } 0\leq i<j\leq k+1.
\end{align*}
We can write these equations more concisely as follows:
\begin{align}
\Y{u,v} &= 0 \ \ \ \forall u,v\in V^{k+1}, (G,u) \not\equiv_\CP (H,v). \label{eqn:Delta}
\end{align}
Note that we are including the implied equations \ref{eqn:tuple comb} in the above set of equations since $(G,u) \equiv_\CP (H,v)$ implies $u\equiv_\CP v$.
Then, we arrive at the $k$-th Sherali-Adams extended formulation $\QX^k_{G,H}$ as follows:
$$\QX^k_{G,H} = \{Y \in \BX^k: \Y{u,v} = 0 \ \forall u,v\in V^k, \CPv_G(u) \ne \CPv_H(v) \}.$$
Note that we have omitted the variables $\Y{u,v}$ from $\QX^k_{G,H}$ where $|\pair{u,v}|=k+1$ since $\Y{u,v}=0$ and the variables $\Y{u,v}$ do not appear in the Birkhoff polytope $\BX^k$.

Note that trivially $(G,u) \not\equiv_\CP (H,v)$ for all $u,v\in V$. Thus, $\QX^1_{G,H}=\BX^1$, so the first iteration is not interesting.

The $k$th Sherali-Adams relaxations of the $\Delta$ polytopes $\Q_G$ and $\Q_{G,H}$ are then the polytopes $\Q^k_G\subseteq B$ and $\Q^k_{G,H}\subseteq B$ respectively defined as the projection of $\QX^k_G$ and $\QX^k_{G,H}$ onto $B$ respectively.
Lastly, if $\Q^k_G=\{\I\}$, then $|AUT(G)|=1$ and if $\Q^k_{G,H}= \emptyset$, then $ISO(G,H)=\emptyset$.

\subsection{Comparison of Polyhedra}
In this section, we prove Lemma \ref{lem:VCWL}, which states that $Q^{k+1}_{G,H}\subseteq \T^k_{G,H} \subseteq Q^k_{G,H}$.
First, we show that $\T^k_{G,H} \subseteq Q^k_{G,H}$, which is implied by the next lemma and corollary.
\begin{lemma}
Let $G,H\in\Graphs$, and ket $Y\in \TX^k_{G,H}$. Then, $\Y{u,v} = 0$ for all $u,v\in V^k$ where $(G,u) \not\equiv_\CP (H,v).$
\end{lemma}
\begin{proof}

Let $u_1,u_2,v_1\in V$ such that $\{u_1,u_2\}\not\in E_G$, and let $I\subseteq V^2$ where $|I|\leq k-2$. Then, from equation (\ref{eqn1:delta non-tuple}), we have
$$\sum_{w\in\delta_H(v_1)} Y_{I\cup\{(u_1,v_1),(u_2,w)\}} - 
\sum_{w\in\delta_G(u_2)}Y_{I\cup\{(u_1,v_1),(w,v_1)\}} =
\sum_{w\in\delta_H(v_1)} Y_{I\cup\{(u_1,v_1),(u_2,w)\}} = 0.  $$
The first equality follows since $Y_{I\cup\{(u_1,v_1),(w,v_1)\}} = 0$ for all $w \in\delta_G(u_2)$ from equations (\ref{eqn:tuple comb}) since $u_1 \not \in \delta_G(u_2)$.
Thus, it follows that $Y_{I\cup\{(u_1,v_1),(u_2,w)\}} = 0$ for all $w \in \delta_H(v_1)$.

Similarly, let $u_1,v_1,v_2\in V$ such that $\{v_1,v_2\}\not\in E_G$, and let $I\subseteq V^2$ where $|I|\leq k-2$. Then, from equation (\ref{eqn1:delta non-tuple}), we have
$$\sum_{w\in\delta_H(v_2)} Y_{I\cup\{(u_1,v_1),(u_1,w)\}} - 
\sum_{w\in\delta_G(u_1)}Y_{I\cup\{(u_1,v_1),(w,v_2)\}} =
-\sum_{w\in\delta_G(u_1)}Y_{I\cup\{(u_1,v_1),(w,v_2)\}} = 0.$$
The first equality follows since $Y_{I\cup(u_1,v_1)(u_1,w)} = 0$ for all $w \in\delta_H(v_2)$ from equations (\ref{eqn:tuple comb}) since $v_1 \not \in \delta_G(v_2)$.
Thus, it follows that $Y_{I\cup\{(u_1,v_1),(w,v_2)\}} = 0$ for all $w \in \delta_G(u_1)$ as required.
\end{proof}
The corollary below then follows directly from the definition of $\QX^k_{G,H}$ and $\TX^k_{G,H}$.
\begin{corollary}
Let $G,H\in\Graphs$. We have $\TX^k_{G,H}\subseteq \QX^k_{G,H}$, and $\T^k_{G,H}\subseteq \Q^k_{G,H}$.
\end{corollary}

We now show that $Q^{k+1}_{G,H} \subseteq \T^k_{G,H}$, which is implied by the following lemma.
First, let $\rho:\BX^{k+1} \rightarrow \BX^k$ be the projection of $\BX^{k+1}$ onto $\BX^k$, that is,
for $Y\in \BX^{k+1}$, we define $\rho(Y)\in\BX^k$ where $\rho(Y)_I= Y_I$ for all $|I|\leq k$.
\begin{lemma}
Let $G,H\in\Graphs$. We have $\rho(\QX^{k+1}_{G,H})\subseteq \TX^k_{G,H}$.
\end{lemma}
\begin{proof}

Let $Y\in Q^{k+1}_{G,H}$. Then, for all $u,v\in V$ and $I \subseteq V^2$ where $|I|\leq k-1$, we have
\begin{align*}
\sum_{w\in\delta_H(v)} Y_{I\cup \{(u,w)\}}
& = \sum_{w\in\delta_H(v)} \sum_{w'\in V} Y_{I\cup\{(u,w),(w'v)\}} \\
& = \sum_{w\in\delta_H(v)} \sum_{w'\in \delta_G(u)} Y_{I\cup\{(u,w),(w'v)\}} \\
& = \sum_{w\in V} \sum_{w'\in \delta_G(u)} Y_{I\cup\{(u,w),(w',v)\}}
= \sum_{w'\in \delta_G(u)} Y_{I\cup\{(w'v)\}}.
\end{align*}
The first and last equalities are valid for $\BX^{k+1}_{G,H}$ and follow from equations (\ref{e2}) and (\ref{e3}) respectively.
The second equality is valid for $\QX^{k+1}_{G,H}$ and follows since $Y_{I\cup\{(u,w),(w'v)\}} = 0$ for all 
$w\in\delta_H(v)$ and $w'\in\deltabar_G(u)$ from equation (\ref{eqn2:Delta}),
and similarly, the third equality is valid for $\QX^{k+1}_{G,H}$ and follows since $Y_{I\cup\{(u,w),(w'v)\}} = 0$ for all 
$w\in\deltabar_H(v)$ and $w'\in\delta_G(u)$ from equation (\ref{eqn1:Delta}).
Thus, $\rho(Y)\in \TX^k_{G,H}$ as required.
\end{proof}

\section{Comparison of Combinatorial and Polyhedral approaches}
\label{sec:Comparison}
In this section, we compare the combinatorial algorithm with the polyhedral approach of using Sherali-Adams relaxations.

In order to compare the combinatorial and polyhedral approaches, we introduce \emph{partition polytopes}, 
which are a natural polyhedral analogue of vertex partitions.
Given a partition $\pi\in\Pi^k$ and a polytope $P\subseteq \BX^k$, we define the \emph{partition polytope of $P$} with respect to $\pi$, written $P(\pi)$, as the following polytope:
$$P(\pi) = \{Y \in P :  \Y{u,v} = 0 \ \forall u,v\in V^k, u \not\equiv_{\pi} v\}.$$
Similarly, given ordered partitions $\piv$ and $\tauv$ and a polytope $P\subseteq \BX^k$, we define the \emph{partition polytope of $P$} with respect to $\piv$ and $\tauv$, written $P(\piv,\tauv)$, as the following polytope:
$$P(\piv,\tauv) = \{Y \in P :  \Y{u,v} = 0 \ \forall u,v\in V^k, [u]_\piv\ne[v]_\tauv\}.$$

We call the important partition polytopes $\BX^k(\pi)$ and $\BX^k(\piv,\tauv)$ as the \emph{Birkhoff partition polytopes.}
There is natural bijection between integer points in $\BX^k(\pi)$ set of bijections $\psi: V\rightarrow V$ such that $\psi(\pi_i)=\pi_i$ for all $\pi_i\in\pi$, and similarly, there is a natural bijection between the integer points in 
$\BX^k(\piv,\tauv)$ and set of bijections $\psi: V\rightarrow V$ such that $\psi(\piv)=\tauv$.

Note that we have $\BX^k(\SGP{G}^k) = \QX^k_G$ and if $(G,\SGPv{G}^k)\equiv_\CPv(H,\SGPv{H}^k)$, then $\BX^k(\SGPv{G}^k,\SGPv{H}^k) = \QX^k_{G,H}$ by construction.

Using the concept of partition polytopes, we can give a precise definition of when a polytope can be considered to contain the same combinatorial information as a V-C equivalence relation.
\begin{definition}
Let $G\in\Graphs$.
We say that a given polyhedron $P\subseteq \BX^k$ and a V-C equivalence relation $\alpha$ are \emph{combinatorially equivalent} with respect to $G$ if for all $\pi\in\Pi^k$and for all $u,v\in V^k$, we have $(G,\alpha_G^*(\pi),u) \not\equiv_{\alpha} (G,\alpha_G^*(\pi),v)$ if and only if $\Y{u,v}=0$ for all $Y\in P(\pi)$.
\end{definition}
It follows immediately from the definition that if $P$ and $\alpha$ are combinatorially equivalent, then $P(\pi)=P(\alpha_G^*(\pi))$.
Also, as shown in the lemma below, if $P$ is combinatorially equivalent to $\alpha$, then crucially  $\alphav^*_G(\pi)$ is complete if and only if $P(\pi)=\{\IX^k\}$ or equivalently $\rho^k(P(\pi)) = \{\I\}$ where $\rho^k:\BX^k\rightarrow B$ is a projective map.
Here, $\IX^k$ is the extended identity matrix where for all $u,v\in V^k$, we have $\hat{\mathcal{I}}^k_{n,\pair{u,v}} = 1$ if $u= v$ and 0 otherwise.
\begin{lemma}
Let $P\subseteq \BX^k$ and let $\alpha$ be a V-C equivalence relation.
If $P$ and $\alpha$ are combinatorially equivalent, then for all $\pi\in\Pi^k$,
we have $\alpha_G^*(\pi)$ is complete if and only if $P(\pi) = \{\IX^k\}$ or
equivalently $\rho^k(P(\pi)) = \{\I\}$.
\end{lemma}
\begin{proof}
First, let $Y\in P(\pi)$ such that $\Y{u,v}=0$ if and only if $u\ne v$ for all $u,v\in V^k$.
Then, from equation (\ref{eqn1:tuple sum}), we must have $\Y{u,u}=1$ for all $u\in V^k$, and so, $Y=\IX^k$.
Now, assume that $\alpha_G^*(\pi)$ is complete.
We must have $P(\pi)\neq \emptyset$ since $\Y{u,u}\ne0$ for some $Y\in P(\pi)$ and $u\in V^k$ from the definition of combinatorial equivalence.
Then, for all $Y\in P(\pi)$, we have $\Y{u,v}=0$ if and only if $u\ne v$ for all $u,v\in V^k$, but this implies $Y=\IX^k$ for all $Y\in P(\pi)$ from above, and thus, $P(\pi)=\{\IX^k\}$.
Second, assume $|P(\pi)|>1$.
Then, there must exist a $Y\in P(\pi)$ such that $\Y{u,v}\ne0$ for some $u\ne v$ (otherwise there is only one possible $Y\in P(\pi)$ from above).
Thus,  $(G,\alpha_G^*(\pi),u) \equiv_{\alpha} (G,\alpha_G^*(\pi),v)$, and $\alpha_G^*(\pi)$ is not complete.

Lastly, we show that $P(\pi) = \{\IX^k\}$ if and only if $\rho^k(P(\pi)) = \{\I\}$.
The forwards implication is clear, so we prove the converse.
Assume $\rho^k(P(\pi))=\{\I\}$, and let $Y\in P(\pi)$.
Then, $\rho^k(P(\pi))=\{\I\}$ implies that for all $u,v\in V^k$ where $|\pair{u,v}|=1$,
we have $\Y{u,v} = 1$ if $u=v$ and $\Y{u,v} = 0$ otherwise.
By Lemma \ref{lem:supset}, we have $\Y{u,v}=0$ for all $u,v\in V^k$ where
$|\pair{u,v}|\geq 2$ and $u \ne v$ since there must exist $u',v'\in V^k$ where
$\pair{u',v'}\subseteq \pair{u,v}$, $u'\ne v'$, $|\pair{u',v'}|=1$ and $\Y{u',v'}=0$.
\end{proof}

We now extend the definition of combinatorial equivalence to V-C preorders.

\begin{definition}
\label{defn: comb equiv}
Let $G,H\in\Graphs$.
We say that a given polyhedron $P\subseteq \BX^k$ and a V-C preorder $\alphav$ are \emph{combinatorially equivalent} with respect to $(G,H)$ if for all $\piv,\tauv\in\Piv^k$ and for all $u,v\in V^k$, we have %where $\piv\equiv_{\Symv}\tauv$
\begin{enumerate}
\item if $(G,\piv,u) \not\equiv_{\alphav} (H,\tauv,v)$, then $\Y{u,v}=0$ for all $Y\in P(\piv,\tauv)$; and
\item if $(G,\piv)\equiv_{\alphav^*}(H,\tauv)$, then $P(\piv,\tauv)\ne\emptyset$, and moreover,
if $(G,\piv,u) \equiv_{\alphav} (H,\tauv,v)$, then there exists $Y\in P(\piv,\tauv)$ such that $\Y{u,v}\ne0$.
\end{enumerate}
\end{definition}
We show below that if $P$ and $\alphav$ are combinatorially equivalent, then $(G,\piv)\not\equiv_{\alphav}(H,\piv)$ implies $P(\piv,\tauv) = \emptyset$,
and moreover, if $\piv \approx \tauv$, then $P(\piv,\tauv) = P(\alphav_G^*(\piv), \alphav_G^*(\piv))$.
Thus, it follows immediately that if $\piv \approx \tauv$, then $P(\piv,\tauv) = \emptyset$ if and only if $(G,\alphav_G^*(\piv))\not\equiv_{\alphav}(H,\alphav_G^*(\piv))$
and $(G,\alphav_G^*(\piv))\equiv_{\alphav}(H,\alphav_H^*(\tauv))$ implies 
$(G,\alphav_G^*(\piv),u) \not\equiv_{\alphav} (H,\alphav_H^*(\tauv),v)$ if and only if $\Y{u,v}=0$ for all $Y\in P(\piv,\tauv)$ and $u,v\in V^k$.
So, if $P$ and $\alphav$ are combinatorially equivalent, then the polytope and the partition contain the same combinatorial information in some sense.
Also, if $\alphav$ and $P$ are combinatorially equivalent, then it follows that $\alpha$ (the induced V-C equivalence relation) and $P$ are also combinatorially equivalent.

We now show some useful properties of combinatorial equivalence.
First, we show that if $P$ and $\alphav$ are combinatorially equivalent, then $(G,\piv)\not\equiv_{\alphav}(H,\piv)$ implies $P(\piv,\tauv) = \emptyset$.
\begin{lemma}\label{lem:partition poly empty}
Let $\piv,\tauv\in\Piv^k$ such that for all $u,v\in V^k$, if $(G,\piv,u) \not\equiv_{\alphav} (H,\tauv,v)$, then $\Y{u,v}=0$ for all $Y\in P(\piv,\tauv)$.
If $(G,\piv)\not\equiv_{\alphav} (H,\tauv)$, then $P(\piv,\tauv) = \emptyset$.
\end{lemma}
\begin{proof}
Since $(G,\piv)\not\equiv_{\alphav}(H,\tauv)$, there exists $\hat{u}\in V^k$ such that $|U|\neq |W|$ where
$$ U=\{w\in V^k: (G,\piv,w) \equiv_{\alphav}(G,\piv,\hat{u})\}\text{ and } W=\{w\in V^k: (H,\tauv,w) \equiv_{\alphav}(G,\piv,\hat{u})\}.$$
Note that $(G,\piv,u)\equiv_{\alphav}(H,\tauv,v)$ for all $u\in U$ and $v \in W$.
Now, from (\ref{eqn1:tuple sum}), for all $u\in U$, we have $\sum_{v\in V^k} \Y{u,v} = 1$ implying $\sum_{v\in W} \Y{u,v} = 1$
since for all $v\not\in W$, we have $(G,\piv,u)\not\equiv_{\alphav}(H,\tauv,v)$ implying that $\Y{u,v}=0$.
Similarly, from (\ref{eqn2:tuple sum}), for all $v\in W$, we have that $\sum_{u\in V^k} \Y{u,v} = 1$ implying $\sum_{u\in U} \Y{u,v} = 1$.
Then,
\begin{align*}
|U| = \sum_{u\in U}1 =
\sum_{u\in U}\left(\sum_{v\in W} \Y{u,v}\right)=
\sum_{v\in W}\left(\sum_{u\in U} \Y{u,v}\right)=
\sum_{v\in W} 1 = |W|.
\end{align*}
Thus, $|U|=|W|$, which is a contradiction.
\end{proof}

The next two lemmas show that if $P$ and $\alphav$ are combinatorially equivalent and $\piv \approx \tauv$, then $P(\piv,\tauv) = P(\alphav_G^*(\piv), \alphav_G^*(\piv))$.
\begin{lemma}\label{lem:partition poly equiv}
Let $P\subseteq \BX^k$ and let $\piv,\tauv\in\Piv^k$ 
such that for all $u,v\in V^k$, if $(G,\piv,u) \not\equiv_{\alphav} (H,\tauv,v)$, then $\Y{u,v}=0$ for all $Y\in P(\piv,\tauv)$.
If $(G,\piv)\equiv_{\alphav} (H,\tauv)$, then $P(\piv,\tauv) = P(\alphav_G(\piv),\alphav_H(\tauv))$.
\end{lemma}
\begin{proof}
Firstly, recall that $(G,\piv)\equiv_{\alphav} (H,\tauv)$ implies that $(\piv,\tauv) \geq (\alphav_G(\piv),\alphav_H(\tauv))$
meaning that $[u]_\piv \ne [v]_\tauv \Rightarrow [u]_{\alphav_G(\piv)}\ne[v]_{\alphav_H(\tauv)}$,
and thus, $P(\piv,\tauv) \supseteq P(\alphav_G(\piv),\alphav_H(\tauv))$.
Secondly,  $(G,\piv)\equiv_{\alphav} (H,\tauv)$ means that
$[u]_{\alphav_G(\piv)}\neq[v]_{\alphav_H(\tauv)} \Leftrightarrow (G,\piv,u)\not\equiv_{\alphav}(G,\tauv,v)$,
and thus, $P(\piv,\tauv) \subseteq P(\alphav_G(\piv),\alphav_H(\tauv))$ as required.
\end{proof}

\begin{lemma}
Let $P\subseteq \BX^k$ and let $\piv,\tauv\in\Piv^k$  where $\piv\approx \tauv$
such that for all $u,v\in V^k$, if $(G,\piv,u) \not\equiv_{\alphav} (H,\tauv,v)$, then $\Y{u,v}=0$ for all $Y\in P(\piv,\tauv)$.
We have $P(\piv,\tauv) = P(\alphav_G^*(\piv), \alphav_H^*(\tauv))$.
\end{lemma}
\begin{proof}
First, assume $(G,\alphav_G^*(\piv)) \equiv_{\alphav} (H,\alphav_H^*(\tauv))$.
Then, by Corollary \ref{cor:equiv for all r}, we have $(G,\alphav_G^r(\piv)) \equiv_{\alphav} (H,\alphav_H^r(\tauv))$ for all $r$ 
implying that $P(\alphav_G^r(\piv),\alphav_H^r(\tauv)) = P(\alphav_G^{r+1}(\piv),\alphav_H^{r+1}(\tauv))$ for all $r$
by Lemma \ref{lem:partition poly equiv} above,
and therefore, $P(\piv,\tauv) = P(\alphav_G^*(\piv), \alphav_H^*(\tauv))$.

Second, assume $(G,\alphav_G^*(\piv)) \not\equiv_{\alphav} (H,\alphav_H^*(\tauv))$.
Then, by Lemma \ref{lem:partition poly empty}, we have $P(\alphav_G^*(\piv), \alphav_H^*(\tauv))= \emptyset$.
Now, there exists an $r'$ such that $(G,\alphav_G^r(\piv)) \equiv_{\alphav} (H,\alphav_H^r(\tauv))$ for all $r<r'$, and $(G,\alphav_G^r(\piv)) \not\equiv_{\alphav} (H,\alphav_H^r(\tauv))$.
Thus, from Lemma \ref{lem:partition poly equiv}, we have $P(\piv,\tauv)=P(\alphav_G^r(\piv),\alphav_H^r(\tauv))$, but 
$P(\alphav_G^r(\piv),\alphav_H^r(\tauv))=\emptyset$ from Lemma \ref{lem:partition poly empty},
and therefore, $P(\piv,\tauv) = P(\alphav_G^*(\piv), \alphav_H^*(\tauv))$.
\end{proof}

In the remainder of this section, we will show the combinatorial equivalence of the V-C preorders and polyhedra we have seen previously.

\subsection{The $\Delta$-V-C algorithm and the $\Delta$-polytope}
In this section, we prove the result that
$\Delta$ is combinatorially equivalent to $\QX^k_G$, % that is, $\Delta_G^*(\pi)$ is complete if and only if $\QX^k_G(\pi)=\{\IX^k\}$ for all $\pi\in\Pi^k$ where $\pi\leq\SGP{G}^k$,
and also that
$\Deltav$ is combinatorially equivalent to $\QX^k_{G,H}$. 
Specifically, we have that $\Delta_G^k$ is complete if and only if $\QX^{k+1}_G=\{\IX^k\}$ or equivalently 
$\Q^{k+1}_G=\{\I\}$, we have $(G,\Deltav_G^k)\equiv_\Deltav(H,\Deltav_H^k)$ if and only if $\QX^k_{G,H}\ne\emptyset$
or equivalently $\Q^k_{G,H}\ne\emptyset$.

 We now show that $\Deltav$ and $\QX^k$ satisfy the two conditions of Definition \ref{defn: comb equiv}.
\begin{lemma}
\label{lem:DeltaPoly}
Let $\piv,\tauv\in\Piv^k$.
For all $u,v\in V^k$, if $(G,\piv,u) \not\equiv_\Deltav (H,\tauv,v)$, then $\Y{u,v}=0$ for all $Y\in \QX^k(\piv,\tauv)$.
\end{lemma}
\begin{proof} 
First, if $(G,\piv,u) \not\equiv_{\CPv} (H,\tauv,v)$ or $(G,\piv,u) \not\equiv_{\Symv} (H,\tauv,v)$, then clearly $\Y{u,v}=0$ for all $Y\in \TX^k_{G,H}(\piv,\tauv)$,
so we assume that $(G,\piv,u) \equiv_{\CPv} (H,\tauv,v)$ and $(G,\piv,u) \equiv_{\Symv} (H,\tauv,v)$.
Then, $(\piv,u) \not\equiv_\Delta (\tauv,v)$ means that $|\Delta^i(u)\cap \piv_t| \ne |\Delta^i(v)\cap \tauv_t|$ for some $1\leq t \leq m$ and $1\leq i \leq k$.

First, we derive some valid equations for $\QX^k_{G,H}(\piv,\tauv)$.
Now, for all $1\leq s,t\leq m$, $1\leq i\leq k$ and $u\in \piv_s$, we have the following valid equation
derived immediately from equations (\ref{eqn2}) and (\ref{eqn3}):
\begin{align}
\sum_{v\in \tauv_t}\left(\sum_{w\in \Delta^i(v)} \Y{u,w} - \sum_{w\in \Delta^i(u)} \Y{w,v}\right) &= 0 \\
\Rightarrow
\sum_{v\in \tauv_t}\sum_{w\in \Delta^i(v)\cap \tauv_s} \Y{u,w} - \sum_{v\in \tauv_t}\sum_{w\in \Delta^i(u)\cap \piv_t} \Y{w,v} &= 0  \\
\Rightarrow
\sum_{w\in \tauv_s} \sum_{v\in \Delta^i(w)\cap \tauv_t} \Y{u,w} - \sum_{w\in \Delta^i(u)\cap \piv_t}\sum_{v\in \tauv_t} \Y{w,v} &= 0 \\
\Rightarrow
\sum_{w\in \tauv_s}|\Delta^i(w)\cap \tauv_t| \Y{u,w} - \sum_{w\in \Delta^i(u)\cap \piv_t}1 &= 0 \\
\Rightarrow
\sum_{w\in \tauv_s}|\Delta^i(w)\cap \tauv_t| \Y{u,w} - |\Delta^i(u)\cap \piv_t| &= 0 \\
\Rightarrow
\sum_{w\in \tauv_s}|\Delta^i(w)\cap \tauv_t| \Y{u,w} - |\Delta^i(u)\cap \piv_t|\sum_{w\in \tauv_s} \Y{u,w} &= 0 \\
\Rightarrow
\sum_{w\in \tauv_s}\left(|\Delta^i(w)\cap \tauv_t|-|\Delta^i(u)\cap \piv_t|\right)\Y{u,w}&=0.
\label{eqn:DeltaPoly1}
\end{align}

Similarly, swapping the role of $\piv$ and $\tauv$, for all $1\leq s,t\leq m$, $1\leq i\leq k$ and $v\in \tauv_s$, we have
the following valid equation:
\begin{align}
\sum_{u\in \piv_t}\left(\sum_{w\in \Delta^i(v)} \Y{u,w} - \sum_{w\in \Delta^i(u)} \Y{w,v}\right) &= 0 \\
\Rightarrow
\sum_{w\in \piv_s}\left(|\Delta^i(w)\cap \piv_t|-|\Delta^i(v)\cap \tauv_t|\right)\Y{w,v}&=0.
\label{eqn:DeltaPoly2}
\end{align}

Using the above equations (\ref{eqn:DeltaPoly1}) and (\ref{eqn:DeltaPoly2}) and the inequalities 
that $\Y{u,v}\ge0$ for all $u,v\in V^k$,
we can imply that $\Y{u,v}=0$ for all $u\in \piv_s$ and
$v\in \tauv_s$ where $|\Delta^i(u)\cap \piv_t|\ne|\Delta^i(v)\cap \tauv_t|$
for some $1\leq t\leq m$ and $1\leq i \leq k$.
This requires some further argument as follows.

We will proceed by induction.
Let $M\in \N^n$.
Assume that $\Y{u,v}=0$ for all $u,v\in V^k$ such
that $|\Delta^i(u)\cap \piv_t|\ne|\Delta^i(v)\cap \tauv_t|$
and $|\Delta^i(u)\cap \piv_t|, |\Delta^i(v)\cap \tauv_t| \le M$.
This is true for $M=0$, so let us assume it is true for $M$, and
we show it is true $M+1$.

Let $1 \leq s\leq m$, $u\in \piv_s$ and $v\in \tauv_s$ such that $|\Delta^i(u)\cap \piv_t|\ne|\Delta^i(v)\cap \tauv_t|$
and $|\Delta^i(u)\cap \piv_t|, |\Delta^i(v)\cap \tauv_t|\le M+1$.
First, assume that $|\Delta^i(u)\cap \piv_t| < |\Delta^i(v)\cap \tauv_t|$.
Then, by equation (\ref{eqn:DeltaPoly1}), we have 
$$\sum_{w\in \tauv_s}\left(|\Delta^i(w)\cap \tauv_t|-|\Delta^i(u)\cap \piv_t|\right)\Y{u,w}=0.$$
Now, by assumption, $\Y{u,w}=0$ for all $w\in \tauv_s$ where $|\Delta^i(w)\cap \tauv_t| < |\Delta^i(u)\cap \piv_t|$;
thus, after eliminating those variables, the equation above becomes a sum of non-negative variables with non-negative coefficients implying
that each variable with a positive coefficient is 0, and in particular, $\Y{u,v}=0$.
Second, assume that $|\Delta^i(v)\cap \tauv_t|<|\Delta^i(u)\cap \piv_t|$.
Then, similarly, using equation (\ref{eqn:DeltaPoly2}), we have  $\Y{u,v}=0$.
\end{proof}

Next, we show that using the partitions $\piv$ and $\tauv$ where $(G,\piv)\equiv_{\Deltav^*} (H,\tauv)$, then we can construct a feasible solution $Y\in \QX^k_{G,H}(\piv,\tauv)$ such that $\Y{u,v}>0$ if and only if $(G,\piv,u)\equiv_\Deltav (H,\tauv,v)$.
\begin{lemma}
\label{lem:Q non-empty}
Let $\piv,\tauv\in\Piv^k$ where $(G,\piv) \equiv_{\Deltav^*} (H,\tauv)$.
Then, $\QX^k_{G,H}(\piv,\tauv)\ne\emptyset$, and more specifically, $Y\in \QX^k_{G,H}(\piv,\tauv)$ where
$Y_\emptyset = 1$ and for all $u,v\in V^k$,
$\Y{u,v} = |\piv_s|^{-1}=|\tauv_s|^{-1}$ if $[u]_\piv=[v]_\tauv=s$ and $\Y{u,v}=0$ otherwise.
\end{lemma}
\begin{proof}
First, we need to show that $Y$ is well-defined meaning that for all
$u\in \piv_s,v \in \tauv_{t},u'\in \piv_{s'},v'\in \tauv_{t'}$
where $\pair{u,v}=\pair{u',v'}$, 
if $s=t$, then $s'=t'$ and $|\piv_s|=|\piv_{s'}|=|\tauv_t|=|\tauv_{t'}|$,
otherwise if $s\ne t$, then $s'\ne t'$.
See Corollary \ref{cor:set sizes} for a proof of this.

Clearly, by definition, $Y$ satisfies the constraints (\ref{eqn:Delta}), (\ref{eqn4}) and (\ref{eqn5}),
so we must show that $Y$ satisfies equations (\ref{eqn2}) and (\ref{eqn3}).
Consider the two $k$-tuples $u\in \piv_s$ and $v \in \tauv_t$.
Then,
$$\sum_{w\in \Delta^i(v)} \Y{u,w} =
\sum_{w\in \Delta^i(v)\cap \tauv_s} \Y{u,w} = 
\sum_{w\in \Delta^i(v)\cap \tauv_s} |\tauv_s|^{-1}
=|\Delta^i(v)\cap \tauv_s||\tauv_s|^{-1}
$$
and
$$\sum_{w\in \Delta^i(u)} \Y{w,v} =
\sum_{w\in \Delta^i(u)\cap \piv_t} \Y{w,v} = 
\sum_{w\in \Delta^i(u)\cap \piv_t} |\piv_t|^{-1}
=|\Delta^i(u)\cap \piv_t||\piv_t|^{-1}.
$$
Let $u'=\phi_i(u,u_j)=(u_1,...,u_{i-1},u_j,u_{i+1},...,u_k)$ and let
$v'= \phi_i(v,v_j)$ $ =(v_1,...,v_{i-1},v_j,v_{i+1},...,v_k)$, and let $1 \leq s',t'\leq m$ where $u'\in \piv_{s'}$
and $v'\in \tauv_{t'}$.  Then, $Y_{\pair{u,v}\setminus\{(u_i,v_i)\}} = \Y{u',v'}$,
and $\Y{u',v'}=|\tauv_{t'}|^{-1}=|\piv_{s'}|^{-1}$ if $s'=t'$ and
$\Y{u',v'}=0$ otherwise.
Thus, in order to show that equations (\ref{eqn2}) and (\ref{eqn3}) are satisfied,
we must show that, if $s'=t'$, then
$|\Delta^i(u)\cap \piv_t||\piv_t|^{-1} = |\piv_{s'}|^{-1}$ and
$|\Delta^i(v)\cap \tauv_s||\tauv_s|^{-1}= |\tauv_{s'}|^{-1}$,
and if $s'\ne t'$, then
$|\Delta^i(u)\cap \piv_t|=0$ and $|\Delta^i(v)\cap \tauv_s|=0$.

First,  assume that $s'=t'$.  By Lemma \ref{lem:mapping}, we have $|\piv_s| = |\Delta^i(u) \cap \piv_s||\piv_{s'}|$
and $|\piv_t| =|\Delta^i(u) \cap \piv_t||\piv_{t'}|$ as required.

Next, assume $s'\ne t'$. %, and so, $\Y{u',v'}=0$. 
First, let $w\in\Delta^i(v)$. 
Then, $|\Delta^i(u)\cap \piv_{s'}|=|\{u'\}|=1$,
but $|\Delta^i(w)\cap \tauv_{s'}| = 0$ since $\tauv_{s'}$ must have the same
combinatorial type as $\piv_{s'}$ meaning that $u_{i-1}=u_i$ for all $u\in \tauv_{s'}$
and the only tuple in $\Delta^i(w)=\Delta^i(v)$ of that type is $v'$, but $v'\in \tauv_{t'}$.
So, $[u]_\piv \ne [w]_\tauv$, and thus, $|\Delta^i(v)\cap \piv_t|=0$.
Similarly, let $w'\in\Delta^i(u)$. 
Then, $|\Delta^i(v)\cap \tauv_{t'}|=|\{v'\}|=1$,
but $|\Delta^i(w')\cap \piv_{s'}|=0$, and therefore, $[w']_\piv \ne [v]_\tauv$, and thus, $|\Delta^i(u)\cap \tauv_s|=0$
as required.

\end{proof}

%As an aside, in the special case that $G$ and $H$ are complete graphs, then $\QX^k_G=\BX^k$ and $\QX^k_{G,H}=\BX^k$,
%so $\Delta$ and $\Deltav$ are also combinatorially equivalent to $\BX^k$.

\subsection{The $\delta$-V-C Algorithm and the Tinhofer Polytope}
\label{sec:VC+SA}
In this section, we prove Theorem \ref{thm:SheraliAdamsVC} by showing that $\delta$ is combinatorially equivalent to $\TX^k_G$,
and also, $\deltav$ is combinatorially equivalent to $\TX^k_{G,H}$.
Specifically, we have that $\delta_G^k$ is complete if and only if $\TX^{k+1}_G=\{\IX^k\}$ or equivalently $\T^{k+1}_G=\{\I\}$,
we have $(G,\deltav_G^k)\equiv_\deltav(H,\deltav_H^k)$ if and only if $\TX^k_{G,H}\ne\emptyset$
or equivalently $\T^k_{G,H}\ne\emptyset$ as required for Theorem \ref{thm:SheraliAdamsVC}.
We now show that $\deltav$ and $\TX^k_{G,H}$ satisfy that two conditions of Definition \ref{defn: comb equiv}.

\begin{lemma}
\label{lem:deltaPoly}
Let $\piv,\tauv\in\Piv^k$.
For all $u,v\in V^k$, if $(G,\piv,u) \not\equiv_{\deltav} (H,\tauv,v)$, then $\Y{u,v}=0$ for all $Y\in \TX^k_{G,H}(\piv,\tauv)$.
\end{lemma}
\begin{proof} 
The proof of this lemma is analogous to the proof of Lemma \ref{lem:deltaPoly}.
First, if $(G,\piv,u) \not\equiv_{\CPv} (H,\tauv,v)$ or $(G,\piv,u) \not\equiv_{\Symv} (H,\tauv,v)$, then clearly $\Y{u,v}=0$ for all $Y\in \TX^k_{G,H}(\piv,\tauv)$,
so we assume that $(G,\piv,u) \equiv_{\CPv} (H,\tauv,v)$ and $(G,\piv,u) \equiv_{\Symv} (H,\tauv,v)$.
Then, $(G,\piv,u) \not\equiv_{\delta} (H,\tauv,v)$ means that
$|\delta_G^i(u)\cap \piv_t| \ne |\delta_H^i(v)\cap \tauv_t|$ or $|\deltabar_G^i(u)\cap \piv_t| \ne |\deltabar_H^i(v)\cap \tauv_t|$ for some $1\leq t \leq m$ and $1\leq i \leq k$.

First, we derive some valid equations for $P(\piv,\tauv)$.
Now, for all $1\leq s,t\leq m$, $1\leq i\leq k$ and $u\in \piv_s$, we have the following valid equation:
\begin{align}
\sum_{v\in \tauv_t}\left(\sum_{w\in \delta_H^i(v)} \Y{u,w} - \sum_{w\in \delta_G^i(u)} \Y{w,v}\right) &= 0 \\
\Rightarrow
\sum_{v\in \tauv_t}\sum_{w\in \delta_H^i(v)\cap \tauv_s} \Y{u,w} - \sum_{v\in \tauv_t}\sum_{w\in \delta_G^i(u)\cap \piv_t} \Y{w,v} &= 0  \\
\Rightarrow
\sum_{w\in \tauv_s} \sum_{v\in \delta^i_H(w)\cap \tauv_t} \Y{u,w} - \sum_{w\in \delta_G^i(u)\cap \piv_t}\sum_{v\in \tauv_t} \Y{w,v} &= 0 \\
\Rightarrow
\sum_{w\in \tauv_s}|\delta^i_H(w)\cap \tauv_t| \Y{u,w} - \sum_{w\in \delta_G^i(u)\cap \piv_t}1 &= 0 \\
\Rightarrow
\sum_{w\in \tauv_s}|\delta^i_H(w)\cap \tauv_t| \Y{u,w} - |\delta_G^i(u)\cap \piv_t| &= 0 \\
\Rightarrow
\sum_{w\in \tauv_s}|\delta^i_H(w)\cap \tauv_t| \Y{u,w} - |\delta_G^i(u)\cap \piv_t|\sum_{w\in \tauv_s} \Y{u,w} &= 0 \\
\Rightarrow
\sum_{w\in \tauv_s}\left(|\delta^i_H(w)\cap \tauv_t|-|\delta_G^i(u)\cap \piv_t|\right)\Y{u,w}&=0.
\label{eqn:delta1}
\end{align}

Similarly, swapping the role of $G$ and $H$, and $\piv$ and $\tauv$, we have
the following valid equation for all $1\leq s,t\leq m$, $1\leq i\leq k$ and $v\in \tauv_s$:
\begin{align}
\sum_{u\in \piv_t}\left(\sum_{w\in \delta_H^i(v)} \Y{u,w} - \sum_{w\in \delta_G^i(u)} \Y{w,v}\right) &= 0 \\
\Rightarrow
\sum_{w\in \piv_s}\left(|\delta^i_G(w)\cap \piv_t|-|\delta_H^i(v)\cap \tauv_t|\right)\Y{w,v}&=0.
\label{eqn:delta2}
\end{align}

Using the above equations (\ref{eqn:delta1}) and (\ref{eqn:delta2}) and the inequalities 
that $\Y{u,v}\ge0$ for all $u,v\in V^k$,
we can imply that $\Y{u,v}=0$ for all $u\in \piv_s$ and
$v\in \tauv_s$ where $|\delta_G^i(u)\cap \piv_t|\ne|\delta_H^i(v)\cap \tauv_t|$
for some $1\leq t\leq m$ and $1\leq i \leq k$.
This requires an argument by induction that is analogous to the argument in the proof of Lemma \ref{lem:DeltaPoly} by replacing $\Delta$ with $\delta_G$.

Next, we prove that $\Y{u,v}=0$ for all $u\in \piv_s$ and
$v\in \tauv_s$ where $|\deltabar^i_G(u)\cap \piv_t|\ne|\deltabar^i_H(v)\cap \tau_t|$
for some $1\leq t\leq m$ and $1\leq i \leq k$.
First, using equations (\ref{eqn1:delta}), (\ref{eqn1:tuple sum}) and (\ref{eqn2:tuple sum}), we have
\[\sum_{w\in\delta^i_H(v)} \Y{u,w} -
\sum_{w\in\delta^i_G(u)}\Y{w,v}=0 
\quad \Rightarrow \quad 
\sum_{w\in\deltabar^i_H(v)} \Y{u,w} -
\sum_{w\in\deltabar^i_G(u)}\Y{w,v}=0.
\]
Thus, we can proceed as per the $\delta$ case replacing $\delta$ with $\deltabar$ as required.
\end{proof}

\begin{lemma}
\label{lem:P non-empty}
Let $\piv,\tauv\in\Piv^k$ where $(G,\piv)\equiv_{\deltav^*}(H,\tauv)$.
Then, $\TX^k_{G,H}(\piv,\tauv)\ne\emptyset$, and more specifically, $Y\in \TX^k_{G,H}(\piv,\tauv)$ where
$Y_\emptyset = 1$ and for all $u,v\in V^k$,
$\Y{u,v} = |\piv_s|^{-1}=|\tauv_s|^{-1}$ if $[u]_\piv=[v]_\tauv=s$ and $\Y{u,v}=0$ otherwise.
\end{lemma}
\begin{proof}
First, we have $(G,\piv)\equiv_{\Deltav^*}(H,\tauv)$. Thus, from Lemma \ref{lem:Q non-empty}, we have the $Y$ is well-defined and $Y\in \QX^k_{G,H}(\piv,\tauv)$.
Thus, it remains to show that $Y$ satisfies (\ref{eqn1:delta}).

Now, for all $1\leq s,t \leq m$, for all $1\leq r \leq k$ and for all $u\in V_s^w,v \in \tauv_t$, we have
\begin{align*}
\sum_{w\in\delta^i_H(v)\cap \tauv_s}\Y{u,w} - \sum_{w\in\delta^i_G(u)\cap \piv_t}\Y{w,v} &=
\sum_{w\in\delta^i_H(v)\cap \tauv_s}|\tauv_s|^{-1} - \sum_{w\in\delta^i_G(u)\cap \piv_t}|\piv_t|^{-1} \\
&= |\delta^i_H(v)\cap \tauv_s||\tauv_s|^{-1} - |\delta^i_G(u)\cap \piv_t||\piv_t|^{-1}.
\end{align*}
Next, we prove that the last expression is 0.
First, note that for all $1\leq s,t\leq m$, for all $1\leq r \leq k$ and for all $u\in \piv_s$
and $v\in \piv_t$, we must have 
\begin{align}
\sum_{u\in \piv_s}|\delta^i_G(u)\cap \piv_t|&=
\sum_{u\in \piv_t}|\delta^i_G(u)\cap \piv_s| \\
\Rightarrow 
|\piv_s||\delta^i_G(u)\cap \piv_t| &= |\piv_t||\delta^i_H(v)\cap \piv_s| \\
\Rightarrow
|\tauv_s||\delta^i_G(u)\cap \piv_t| &= |\piv_t||\delta^i_H(v)\cap \tauv_s| \\
\Rightarrow
|\delta^i_G(u)\cap \piv_t||\piv_t|^{-1} &= |\delta^i_H(v)\cap \tauv_s||\tauv_s|^{-1}
\end{align}
as required.
\end{proof}

\subsection{The $\wl$-V-C Algorithm and the $\Delta$-Polytope}
\label{sec:WL+BP}
In this section, we prove Theorem \ref{thm:SheraliAdamsWL}.
We have seen that $\Delta$ and $\Deltav$ are combinatorially equivalent to $\QX^k_G$ and $\QX^k_{G,H}$ respectively.
Moreover, we have seen that the $k$-dim $\wl$ and $\wlv$ are essentially equivalent to $(k+1)$-dim $\Delta$ and $\Deltav$ respectively.
Combining these facts, the equivalence relation $\wl$ and the polyhedron $\QX^{k+1}_G$ (actually its projection onto $\BX^k$) are combinatorially equivalent in the following sense.
\begin{corollary}
Let $G\in\Graphs$ and $k>1$. Let $\pi\in\Pi^k$. 
We have $\wl_G^*(\pi)$ is complete if and only if $\QX^{k+1}_G(\nu(\pi))=\{\IX^{k+1}\}$.
Moreover, for all $u,v\in V^k$, we have
$(G,\wl_G^*(\pi),u) \not\equiv_\wl (G,\wl_G^*(\pi),v)$ if and only if $\Y{u,v}=0$ for all $Y\in \QX^{k+1}_G(\nu(\pi))$.
\end{corollary}
\begin{proof}
Let $\pi'=\nu(\pi)$, and let $u'=u'$ and $v'=v'$.
Firstly, from Corollary \ref{cor:xi=Delta}, we have that $\wl_G^*(\pi)$ is complete if and only if $\Delta_G^*(\pi')$ is complete,
and thus, $\wl_G^*(\pi)$ is complete if and only if $\QX^{k+1}_G(\pi')=\{\IX^{k+1}\}$ since $\Delta$ is combinatorially equivalent to $\QX^{k+1}$. 
Secondly, we have $(G,\wl_G^*(\pi),u) \not\equiv_\wl (G,\wl_G^*(\pi),v)$ if and only if $(\Delta_G^*(\pi'),u') \not\equiv_{\Delta} (\Delta_G^*(\pi'),v')$ since $\wl_G^*(\pi) = \rho(\Delta_G^*(\pi'))$ from Corollary \ref{cor:xi=Delta}.
Also, $(G,\Delta_G^*(\pi'),u') \not\equiv_{\Delta} (G,\Delta_G^*(\pi'),v')$ if and only if $\Y{u',v'}=0$ for all $Y\in \QX^{k+1}_G(\pi')$ since $\Delta$ and $\QX^{k+1}_G$ are combinatorially equivalent.
The result follows then since $\pair{u,v}=\pair{u',v'}$.
\end{proof}

Specifically, we have that $\wl_G^k$ is complete if and only if $\QX^{k+1}_G=\{\IX^{k+1}\}$ or equivalently $\Q^{k+1}_G=\{\I\}$
 as required in Theorem \ref{thm:SheraliAdamsWL}.
Similarly, $\wlv$ is combinatorially equivalent to $\QX^{k+1}_{G,H}$ (actually its projection onto $\BX^k$) in the following sense.
\begin{corollary}
Let $G,H\in\Graphs$ and let $k>1$. Let $\piv,\tauv\in\Piv^k$ where $\piv\approx\tauv$ and $\nuv(\piv)\approx\nuv(\tauv)$.
Then, $\QX^k_{G,H}(\nuv(\piv),\nuv(\tauv)) \ne \emptyset$ if and only if $(G,\wlv_G^*(\piv))\equiv_\wlv (H,\wlv_H^*(\tauv))$.
Also, if $(G,\wlv_G^*(\piv))\equiv_\wlv (H,\wlv_H^*(\tauv))$, then for all  $u,v\in V^k$,
we have $(G,\wlv_G^*(\piv),u) \not\equiv_\wlv (H,\wlv_H^*(\tauv),v)$ if and only if $\Y{u,v}=0$ for all $Y\in \QX^k_{G,H}(\nuv(\piv),\nuv(\tauv))$.
\end{corollary}
\begin{proof}
Let $\piv'=\nuv(\piv)$ and $\tauv'=\nuv(\tauv)$, and let $u'=\nu(u)$ and $v'=\nu(v)$.
Firstly, from Corollary \ref{cor:xiv=Deltav}, $(G,\wlv_G^*(\piv))\equiv_\wlv (H,\wlv_H^*(\tauv))$ if and only if $(G,\Deltav_G^*(\piv'))\equiv_\Deltav (H,\Deltav_H^*(\tauv'))$, and thus, since $\Delta$ is combinatorial equivalent to $\QX^k_{G,H}$, we have $(G,\wlv_G^*(\piv))\equiv_\wlv (H,\wlv_H^*(\tauv))$ if and only if $\QX^k_{G,H}(\piv',\tauv')\ne\emptyset$.
Secondly,
we have $(G,\wlv_G^*(\piv),u) \not\equiv_\wlv (H,\wlv_H^*(\tauv),v)$ if and only if $(G,\Deltav_G^*(\piv'),u') \not\equiv_{\Deltav} (H,\Deltav_H^*(\tauv'),v')$ since $(\wlv_G^*(\piv),\wlv_H^*(\tauv)) \simeq (\rho(\Deltav_G^*(\piv')), \rho(\Deltav_H^*(\tauv')))$ from Corollary \ref{cor:xiv=Deltav}.
Then, since $\Delta$ is combinatorial equivalent to $\QX^k_{G,H}$, we have that $(G,\Deltav_G^*(\piv'),u') \not\equiv_{\Deltav} (H,\Deltav_H^*(\tauv'),v')$ if and only if $\Y{u',v'}=0$ for all $Y\in \QX^k_{G,H}(\piv',\tauv')$. The result follows since $\pair{u,v}=\pair{u',v'}$.
\end{proof}
Specifically, we have $(G,\wlv_G^k)\equiv_\wlv (H,\wlv_H^k)$ if and only if $\QX^{k+1}_{G,H}\ne\emptyset$
or equivalently $\Q^{k+1}_{G,H}\ne\emptyset$ as required for Theorem \ref{thm:SheraliAdamsWL}.

\section{Acknowledgements}
We thank Mohamed Omar for his great help in writing this paper and many fruitful discussions about the results within.

\section{Appendix}
In this appendix, we prove technical results that are necessary to prove the main results of the paper.

\begin{lemma}
\label{lem:tech}
Let $G,H\in\Graphs$. Let $\piv,\tauv\in\Piv^k$ such that $(G,\piv)\equiv_{\Deltav^*}(H,\tauv)$, and let $u,v\in V^k$ such that $(\piv,u) \equiv_\Deltav (\tauv,v)$.
Then, we have $[\phi_i(u,u_j)]_\piv = [\phi_i(v,v_j)]_\tauv$ for all $1\leq i,j \leq k$.
\end{lemma}
\begin{proof}
We prove the contrapositive. 
Let $u'= \phi_i(u,u_j)\in\piv_{s'}$ and $v' = \phi_i(v,v_j)$ such that $[\phi_i(u,u_j)]_\piv \ne [\phi_i(v,v_j)]_\tauv$.
Then, $|\Delta^i(u)\cap \piv_{s'}|=|\{u'\}|=1$ since $\piv_{s'}$ only contains tuples of the same combinatorial type as $u'$
meaning that $w_i=w_j$ for all $w\in \piv_{s'}$.
But, $|\Delta^i(v)\cap \tauv_{s'}|=0$ since the only tuple in $\Delta^i(v)$ with the same combinatorial type as $u'$ is $v'$ and $v'\not\in\tauv_{s'}$ by assumption.
Thus, $(\piv,u)\not\equiv_\Deltav(\tauv,v)$.
\end{proof}

\begin{corollary}
\label{cor:tech}
Let $G\in\Graphs$. Let $\pi\in\Pi^k$ where $\Delta_G(\pi)=\pi$, and let $u,v\in V^k$ such that $(\pi,u) \equiv_\Delta (\pi,v)$.
Then, we have $\phi_i(u,u_j) \equiv_{\pi} \phi_i(v,v_j)$ for all $1\leq i,j \leq k$.
\end{corollary}

\begin{lemma}\label{lem:idv}
Let $G\in\Graphs$. Let $\piv\in\Piv^k$.
Then, $\rhov(\nuv(\piv)) \preceq \piv$.
If $\Deltav_G(\piv)=\piv$, then $\rhov(\nuv(\piv))=\piv$.
\end{lemma}
\begin{proof}
Let $u,v\in V^k$, and let $u'=\nu(u)$ and $v'=\nu(v)$.
Let $\piv' = \nuv(\piv)$ and $\piv'' = \rhov(\piv')$.
First, $u \gneqq_\piv v$ implies $u'\gneqq_{\piv'} v'$ by definition of $\nuv$.
Then, $u \gneqq_{\piv''} v$ by definition of $\rhov$. 
Hence, $u \gneqq_\piv v$ implies $u \gneqq_{\piv''} v$.
The contrapositive is thus $u \leqq_{\piv''} v$ implies $u \leqq_\piv v$, 
and thus, $\rhov(\nuv(\piv)) \preceq \piv$.

Second, assume $\Deltav_G(\piv)=\piv$.
By definition of $\rhov$, we have $u \leqq_{\piv''} v$ if and only if $u' \leqq_{\piv'} v'$.
Then, by definition of $\nuv$, we have $u' \leqq_{\piv'} v'$ if and only if $u \lneqq_\piv v$ or $u \equiv_\piv v$
and $\wlv_\piv(u,u_k) \leq_{lex} \wlv_\piv(v,v_k)$.
But, if $u \equiv_\piv v$, then $(G,\piv,u)\equiv_{\Deltav}(H,\tauv,v)$, and thus, $\phi_i(u, u_k) \equiv_\piv \phi_i(v,v_k)$ for all $1\leq i \leq k$ from Lemma \ref{lem:tech}.
So, $u' \leqq_{\piv'} v'$ if and only if $u \leqq_\piv v$.
Thus, $u \leqq_{\piv''} v$ if and only if $u\leqq_\piv v$, and $\piv'' = \piv$ as required.
\end{proof}

\begin{corollary}\label{cor:id}
Let $G\in\Graphs$. Let $\pi\in\Pi^k$.
Then, $\rho(\nu(\pi))\leq\pi$.
If $\Delta_G(\pi)=\pi$, then $\rho(\nu(\pi))=\pi$.
\end{corollary}

\begin{lemma} \label{lem:nuv props}
Let $k>1$. Let $G,H\in\Graphs$. Let $\piv,\tauv\in\Piv^k$ where $(G,\piv) \equiv_{\wlv} (H,\tauv)$. Then, $\nuv(\piv)\approx\nuv(\tauv)$, and
for all $u,v\in V^{k+1}$, we have $[u]_{\nuv(\piv)} = [v]_{\nuv(\tauv)}$ if and only if
$[\rho(u)]_\piv=[\rho(v)]_\tauv$ and $\wlv_\piv(\rho(u),u_{k+1}) = \wlv_\tauv(\rho(v),v_{k+1})$.
Also, for all $\piv'',\tauv''\in \Piv^k$ where $(\piv,\tauv)\leq(\piv'',\tauv'')$,
we have $(\nuv(\piv),\nuv(\tauv))\leq(\nuv(\piv''),\nuv(\tauv''))$.
Furthermore, we have $(G,\piv) \equiv_{\Deltav^*} (H,\tauv)$ implies $\nuv(\piv) \equiv_{\Symv^*} \nuv(\tauv)$ and $(G,\nuv(\piv)) \equiv_{\CPv^*} (H,\nuv(\tauv))$.
\end{lemma}
\begin{proof}
Let $\piv'=\nuv(\piv)$ and $\tauv'=\nuv(\tauv)$. 
Since $(G,\piv) \equiv_\wlv (H,\tauv)$, 
there exists a tuple bijection $\gamma':V^k\rightarrow V^k$ such that $(G,\piv',u')\equiv_\wlv(H,\tauv',\gamma'(u'))$ for all $u'\in V^k$.
Then, for all $u'\in V^k$, since $(G,\piv',u')\equiv_\wlv (H,\tauv',\gamma(u'))$, there must exist a bijection $\psi:V\rightarrow V$ such that 
$\wlv_{\piv'}(u',w) = \wlv_{\tauv'}(\gamma'(u'),\psi(w))$ for all $w\in V$.
We now define the bijection $\gamma:V^{k+1}\rightarrow V^{k+1}$ where for all $u\in V^{k+1}$, we have $\gamma(u) = (\gamma(u')_1,...,\gamma(u')_k,\psi(u_{k+1}))$
where $u'=\rho(u)$.
Thus, for all $u\in V^{k+1}$, we have 
$[\rho(u)]_\piv=[\rho(\gamma(u))]_\tauv$ and $\wlv_\piv(\rho(u),u_{k+1}) = \wlv_\tauv(\rho(\gamma(u)),\gamma(u)_{k+1})$.
Then, by construction of $\nuv$, for all $u,u''\in V^{k+1}$, we have $u\leqq_{\piv'} u''$ if and only
if $\gamma(u) \leqq_{\tauv'} \gamma(u'')$ and thus $[u]_{\piv'} = [\gamma(u)]_{\tauv'}$;
therefore, $\piv'\approx\tauv'$ and $[u]_{\piv'} = [v]_{\tauv'}$ if and only if
$[\rho(u)]_\piv=[\rho(v)]_\tauv$ and $\wlv_\piv(\rho(u),u_{k+1}) = \wlv_\tauv(\rho(v),v_{k+1})$.

Let $\piv'',\tauv''\in\Piv^k$ such that $(\piv,\tauv)\leq(\piv'',\tauv'')$. Let $u,v\in V^{k+1}$ such that $[u]_{\piv'} = [v]_{\tauv'}$.
From above, we have 
$[\rho(u)]_\piv=[\rho(v)]_\tauv$ and $\wlv_\piv(\rho(u),u_{k+1}) = \wlv_\tauv(\rho(v),v_{k+1})$
implying $[\rho(u)]_{\piv''}=[\rho(v)]_{\tauv''}$ and $\wlv_{\piv''}(\rho(u),u_{k+1}) = \wlv_{\tauv''}(\rho(v),v_{k+1})$
since $(\piv,\tauv)\leq(\piv'',\tauv'')$, and thus, since $(G,\piv'')\equiv_\wlv (H,\tauv'')$ from Lemma \ref{lem:equiv props}, 
we have $[u]_{\nuv(\piv'')} = [v]_{\nuv(\tauv'')}$ from above implying $(\piv',\tauv')\leq(\nuv(\piv''),\nuv(\tauv''))$ as required.

Assume $(G,\piv) \equiv_{\Deltav^*} (H,\tauv)$, and so, $(G,\piv)\equiv_{\CPv^*}(H,\tauv)$.
Let $u',v'\in V^{k+1}$ where $[u']_{\piv'} = [v']_{\tauv'}$ and let $u=\rho(u')$ and $v=\rho(v')$.
We must show that $(G,\piv',u')\equiv_{\CPv} (H,\tauv',v')$, and it follows that
$(G,\piv')\equiv_{\CPv^*}(H,\tauv')$ since $\piv'\approx\tauv'$.
From above, we have $[u]_\piv=[v]_\tauv$ and $\wlv_\piv(u,u'_{k+1}) = \wlv_\tauv(v,v'_{k+1})$,
and so by assumption, $(G,u)\equiv_{\CPv} (H,v)$ and $(G,\phi_i(u,u'_{k+1}))\equiv_{\CPv} (H,\phi_i(v,v'_{k+1}))$, 
which implies that $(G,u') \equiv_{\CPv} (H,v')$, and thus, $(G,\piv', u') \equiv_\CPv(H,\tauv', v')$ as required.

Assume $(G,\piv) \equiv_{\Deltav^*} (H,\tauv)$, and so, $\piv\equiv_{\Symv^*}\tauv$.
Let $u',v'\in V^{k+1}$ and let $u=\rho(u')$ and $v=\rho(v')$.
Then, it follows by construction that
$[u]_\piv =[v]_\tauv$ and $[\phi_i(u,u'_{k+1})]_\piv = [\phi_i(v,v'_{k+1})]_\tauv$ for all $1\leq i\leq k$
if and only if $[\rho(\sigma(u'))]_\piv = [\rho(\sigma(v'))]_\tauv$ and 
$[\phi_i(\rho(\sigma(u')),\sigma(u')_{k+1})]_\piv = [\phi_i(\rho(\sigma(v')),\sigma(v')_{k+1})]_\tauv$ 
for all $\sigma\in\Sym_{k+1}$.
Thus, $[u']_{\piv'} =[v']_{\tauv'}$ if and only if $[\sigma(u')]_{\piv'} = [\sigma(v')]_{\tauv'}$  and for all $\sigma\in\Sym_{k+1}$.
Thus, $\piv'\equiv_{\Symv^*}\tauv'$ since $\piv'\approx\tauv'$.
\end{proof}

\begin{corollary}\label{lem:nu props}
Let $k>1$ and let $G\in\Graphs$, and let $\pi\in\Pi^k$.
For all $\tau\in\Pi^k$ where $\pi\leq \tau$, we have $\nu(\pi)\leq\nu(\tau)$.
Also, $\wl_G(\pi)=\pi$ implies $\CP_G(\nu(\pi)) = \nu(\pi)$ and $\Sym(\nu(\pi))=\nu(\pi)$.
\end{corollary}

\begin{lemma} \label{lem:rhov props}
Let $k>1$. Let $G,H\in\Graphs$. Let $\piv,\tauv\in\Piv^k$ where $(G,\piv) \equiv_{\Deltav} (H,\tauv)$.
Then, $\rhov(\piv)\approx\rhov(\tauv)$ and for all $u',v'\in V^{k-1}$, 
we have $[u']_{\rhov(\piv)} = [v']_{\rhov(\tauv)}$ if and only if $[\nu(u')]_\piv=[\nu(v')]_\tauv$.
Moreover, for all $u,v\in V^k$, we have $(G,\piv,u)\equiv_\Deltav(H,\tauv,v)$ implies $[\rho(u)]_{\rhov(\piv)} = [\rho(v)]_{\rhov(\tauv)}$.
Also, for all $\piv'',\tauv''\in \Piv^k$ where $(\piv,\tauv)\leq(\piv'',\tauv'')$,
we have $(\rhov(\piv),\rhov(\tauv))\leq(\rhov(\piv''),\rhov(\tauv''))$.
Furthermore, $(G,\piv) \equiv_{\Deltav^*} (H,\tauv)$ implies $\rhov(\piv) \equiv_{\Symv^*} \rhov(\tauv)$ and $(G,\rhov(\piv)) \equiv_{\CPv^*} (H,\rhov(\tauv))$.
\end{lemma}
\begin{proof}
Let $\piv'=\rhov(\piv)$ and $\tauv'=\rhov(\tauv)$. 
Since $(G,\piv) \equiv_\Deltav(H,\tauv)$,
there exists a tuple bijection $\gamma:V^k\rightarrow V^k$ such that $(G,\piv,u)\equiv_\Deltav(H,\tauv,\gamma(u))$ for all $u\in V^k$.
Now, we define the bijection $\gamma':V^{k-1}\rightarrow V^{k-1}$ where for all $u'\in V^{k-1}$, we have $\gamma'(u') = \rho(\gamma(\nu(u')))$.
Since $(G,\piv,u)\equiv_\Deltav(H,\tauv,\gamma(u))$ and thus $(G,\piv,u)\equiv_\CPv(H,\tauv,\gamma(u))$,
we have %$u_{k-1}=u_k$ if and only if $\gamma(u)_{k-1}=\gamma(u)_k$, that is, 
$\nu(\rho(u))=u$ if and only if $\nu(\rho(\gamma(u)))=\gamma(u)$.
Thus,  for all $u'\in V^{k-1}$, we have $[\nu(u')]_\piv=[\nu(\gamma'(u'))]_\tauv$ 
since $\nu(\rho(\nu(u')))=\nu(u')$ and thus $\nu(\gamma'(u'))=\nu(\rho(\gamma(\nu(u'))))=\gamma(\nu(u'))$.
Then, by construction of $\rhov$, for all $u',u''\in V^{k-1}$, we have $u' \leqq_{\piv'} u''$ if and only
if $\gamma(u') \leqq_{\tauv'} \gamma(u'')$ and thus $[u']_{\piv'} = [\gamma(u')]_{\tauv'}$;
therefore, $\piv'\approx\tauv'$ and $[u']_{\rhov(\piv)} = [v']_{\rhov(\tauv)}$ if and only if $[\nu(u')]_\piv=[\nu(v')]_\tauv$ for all $u',v'\in V^{k-1}$.

Let $u,v\in V^k$ such that $(G,\piv,u)\equiv_\Deltav(H,\tauv,v)$. 
Then, by Lemma \ref{lem:tech}, we have $[\phi_k(u,u_{k-1})]_{\piv} = [\phi_k(v,v_{k-1})]_{\tauv}$
implying that $[\rho(u)]_{\piv'} = [\rho(v)]_{\tauv'}$ from above since $\nu(\rho(u)) = \phi_k(u,u_{k-1})$
and $\nu(\rho(v)) = \phi_k(v,v_{k-1})$ as required.

Let $\piv'',\tauv''\in\Piv^k$ such that $(\piv,\tauv)\leq(\piv'',\tauv'')$. 
Let $u,v\in V^{k-1}$ such that $[u]_{\rhov(\piv)} = [v]_{\rhov(\tauv)}$.
From above, we have 
$[\nu(u)]_\piv=[\nu(v)]_\tauv$ implying $[\nu(u)]_{\piv''}=[\nu(v)]_{\tauv''}$ 
since $(\piv,\tauv)\leq(\piv'',\tauv'')$, and thus, from above,
$[u]_{\rhov(\piv'')}=[v]_{\rhov(\tauv'')}$ implying 
$(\rhov(\piv),\rhov(\tauv))\leq(\rhov(\piv''),\rhov(\tauv''))$ as required.

Assume $(G,\piv)\equiv_{\Deltav^*}(H,\tauv)$, and so,  $(G,\piv)\equiv_{\CPv^*}(H,\tauv)$.
Let $u',v'\in V^{k-1}$ where $[u']_{\piv'}=[v']_{\tauv'}$, and let $u=\nu(u')$ and $v=\nu(v')$.
We must show that $(G,\piv',u')\equiv_{\CPv}(H,\tauv',v')$ and it follows that 
$(G,\piv') \equiv_{\CPv^*} (H,\tauv')$ since $\piv'\approx\tauv'$.
First, $[u']_{\piv'}=[v']_{\tauv'}$ implies $[u]_{\piv} = [v]_{\tauv}$ from above, and thus, $(G,u)\equiv_{\CPv} (H,v)$ by assumption,
which implies that $(G,u') \equiv_{\CPv} (H,v')$.
Thus, $(G,\piv',u')\equiv_{\CPv} (H,\tauv',v')$ as required.

Assume $(G,\piv)\equiv_{\Deltav^*}(H,\tauv)$, and so, $(G,\piv)\equiv_{\Symv^*}(H,\tauv)$. 
Let $u',v'\in V^{k-1}$ where $[u']_{\piv'}=[v']_{\tauv'}$ and let $u=\nu(u')$ and $v=\nu(v')$.
Let $\sigma'\in\Sym_{k-1}$. We must show that $[\sigma'(u')]_{\piv'} = [\sigma'(v')]_{\tauv'}$.
Since  $[u']_{\piv'}=[v']_{\tauv'}$, we have $[u]_{\piv} = [v]_{\tauv}$.
Let $\sigma\in\Sym_k$ such that $\rho(\sigma(u))=\sigma'(u')$ and $\rho(\sigma(v))=\sigma'(v')$, which clearly exists.
Then, $[\sigma(u)]_{\piv} = [\sigma(v)]_{\tauv}$ by assumption.
Thus, $[\sigma'(u')]_{\piv'} = [\sigma'(v')]_{\tauv'}$ from above.
Thus, $\piv'\equiv_{\Symv^*}\tauv'$ since $\piv'\approx\tauv'$.
\end{proof}

\begin{corollary} \label{lem:rho props}
Let $k>1$ and let $G\in\Graphs$, and let $\pi\in\Pi^k$.
For all $u,v\in V^k$, we have $(G,\pi,u)\equiv_\Delta(G,\pi,v)$ implies $\rho(u)\equiv_{\rho(\pi)} \rho(v)$.
For all $\tau\in \Piv^k$ where $\pi\leq\tau$, we have $\rho(\pi)\leq\rho(\tau)$.
Also, $\Delta_G(\pi)=\pi$ implies $\CP_G(\rho(\pi)) = \rho(\pi)$ and $\Sym(\rho(\pi))=\rho(\pi)$.
\end{corollary}

\begin{lemma}\label{lem:mapping}
Let $\pi=\{\pi_1,...,\pi_m\}\in\Pi^k$ where $\Delta_G(\pi)=\pi$.
Let $u \in \pi_s$ and $u'\in \pi_{s'}$ where $u'=\phi_i(u, u_j)$ for some $1\leq i,j\leq k$ where $j\ne i$.
Then, we have $\phi_{i,j}(\pi_s)=\{\phi_i(v, v_j):v\in \pi_s\} = \pi_{s'}$ and $|\pi_s| =
|\Delta^i(u) \cap \pi_s||\pi_{s'}|$.
\end{lemma}
\begin{proof}
First, we show that $\phi_{i,j}(\pi_s) \subseteq \pi_{s'}$.
Let $v\in \pi_s$, and let $v'=\phi_i(v,v_j)$. 
From Corollary \ref{cor:tech}, we have $v'\equiv_{\pi} u'$, and thus, $v'\in\pi_{s'}$.
So,  $\phi_{i,j}(\pi_s) \subseteq \pi_{s'}$.

Next, we show that $\phi_{i,j}(\pi_s) \supseteq \pi_{s'}$.
Let $v'\in \pi_{s'}$. 
Assume that for all $v\in V^k$ where $v'=\phi_i(v,v_j)$, we have $v\not\in \pi_s$.
Then, $|\Delta^i(v')\cap \pi_s| = 0$.
But, $|\Delta^i(u')\cap \pi_s| > 0$ since $u\in \Delta^i(u')\cap \pi_s$.
Thus, $u'\not\equiv_{\pi} v'$, a contradiction.
Hence, there exists $v\in \pi_s$ such that $v'=\phi_i(v,v_j)$,
and thus, $\phi_{i,j}(\pi_s) \supseteq \pi_{s'}$. Therefore,
$\phi_{i,j}(\pi_s) = \pi_{s'}$.

Now, since $\phi_{i,j}(\pi_s) = \pi_{s'}$, 
we have $\bigcup_{v'\in \pi_{s'}} \Delta^i(v')\cap \pi_s = \pi_s$, 
and for all $w,w'\in \pi_{s'}$, we have $\Delta^i(w) \cap \Delta^i(w')=\emptyset$ by construction, and $|\Delta^i(w)\cap \pi_s|= |\Delta^i(w')\cap \pi_s|$ since $\Delta_G(\pi)=\pi$.
Thus, $|\pi_s| =|\Delta^i(u') \cap \pi_s||\pi_{s'}|=|\Delta^i(u)\cap \pi_{s'}||\pi_{s'}|$
as required.
\end{proof}

\begin{definition}
Let $u\in V^k$. We define $\pair{u}= \{u_1,...,u_k\}$.
\end{definition}

\begin{lemma}\label{lem:tuple invariance}
Let $\pi=\{\pi_1,...,\pi_m\}\in\Pi^k$  where $\Delta_G(\pi)=\pi$, and let $u\in \pi_s,v\in \pi_{t}$.
If $\pair{u} = \pair{v}$, then $|\pi_s|=|\pi_{t}|$.
\end{lemma}
\begin{proof}
Assume $\pair{u}=\pair{v}$, but $u\neq v$.
Since $\Sym(\pi)=\pi$, we may assume that $\pair{u}=\{u_1,...,u_r\}$
and $\pair{v}=\{v_1,...,v_r\}$ where $u_i=v_i$ for all $1\leq i \leq r$ where
$r\geq|\pair{u}|=|\pair{v}|$.
Now, there exists $r < i \leq k$ such that $u_i\neq v_i$.
Since $\pair{u}=\pair{v}$, we must have $u_i=u_j=v_j$ for some $1\leq j \leq r$.
Let $u'=\phi_i(u, u_j)\in \pi_{s'}$, so $u'$ differs from $v$ in one less component than $u$, and $\pair{u}=\pair{u'}$.
By Lemma \ref{lem:mapping}, we have $|\pi_s|=|\pi_{s'}|$.
Repeating this process until we arrive at $v$ proves the result.
\end{proof}

\begin{corollary}\label{cor:set sizes}
Let $G,H\in\Graphs$. Let $\piv=(\piv_1,...,\piv_m),\tauv=(\tauv_1,...,\tauv_m)\in\Piv^k$ where $(G,\piv)\equiv_{\Deltav^*}(H,\tauv)$, and let $u\in \piv_s,v\in \tauv_s$
and $u'\in \piv_t,v'\in \tauv_t$.
If $\pair{u,v} = \pair{u',v'}$, then $|\piv_s|=|\piv_t|=|\tauv_s|=|\tauv_t|$.
\end{corollary}
\begin{proof}
The fact $\pair{u,v}=\pair{u',v'}$ implies $\pair{u}=\pair{u'}$,
so applying Lemma \ref{lem:tuple invariance}, the result follows.
\end{proof}

\begin{lemma}\label{lem:supset}
Let $P\subseteq \BX^k$, and let $Y\in P$. Let $I \subseteq V^2$.
If $Y_I = 0$, then $Y_{I'}= 0$ for all $I'\supseteq I$.
\end{lemma}
\begin{proof}
We show by induction that $Y_{I'}=0$ for all $I'\supseteq I$ where $|I'\setminus I| \leq r$
for all $r$.
It is trivially true for $r=0$, so assume true for $r$.
Let $I'\supseteq I$ where $|I'\setminus I| = r+1$.
Then, there exists $I''\supseteq I$ where $|I''\setminus I| = r$
and $I' = I''\cup\{u,v\}$ for some $u,v\in V$.
By assumption $Y_{I''} = 0$. Then, equations (\ref{e2}) imply that $Y_{I'} = 0$ as required.
\end{proof}

\bibliography{references}{}
\bibliographystyle{amsplain}

\end{document}